\documentclass{amsart}

 \usepackage{amsmath}
\usepackage{amsfonts}
\usepackage{bbm}
\usepackage[dvips]{graphicx}
\usepackage{color}
\usepackage{mathtools}
\usepackage{float} 
\usepackage{url}
\usepackage{mhsetup}
\usepackage{fullpage}
\usepackage{epsfig} 

\def\Prox{{\hbox{\rm Prox}}}

\def\Prox{{\hbox{\rm Prox}}}
\def\argmin{{\mathop{\hbox{\rm argmin}\,}}}

\newtheorem{thm}{\bf{Theorem}}[section]
\newtheorem{lemma}[thm]{\bf{Lemma}}

\newtheorem{df}[thm]{\bf{Definition}}
\newtheorem{cor}[thm]{\bf{Corollary}}

\newtheorem{prop}[thm]{\bf{Proposition}}

\newtheorem{ex}[thm]{\bf{Example}}
\newtheorem{rem}[thm]{\bf{Remark}}

\begin{document}

\title[]{Inexact Stochastic Mirror Descent for two-stage nonlinear stochastic programs}

\maketitle 

\begin{center}
Vincent Guigues\\
School of Applied Mathematics, FGV\\
Praia de Botafogo, Rio de Janeiro, Brazil\\ 
{\tt vguigues@fgv.br}
\end{center}

\date{}

\begin{abstract} We introduce an inexact variant of Stochastic Mirror Descent (SMD), called Inexact Stochastic Mirror Descent (ISMD), to solve
nonlinear two-stage stochastic programs where the second stage problem 
has linear and nonlinear coupling constraints and a 
nonlinear objective function which depends on both first and second
stage decisions. Given a candidate first stage solution and a realization of the second stage random vector, each iteration of ISMD combines a stochastic subgradient descent using a prox-mapping
with the computation of approximate (instead of exact for SMD) primal and dual second stage solutions.
We provide two convergence analysis of ISMD, under two sets of assumptions.
The first convergence analysis is based on the formulas for inexact cuts of value functions of
convex optimization problems shown recently in \cite{guigues2016isddp}.
The second convergence analysis provides a convergence rate (the same as SMD) and relies on
new formulas that we derive for inexact cuts of value functions of convex optimization problems 
assuming that the dual function of the second stage problem for all fixed first stage 
solution and realization of the second stage random vector, is strongly concave. 
We show that this assumption of strong concavity is satisfied for some classes of problems and present the
results of numerical experiments on two simple two-stage problems which show that 
solving approximately the second stage problem for the first iterations of ISMD can help us obtain a
good approximate first stage solution quicker than with SMD.\\
\end{abstract}



\section{Introduction}

We are interested in inexact solution methods for two-stage nonlinear stochastic programs
of form 
\begin{equation}\label{defpb}
\left\{
\begin{array}{l}
\min \; f(x_1):=f_1( x_1 ) + \mathcal{Q}(x_1) \\
x_1 \in X_1
\end{array}
\right.
\end{equation}
with $X_1 \subset \mathbb{R}^n$ a convex, nonempty, and compact set, and $\mathcal{Q}(x_1)=\mathbb{E}_{\xi_2}[ \mathfrak{Q}( x_1 , \xi_2 ) ]$ where
$\mathbb{E}$ is the expectation operator, $\xi_2$ is a random vector
with probability distribution $P$ on  $\Xi \subset \mathbb{R}^k$, and
\begin{equation}\label{pbsecondstage}
\mathfrak{Q}( x_1 , \xi_2 ) = 
\left\{
\begin{array}{l}
\min_{x_2} \; f_2( x_2 , x_1 , \xi_2) \\
x_2 \in X_2(x_1 , \xi_2 ):=\{x_2 \in \mathcal{X}_2 : \;A x_2 + B x_1 = b, \;g(x_2, x_1 , \xi_2) \leq 0\}.
\end{array}
\right.
\end{equation}
In the problem above vector $\xi_2$ contains in particular the random elements in matrices
$A, B$, and vector $b$.
Problem \eqref{defpb} is the first stage problem while
problem \eqref{pbsecondstage} is the second stage problem which has abstract constraints ($x_2 \in \mathcal{X}_2$),
and linear ($A x_2 + B x_1 = b$) and nonlinear ($g(x_2, x_1 , \xi_2) \leq 0$) constraints both of which couple first stage
decision $x_1$ and second stage decision $x_2$. Our solution methods are suited for the following framework:
\begin{itemize}
\item[a)] first stage problem \eqref{defpb} is convex;
\item[b)] second stage problem \eqref{pbsecondstage} is convex, i.e., 
$\mathcal{X}_2$ is convex and for every $\xi_2 \in \Xi$ functions 
$f_2(\cdot,\cdot,\xi_2)$ and
$g(\cdot,\cdot,\xi_2)$ are convex;
\item[c)] for every realization ${\tilde \xi}_2$ of $\xi_2$, the primal second stage problem
obtained replacing $\xi_2$ by ${\tilde \xi}_2$ in \eqref{pbsecondstage}
with optimal value $\mathfrak{Q}( x_1 , {\tilde \xi}_2 )$
and its dual (obtained dualizing coupling constraints) are solved approximately.
\end{itemize}
There is a large literature on solution methods for two-stage risk-neutral stochastic programs.
Essentially, these methods can be cast in two categories: (A) decomposition methods based on sampling 
and cutting plane approximations of $\mathcal{Q}$ (which date back to \cite{danglynn},\cite{infanger}) 
and their variants with regularization such as \cite{ruz} and (B) Robust Stochastic Approximation \cite{polyakjud92} and its variants
such as stochastic Primal-Dual subgradient methods \cite{nestioud2010},
Stochastic Mirror Descent (SMD) \cite{nemjudlannem09}, \cite{nemlansh09}, or Multistep Stochastic Mirror Descent (MSMD) \cite{guiguesmstep17}. 
These methods have been extended to solve multistage problems, for instance Stochastic Dual Dynamic 
Programming \cite{pereira}, belonging to class (A), and recently Dynamic Stochastic Approximation \cite{lan2017}, belonging to class (B).

However, for all these methods, it is assumed that second stage problems 
 are solved exactly. This latter assumption is not satisfied when the second stage problem is nonlinear since in this setting
only approximate solutions are available. On top of that, for the first iterations, we still have crude approximations
of the first stage solution and it may be useful to solve inexactly, with less accuracy, the second stage problem for these
iterations and to increase the accuracy of the second stage solutions computed when the algorithm progresses in order to decrease
the overall computational bulk.

Therefore the objective of this paper is to fill a gap considering 
the situation when second stage problems are nonlinear and solved approximately (both primal and dual, see Assumption c) above).
More precisely, to account for Assumption (c), as an extension of the methods from 
class (B) we derive an Inexact Stochastic Mirror Descent (ISMD) algorithm, designed to solve
problems of form \eqref{defpb}. This inexact solution  method is based on an inexact black box for the objective in \eqref{defpb}.
To this end, we compute inexact cuts (affine lower bounding functions) for value function 
$\mathfrak{Q}(\cdot,\xi_2)$ in  \eqref{pbsecondstage}.
For this analysis, we first need formulas for exact cuts (cuts based on exact primal and dual solutions).
We had shown such formulas in \cite[Lemma 2.1]{guiguessiopt2016} using convex analysis tools, in particular
standard calculus on normal and tangeant cones. We derive in Proposition \ref{dervaluefunction} a proof for these formulas
based purely on duality. This is an adaptation of the proof of the formulas we gave in \cite[Proposition 2.7]{guigues2016isddp} for inexact cuts,
considering exact solutions instead of inexact solutions.
To our knowledge, the computation of inexact cuts for value functions has only been discussed in \cite{guigues2016isddp} so far (see Proposition \ref{varprop1}).
We propose in Section \ref{sec:computing}  new formulas for computing inexact cuts based in particular on the strong concavity of the dual function.
 In Section \ref{strongconvdfunction}, we provide, for several classes of problems, conditions ensuring that the dual function of an optimization problem is strongly concave and
give formulas for computing the corresponding constant of strong concavity when possible.
It turns out that our results improve Theorem 10 in \cite{YuNeely2015} (the only reference we are aware of on the strong concavity of the dual function) which proves the strong concavity of the dual function under
stronger assumptions. The tools developped in Sections \ref{strongconvdfunction} and \ref{sec:computing}  allow us to build the inexact black boxes necessary
for the  Inexact Stochastic Mirror Descent (ISMD) algorithm and its convergence analysis presented 
in Section \ref{sec:ismd}. Finally, in Section \ref{sec:simualgorithms}
we report the results of
numerical tests comparing the performance of SMD and ISMD on two simple two-stage nonlinear stochastic programs.\\
\par Throughout the paper, we use the following notation:
\begin{itemize}
\item The domain $\mbox{dom}(f)$ of a function $f:X \rightarrow {\bar{\mathbb{R}}}$ is the set of points in $X$ such that
$f$ is finite: $\mbox{dom}(f)=\{x \in X :  -\infty< f(x) < +\infty\}$.
\item The largest (resp. smallest) eigenvalue of a matrix $Q$ having real-valued eigenvalues is denoted by $\lambda_{\max}(Q)$ (resp. $\lambda_{\min}(Q)$).
\item The $\|\cdot\|_2$ of a matrix $A$ is given by $\|A\|_2 = \max_{x \neq 0} \frac{\|Ax\|_2}{\|x\|_2}$.
\item Diag($x_1,x_2,\ldots,x_n$) is the $n\times n$ diagonal matrix whose entry $(i,i)$ is $x_i$. 
\item For a linear application $\mathcal{A}$, Ker($\mathcal{A}$) is its kernel and Im($\mathcal{A}$) its image.
\item $\langle \cdot, \cdot, \rangle $ is the usual scalar product in $\mathbb{R}^n$: $\langle x , y \rangle = \sum_{i=1}^n x_i y_i$
which induces the norm $\|x\|_2$~$=\sqrt{\sum_{i=1}^n x_i^2}$.
\item Let $f:\mathbb{R}^n \rightarrow \bar{\mathbb{R}}$ be an extended real-valued function.
The Fenchel conjugate $f^*$ of $f$ is the function given by 
$f^*( x^* ) = \sup_{x \in \mathbb{R}^n} \langle x^* , x \rangle - f( x )$. 
\item For functions $f:X \rightarrow Y$ and 
$g:Y \rightarrow Z$, the function $g \circ f: X \rightarrow Z$
is the composition of functions $g$ and 
$f$ given by
$(g \circ f)(x)=g(f(x))$
for every $x \in X$.
\end{itemize}

\section{On the strong concavity of the dual function of an optimization problem} \label{strongconvdfunction}

The study of the strong concavity of the dual function  of an optimization problem on some set
has  applications in numerical optimization.
For instance, the strong concavity of the dual function and the knowledge of the associated constant of
strong concavity are used by the Drift-Plus-Penalty algorithm in \cite{YuNeely2015} 
and by the (convergence proof of) Inexact SMD algorithm presented in Section \ref{sec:ismd} when inexact cuts are computed using
Proposition \ref{defcutctk}.

The only paper we are aware of providing
conditions ensuring this strong concavity property is \cite{YuNeely2015}.
In this section, we prove similar results under weaker assumptions and study
an additional class of problems (quadratic with a quadratic constraint, see Proposition \ref{strongconcquad}).

\subsection{Preliminaries}

In what follows, $X \subset \mathbb{R}^n$ is a nonempty convex set.
\begin{df}[Strongly convex functions]\label{defsconv}  Function $f:X \rightarrow \mathbb{R}\cup \{+\infty\}$ is strongly convex 
with constant of strong convexity $\alpha > 0$  with respect to norm $\|\cdot\|$ if
for every $x, y \in \mbox{dom(f)}$ we have 
$$
f( t x  + (1-t)y ) \leq t f(x) + (1-t)f(y)  - \frac{\alpha t(1-t)}{2} \|y-x\|^2, 
$$
for all $0 \leq t \leq 1$.
\end{df}
We can deduce the following well known characterizations of strongly convex functions $f:\mathbb{R}^n \rightarrow \mathbb{R} \cup \{+\infty\}$ (see for instance \cite{hhlem}):
\begin{prop}\label{charc1scfunc} (i) Function $f: X \rightarrow \mathbb{R}\cup \{+\infty\}$
is strongly convex with constant of strong convexity $\alpha > 0$ 
with respect to norm $\|\cdot\|$ if and only if for every $x, y \in \mbox{dom(f)}$ we have
$$
f(y) \geq f(x) + s^T (y-x)   + \frac{\alpha}{2}\|y-x\|^2, \;\forall s \in \partial f(x).
$$
\par (ii) Function $f:X \rightarrow \mathbb{R}\cup \{+\infty\}$ is strongly convex with constant of strong convexity $\alpha > 0$ 
with respect to norm $\|\cdot\|$
if and only if for every $x, y \in \mbox{dom(f)}$ we have
$$
f(y) \geq f(x) + f'(x;y-x)  + \frac{\alpha}{2}\|y-x\|^2,
$$
where $f'(x;y-x)$ denotes the derivative of $f$ at $x$ in the direction $y-x$.

\par (iii) Let $f: X \rightarrow \mathbb{R}\cup \{+\infty\}$ be differentiable. Then $f$ is strongly convex 
with constant of strong convexity $\alpha > 0$  with respect to norm $\|\cdot\|$ if and only if
for every $x, y \in \mbox{dom(f)}$ we have
$$
(\nabla f(y) - \nabla f (x))^T (y-x) \geq \alpha \|y-x\|^2.
$$
\par (iv) Let $f: X \rightarrow \mathbb{R}\cup \{+\infty\}$ be twice differentiable. Then $f$ is strongly convex on $X \subset \mathbb{R}^n$ with constant of strong convexity $\alpha > 0$ 
with respect to norm $\|\cdot\|$
if and only if for every $x \in \mbox{dom(f)}$ we have
$$
h^T \nabla^2 f(x) h \geq \alpha \|h\|^2, \forall h \in \mathbb{R}^n.
$$
\end{prop}
\begin{df}[Strongly concave functions] $f: X \rightarrow \mathbb{R}\cup \{-\infty\}$
is strongly concave with constant of strong concavity $\alpha > 0$ 
with respect to norm $\|\cdot\|$
if and only if $-f$ is strongly convex with constant of strong convexity $\alpha > 0$ 
with respect to norm $\|\cdot\|$.
\end{df}
The following propositions are immediate and will be used in the sequel:
\begin{prop}\label{immediate1} If $f:X \rightarrow \mathbb{R}\cup \{+\infty\}$
is strongly convex with constant of strong convexity $\alpha > 0$ 
with respect to norm $\|\cdot\|$ and $\ell: \mathbb{R}^n \rightarrow \mathbb{R}$
is linear then $f+\ell$ is strongly convex on $X$ with constant of strong convexity $\alpha > 0$ 
with respect to norm $\|\cdot\|$.
\end{prop}
\begin{prop}\label{immediate2} Let  $X \subset \mathbb{R}^m, Y \subset \mathbb{R}^n$, be two nonempty convex sets.
Let $\mathcal{A}: X \rightarrow Y$ be a linear operator and let
$f: Y \rightarrow \mathbb{R}\cup \{+\infty\}$ be a strongly convex function 
with constant of strong convexity $\alpha > 0$ 
with respect to a norm  $\|\cdot\|_{n}$ on $\mathbb{R}^n$ induced by scalar product $\langle \cdot, \cdot \rangle_n$ on $\mathbb{R}^n$.
Assume that Ker$(\mathcal{A}^* \circ \mathcal{A} )=\{0\}$.
Then $g= f \circ \mathcal{A}$ is strongly convex on $X$ with constant of strong convexity $\alpha \lambda_{\min}( \mathcal{A}^* \circ \mathcal{A} )$
with respect to norm $\|\cdot\|_{m}$.
\end{prop}
\begin{proof} For every $x, y \in X$, using Proposition \ref{charc1scfunc}-(ii) we have
$$
f(\mathcal{A}(y)) \geq  f(\mathcal{A}(x)) + f'(\mathcal{A}(x) ; \mathcal{A}(y-x))  + \frac{\alpha}{2} \| \mathcal{A} (y-x) \|_{n}^2 
$$
and since $g'(x;y-x)=f'(\mathcal{A}(x) ; \mathcal{A}(y-x))$, we get
$$
g(y) \geq g(x) + g'(x;y-x) + \frac{1}{2} \alpha \lambda_{\min}( \mathcal{A}^*  \circ \mathcal{A} )  \|y - x\|_{m}^2
$$
with $\alpha \lambda_{\min}( \mathcal{A}^* \circ \mathcal{A} )>0$ ($\lambda_{\min}( \mathcal{A}^* \circ  \mathcal{A} )$
is nonnegative because $\mathcal{A}^*  \circ \mathcal{A} $ is self-adjoint and it cannot be zero because
$\mathcal{A}^* \circ \mathcal{A}$ is nondegenerate).\hfill
\end{proof}
In the rest of this section, we fix $\|\cdot\|=\|\cdot\|_2$ and provide, under some assumptions, the constant
of strong concavity of the dual function of an optimization problem for this norm.\footnote{Using the equivalence between
norms in $\mathbb{R}^n$, we can derive a valid constant of strong concavity for other norms, for instance $\|\cdot\|_{\infty}$ and $\|\cdot\|_1$.}

\subsection{Problems with linear constraints}

Consider the optimization problem 
\begin{equation} \label{pbinit0}
\left\{
\begin{array}{l}
\inf f(x)\\
A x \leq b
\end{array}
\right.
\end{equation}
where $f:\mathbb{R}^n \rightarrow \mathbb{R} \cup \{+\infty\}$,
$b \in \mathbb{R}^q$, and $A$ is a $q \times n$ real matrix.

\par We will use the following known fact, see for instance \cite{Rockafellar}:
\begin{prop}\label{conjstconvex}
Let $f:\mathbb{R}^n \rightarrow \mathbb{R}\cup \{+\infty\}$ be a proper convex lower semicontinuous function. Then
$f^*$ is strongly convex with constant of strong convexity $\alpha>0$ for norm $\|\cdot\|_2$ if and only if $f$ is differentiable and $\nabla f$ is Lipschitz continuous with
constant $1/\alpha$ for norm $\|\cdot\|_2$.
\end{prop}
\begin{prop}\label{dualfunctionpbinit0}
Let $\theta$ be the dual function of \eqref{pbinit0} given by
\begin{equation}\label{dualfunctionfirst}
\theta( \lambda ) = \displaystyle  \inf_{x \in \mathbb{R}^n} \{f(x) + \lambda^T ( Ax  - b ) \},
\end{equation}
for $\lambda \in \mathbb{R}^q$. Assume that the rows of matrix $A$ are independent, that 
$f$ is convex, differentiable, and $\nabla f$ is Lipschitz continuous with
constant $L\geq 0$ with respect to norm $\|\cdot\|_2$.
Then dual function $\theta$ is strongly concave on
$\mathbb{R}^q$ with constant of strong concavity $ \frac{\lambda_{\min}( A A^T )}{L}$ with respect to norm $\|\cdot\|_2$ on $\mathbb{R}^q$.
\end{prop}
\begin{proof}
The dual function of \eqref{pbinit0} can be written
\begin{equation}\label{thetaintermsoffstar}
\begin{array}{lll}
\theta( \lambda )& =&\displaystyle  \inf_{x \in \mathbb{R}^n} \{f(x) + \lambda^T  ( A x - b ) \} = 
-\lambda^T b - \sup_{x \in \mathbb{R}^n} \{ -x^T A^T \lambda  -f(x) \} \\
& = & - \lambda^T b - f^*( -A^T \lambda ) \mbox{ by definition of }f^*.
\end{array}
\end{equation}
Since the rows of $A$ are independent, matrix $AA^T$ is invertible and $\mbox{Ker}(A A^T)=\{0\}$.
The result follows 
from the above representation of $\theta$ and Propositions \ref{immediate1}, \ref{immediate2}, and \ref{conjstconvex}.\hfill
\end{proof}
The strong concavity of the dual function of \eqref{pbinit0} was shown in
Corollary 5 in \cite{YuNeely2015} assuming that $f$ is second-order continuously differentiable and strongly convex.
Therefore Proposition \ref{dualfunctionpbinit0} (whose proof is very short), which only assumes that $f$ is convex, differentiable, and 
has Lipschitz continuous gradient, 
improves existing results (neither second-order differentiability nor strong convexity is required). 
\if{
Below, we discuss several examples. In particular, in 
Examples \ref{exquadprogconcth} and \ref{genexample}, $f$ may not be
strongly convex.
\begin{ex}[Linear programs] Let $f:\mathbb{R}^n \rightarrow \mathbb{R}$ given by
\begin{equation}
f(x)=c^T x + c_0 
\end{equation}
where $c \in \mathbb{R}^n$, $c_0 \in \mathbb{R}$. Clearly $f$ is convex differentiable with Lipschitz continuous gradients; any
$L \geq 0$ being a valid Lipschitz constant. Proposition \eqref{dualfunctionpbinit0}
tells us that if the rows of $A$ are independent then dual function $\theta$ of \eqref{pbinit0}
given by \eqref{dualfunctionfirst} is strongly concave on $\mathbb{R}^q$.
In this case, the strong concavity can be checked directly computing $\theta$.
Indeed, we have 
$$
f^*(x)=\left\{
\begin{array}{l}
-c_0  \;\mbox{    if }x=c,\\
+\infty  \mbox{ if }x \neq c,
\end{array}
\right.
$$
and plugging this expression of $f^*$ into \eqref{thetaintermsoffstar}, we get\footnote{In this simple case, the dual function
is well known and can also be obtained without using the conjugate of $f$}
$$
\theta(\lambda ) =\left\{
\begin{array}{ll}
-\lambda^T b +c_0 & \mbox{if }A^T \lambda = -c,\\
-\infty & \mbox{if }A^T \lambda \neq -c.
\end{array}
\right.
$$
Therefore if $c \in \mbox{Im}(A^T)$ then there is $\lambda \in \mathbb{R}^q$ such that 
\begin{equation}\label{equac}
 A^T \lambda = - c,
\end{equation}
and if the rows of
$A$ are independent then there is only one $\lambda$, let us call it $\lambda_0$, satisfying \eqref{equac}. In this situation,
the domain of $\theta$ is a singleton: $\mbox{dom}(\theta)=\{\lambda_0\}$, and $\theta$ indeed is strongly convex (see Definition \ref{defsconv}).
If $c \notin \mbox{Im}(A^T )$ then $\mbox{dom}(\theta)=\emptyset$ and $\theta$ is again strongly convex. 
\end{ex}
The example which follows gives a class of problems where the dual function is strongly concave on $\mathbb{R}^q$:
\begin{ex}[Quadratic convex programs]\label{exquadprogconcth} Consider a problem of form \eqref{pbinit0}  
where $f(x)=\frac{1}{2} x^T Q_0 x + a_0^T x + b_0$, 
$Q_0$ is an $n\times n$ nonnull positive semidefinite matrix,
$A$ is a $q \times n$ real matrix, $a_0 \in \mbox{Im}(Q_0)$, and $b_0 \in \mathbb{R}$.
Clearly, $f$ is convex, differentiable, and  $\nabla f$ is Lipschitz continuous with Lipschitz constant
$L=\|Q_0\|_2= \lambda_{\max}(Q_0)>0$ with respect to $\|\cdot\|_2$ on $\mathbb{R}^n$.
If the rows of $A$ are independent, using Proposition \ref{dualfunctionpbinit0}
we obtain that the dual function of \eqref{pbinit0} is strongly concave with constant of
strong concavity $\frac{\lambda_{\min}(A A^T)}{\lambda_{\max}(Q_0)}>0$ with respect to norm $\|\cdot\|_2$ on $\mathbb{R}^q$.
Observe that strong concavity holds in particular if $Q_0$ is not positive definite, in which case $f$ is not strongly convex.

Since $f$ is convex, differentiable, its gradient being Lipschitz continuous with Lipschitz constant
$\lambda_{\max}(Q_0)$, from Proposition \ref{conjstconvex}, we know that $f^*$ is strongly convex with constant of
strong convexity $1/\lambda_{\max}(Q_0)$.
This can be checked by direct computation. Indeed, let
$\lambda_{\max}(Q_0)=\lambda_1(Q_0) \geq \lambda_{2}(Q_0) \geq \ldots \geq \lambda_{r}(Q_0)> \lambda_{r+1}(Q_0) = \lambda_{r+2}(Q_0)=\ldots = \lambda_{n}(Q_0)=0$ be the ordered eigenvalues of $Q_0$ where $r$ is the rank of $Q_0$. 
Let $P$ be a corresponding orthogonal matrix of eigenvectors for $Q_0$, i.e., $\emph{Diag}(\lambda_1(Q_0),\ldots,\lambda_n(Q_0))=P^T Q_0 P$ with $P P^T = P^T P =I_n$.
Defining 
$$
Q_0^{+}=P \emph{Diag}\Big(\frac{1}{\lambda_1(Q_0)},\ldots,\frac{1}{\lambda_{r}(Q_0)},\underbrace{0,\ldots,0}_{\mbox{n-r times}}\Big)P^T,
$$
it is straightforward to check that
\begin{equation}\label{expressionfstar}
f^*(x)=
\left\{
\begin{array}{ll}
-b_0+\frac{1}{2}(x-a_0)^T Q_0^{+}(x-a_0)&\mbox{if }x \in \mbox{Im}(Q_0),\\
+ \infty & \mbox{otherwise},
\end{array}
\right.
\end{equation}
and plugging expression \eqref{expressionfstar} of $f^*$ into \eqref{thetaintermsoffstar}, we get
$$
\theta(\lambda)=
\left\{
\begin{array}{ll}
b_0-\lambda^T b -\frac{1}{2}(a_0 + A^T \lambda)^T Q_0^{+}  (a_0 + A^T \lambda) & \mbox{if }A^T \lambda \in \mbox{Im}(Q_0),\\
-\infty & \mbox{otherwise.}
\end{array}
\right.
$$
If $x'=(x'_1,\ldots,x'_n)$ is the vector of the coordinates of $x$ 
in the basis $(v_1,v_2,\ldots,v_n)$ where $v_i$ is $i$th column of $P=[v_1,v_2,\ldots,v_n]$ (i.e., $(v_1,\ldots,v_r)$ is a basis
of Im($Q_0$) and $(v_{r+1},\ldots,v_n)$ is a basis of Ker($Q_0$)) and writing
$a_0 = \sum_{i=1}^r a'_{0 i} v_i$, we obtain
$$
f^*(x)=
\left\{
\begin{array}{l}
g(P^T x) \mbox{ where } g:\mathbb{R}^n \rightarrow \mathbb{R} \mbox{ is given by }g(x')  = -b_0 + \sum_{i=1}^r \frac{(x'_i - a'_{0 i} )^2}{2 \lambda_i( Q_0 )}
 \mbox{ if }x \in \mbox{Im}(Q_0),\\
+\infty \mbox{ otherwise.}
\end{array}
\right.
$$
Observe that for $x', y' \in \mathbb{R}^r \small{\times} \{ \underbrace{(0,\ldots,0)}_{\mbox{n-r times}} \}$
we have
$$
g(y') \geq g(x') + \nabla g(x')^T (y'-x') + \frac{1}{2 \lambda_1(Q_0)}\|y'-x'\|_2^2
$$
and $g$ is strongly convex with constant of strong convexity $\frac{1}{\lambda_1(Q_0)}$
with respect to norm $\|\cdot\|_2$ on $\mathbb{R}^r \small{\times} \{ \underbrace{(0,\ldots,0)}_{\mbox{n-r times}} \}$.
Recalling that $f^*(x)=g(P^T x)$ for $x \in \mbox{dom}(f^*)=\mbox{Im}(Q_0)$, that
$P^T x \in \mathbb{R}^r \small{\times} \{ \underbrace{(0,\ldots,0)}_{\mbox{n-r times}} \}$ for $x \in \mbox{Im}(Q_0)$,
and using Proposition \ref{immediate2}, we get that $f^*$
is strongly convex with constant of strong convexity
$$
\frac{\lambda_{\min}(P P^T)}{\lambda_1(Q_0)} = \frac{\lambda_{\min}(I_n)}{\lambda_{\max}( Q_0 )} = \frac{1}{\lambda_{\max}( Q_0 )}
$$
with respect to norm $\|\cdot\|_2$.
\end{ex}
\begin{ex}\label{genexample} Let $f(x)=\sum_{k=1}^M \alpha_k f_k(x)$ for $\alpha_k \in \mathbb{R}$
and $f_k: \mathbb{R}^n \rightarrow \mathbb{R}$ convex differentiable with Lipschitz constant
$L_k \geq 0$ with respect to norm $\|\cdot\|_2$ on $\mathbb{R}^n$ for $k=1,\ldots,M$.
Let $A$ be a $q\times n$ matrix with independent rows. Then dual function \eqref{dualfunctionfirst}
of \eqref{pbinit0} is strongly concave on $\mathbb{R}^q$ with constant of strong concavity
$\lambda_{\min}(A A^T)/\sum_{k=1}^M \alpha_k L_k$  with respect to $\|\cdot\|_2$.
\end{ex}
}\fi

\subsection{Problems with quadratic objective and a quadratic constraint}

We now consider the following quadratically constrained quadratic optimization problem
\begin{equation} \label{pbquadratic}
\left\{
\begin{array}{l}
\inf_{x \in \mathbb{R}^n} f(x):=\frac{1}{2}x^T Q_0 x + a_0^T x + b_0\\
g_1(x):=\frac{1}{2}x^T Q_1 x + a_1^T x + b_1 \leq 0,
\end{array}
\right.
\end{equation}
with $Q_0$ positive definite and $Q_1$, positive semidefinite.
The dual function $\theta$ of this problem is known in closed-form: for $\mu \geq 0$, we have
\begin{equation}\label{dualquadf}
\theta(\mu ) = \inf_{x \in \mathbb{R}^n} \{ f(x) + \mu g_1(x) \} =
-\frac{1}{2} \mathcal{A}(\mu)^T \mathcal{Q}(\mu)^{-1} \mathcal{A}(\mu) + \mathcal{B}( \mu ) 
\end{equation}
where
$$
\mathcal{A}(\mu)=a_0 + \mu a_1, \,\mathcal{Q}(\mu)=Q_0 + \mu Q_1, \,\mbox{ and }\mathcal{B}(\mu)=b_0 + \mu_i b_1.
$$
We can show, under some assumptions, that dual function $\theta$ is strongly concave
on some set and 
compute analytically the corresponding constant of strong concavity:
\begin{prop}\label{strongconcquad}
Consider optimization problem \eqref{pbquadratic}.
Assume that $Q_0, Q_1$, are positive definite, that there exists $x_0$ such that $g_1( x_ 0 ) <0$, 
and that $a_0 \neq Q_0 Q_1^{-1} a_1$.
Let $\mathcal{L}$ be any lower bound on the optimal value
of \eqref{pbquadratic} and let ${\bar{\mu}}=(\mathcal{L} - f( x_0 ))/g_1( x_0 ) \geq 0$.
Then the optimal solution of the dual problem 
$$
\max_{\mu \geq 0} \theta( \mu )
$$
is contained in the interval $[0, \bar \mu]$ and the 
dual function $\theta$ given by \eqref{dualquadf} is strongly concave on the interval 
$[0, {\bar{\mu}}]$ with constant of strong concavity $\alpha_D=(Q_1^{-1/2} (a_0-Q_0 Q_1^{-1}a_1))^T (Q_1^{-1/2}  Q_0 Q_1^{-1/2}  + \bar \mu I_n )^{-3} Q_1^{-1/2} (a_0-Q_0 Q_1^{-1}a_1)>0$.
\end{prop}
\begin{proof}
Making the change of variable 
$x=y-Q_1^{-1} a_1$, we can rewrite \eqref{pbquadratic} without linear terms in $g_1$ under the form:
$$
\left\{
\begin{array}{l}
\inf_{x \in \mathbb{R}^n} \frac{1}{2}x^T Q_0 x + (a_0-Q_0 Q_1^{-1} a_1)^T x + b_0 
+ \frac{1}{2}a_1^T Q_1^{-1} Q_0 Q_1^{-1}a_1 - a_0^T Q_1^{-1} a_1 \\
\frac{1}{2} x^T Q_1 x +b_1-\frac{1}{2}a_1^T Q_1^{-1}a_1 \leq 0,
\end{array}
\right.
$$
with corresponding dual function given by
$$
\theta( \mu )=-\frac{1}{2}\bar a_0^T (Q_0 + \mu Q_1 )^{-1} \bar a_0 + (b_1 -\frac{1}{2}a_1^T Q_1^{-1} a_1) \mu + b_0 - a_0^T Q_1^{-1} a_1 + \frac{1}{2} a_1^T Q_1^{-1} Q_0 Q_1^{-1} a_1
$$
where we have set $\bar a _0 = a_0 - Q_0 Q_1^{-1} a_1$
(see \eqref{dualquadf}).

Using \cite[Remark 2.3.3, p.313]{hhlem} we obtain that the optimal dual solutions are contained in the interval $[0,\bar \mu]$.
Setting ${\tilde a}_0 =Q_1^{-1/2} \bar a_0$ and $A=Q_1^{-1/2}  Q_0 Q_1^{-1/2}$, we compute the first and second
derivatives of the nonlinear term  $\theta_q( \mu )=- \frac{1}{2}\bar a_0^T (Q_0 + \mu Q_1)^{-1} \bar a_0= -\frac{1}{2}{\tilde a}_0^T ( A + \mu I_n )^{-1} {\tilde a}_0$  of $\theta$ on 
$[0, {\bar{\mu}}]$:
$$
\begin{array}{l}
\theta_q'( \mu )  = \frac{1}{2} {\tilde a}_0^T (A+ \mu I_n )^{-2} {\tilde a}_0\;\mbox{ and }\theta_q''( \mu )  =  -{\tilde a}_0^T (A+ \mu I_n )^{-3} {\tilde a}_0.
\end{array}
$$
For these computations we have used the fact that for $\mathcal{F}:\mathcal{I} \rightarrow $ GL$_{n}(\mathbb{R})$ differentiable on $\mathcal{I} \subset \mathbb{R}$, we have
$\frac{d \mathcal{F}(t)^{-1}}{dt}=-\mathcal{F}(t)^{-1}  \frac{d \mathcal{F}(t)}{dt} \mathcal{F}(t)^{-1}$.
Since $-\theta_q''( \mu )$ is decreasing on $[0, {\bar{\mu}}]$, we get $-\theta_q''( \mu ) \geq \alpha_D = -\theta_q''( {\bar \mu} )$ on $[0, {\bar{\mu}}]$.
This computation, together with Proposition \ref{charc1scfunc}-(iv), shows that
$\theta$ is strongly concave on $[0, {\bar{\mu}}]$ with constant of strong concavity $\alpha_D$. \hfill
\end{proof}

\subsection{General case: problems with linear and nonlinear constraints}

Let us add to problem  \eqref{pbinit0} nonlinear constraints.
More precisely, given $f:\mathbb{R}^n \rightarrow \mathbb{R}$,
a $q \times n$ real matrix $A$, $b \in \mathbb{R}^q$, 
and $g:\mathbb{R}^n \rightarrow \mathbb{R}^p$ with convex component functions
$g_i, i=1,\ldots,p$, we consider  the optimization problem 
\begin{equation} \label{optclassgeneral}
\left\{
\begin{array}{l}
\inf f(x)\\
x \in X, A x \leq  b, g(x) \leq 0.
\end{array}
\right.
\end{equation}
Let $v$ be the value function
of this problem given by
\begin{equation}\label{valuedefv}
v(c)= v(c_1, c_2)=\left\{
\begin{array}{l}
\inf f(x)\\
x \in X, A x - b + c_1 \leq  0, g(x) + c_2 \leq 0,
\end{array}
\right.
\end{equation}
for $c_1 \in \mathbb{R}^q, c_2 \in \mathbb{R}^p$. In the next lemma, we relate the conjugate of $v$ to the dual function 
$$
\theta(\lambda ,\mu ) =   \left\{
\begin{array}{l}
\inf f(x) + \lambda^T ( Ax - b ) + \mu^T g(x)\\
x \in X,
\end{array}
\right.
$$
of this problem:
\begin{lemma}\label{conjv1} If $v^*$ is the conjugate of the value function $v$ then
$v^*(\lambda ,\mu) = -\theta(\lambda , \mu)$ for every $(\lambda, \mu) \in \mathbb{R}_{+}^q \small{\times} \mathbb{R}_{+}^p$.
\end{lemma}
\begin{proof}
For $(\lambda, \mu) \in \mathbb{R}_+^q \small{\times} \mathbb{R}_{+}^p$, we have
\begin{eqnarray*}
-v^*(\lambda, \mu) & = & -\sup_{(c_1, c_2) \in  \mathbb{R}^q \small{\times} \mathbb{R}^p } \lambda^T  c_1  + \mu^T c_2 -v(c_1, c_2)\\
& =  & \left\{
\begin{array}{l}
\inf -\lambda^T c_1 - \mu^T c_2 + f(x)\\
x \in X, Ax - b + c_1 \leq  0, g(x) + c_2 \leq 0,\\
c_1 \in \mathbb{R}^q, c_2 \in \mathbb{R}^p,
\end{array}
\right.\\
& =  & \left\{
\begin{array}{l}
\inf f(x) + \lambda^T ( Ax - b ) + \mu^T g(x) \\
x \in X,
\end{array}
\right.\\
& = & \theta( \lambda, \mu).
\end{eqnarray*}\hfill
\end{proof}
From Lemma \ref{conjv1} and Proposition \ref{conjstconvex}, we obtain that dual function $\theta$ of problem \eqref{optclassgeneral}
is strongly concave with constant $\alpha$ with respect to norm $\|\cdot\|_2$ on $\mathbb{R}^{p+q}$
if and only if the value function $v$ given by \eqref{valuedefv} is differentiable and $\nabla v$  is Lipschitz continuous with
constant $1/\alpha$ with respect to norm $\|\cdot\|_2$ on $\mathbb{R}^{p+q}$. Using Lemma 2.1 in \cite{guiguessiopt2016} the subdifferential of the value function is the set of optimal dual solutions of \eqref{valuedefv}.
Therefore $\theta$ is strongly concave with constant $\alpha$ with respect to norm $\|\cdot\|_2$ on $\mathbb{R}^{p+q}$ if and only if the value function is differentiable and the dual solution
of \eqref{valuedefv} seen as a function of $(c_1,c_2)$ is Lipschitz continuous with Lipschitz constant $1/\alpha$ with respect to norm $\|\cdot\|_2$ on $\mathbb{R}^{p+q}$.

We now provide conditions ensuring that the dual function is strongly concave in a neighborhood of the optimal dual solution.
\begin{thm}\label{thmstrongconcloc} Consider the optimization  problem 
\begin{equation}\label{defpblinnonlinstrongconc}
\inf_{x \in \mathbb{R}^n} \{ f(x) :  Ax \leq b, g_i(x) \leq 0, i=1,\ldots,p\}.
\end{equation}
We assume that
\begin{itemize}
\item[(A1)] $f:\mathbb{R}^n \rightarrow \mathbb{R} \cup \{+\infty\}$ is strongly convex and has Lipschitz continuous gradient;
\item[(A2)] $g_i:\mathbb{R}^n \rightarrow \mathbb{R}\cup \{+\infty\}, i=1,\ldots,p$, are convex and have Lipschitz continuous gradients;
\item[(A3)] if $x_*$ is the optimal solution of \eqref{defpblinnonlinstrongconc} then
the rows of matrix  
 $\left(\begin{array}{c}A \\ J_g(x_* )\end{array}  \right)$ are linearly independent
 where $J_g(x)$ denotes the Jacobian matrix of $g(x)=(g_1(x),\ldots,g_p(x))$ at $x$;
\item[(A4)] there is $x_0 \in \mbox{ri}(\{g \leq 0\})$ such that $A x_0 \leq b$. 
\end{itemize}
Let $\theta$ be the dual function of this problem:
\begin{equation}\label{defthetadualgen}
\theta(\lambda , \mu ) = 
\left\{
\begin{array}{l}
\inf \;f(x) + \lambda^T ( Ax - b )  + \mu^T g(x) \\
x \in \mathbb{R}^n.
\end{array}
\right.
\end{equation}
Let $(\lambda_* , \mu_* ) \geq 0$ be an optimal solution of the dual problem
$$
\sup_{\lambda \geq 0, \mu \geq 0} \theta(\lambda , \mu ).
$$
Then there is some  neighborhood $\mathcal{N}$ of $(\lambda_* , \mu_* )$ such that $\theta$ is strongly concave on $\mathcal{N} \cap \mathbb{R}_{+}^{p+q}$.
\end{thm}
\begin{proof}
Due to (A1) the optimization problem \eqref{defthetadualgen} has a unique
optimal solution that we denote by $x(\lambda, \mu)$.
Assumptions (A2) and  (A3) imply that there is some neighborhood $\mathcal{V}_{\varepsilon}(x_*)=\left\{x \in \mathbb{R}^n : \|x - x_*\|_{2} \leq \varepsilon \right\}$ 
of $x_*$ for some $\varepsilon>0$
such that 
the rows of matrix $\left(\begin{array}{c}A \\ J_g(x )\end{array}  \right)$ are independent for $x$ in $\mathcal{V}_{\varepsilon}( x_* )$.

We argue that $(\lambda,\mu) \rightarrow x(\lambda,\mu)$ is continuous on $\mathbb{R}^q \small{\times} \mathbb{R}^p$. Indeed, let 
$(\bar \lambda, \bar \mu) \in \mathbb{R}^q \small{\times} \mathbb{R}^p$ and take a sequence $(\lambda_k,\mu_k)$ converging to 
$(\bar \lambda, \bar \mu)$. We want to show that $x(\lambda_k, \mu_k)$ converges to $x(\bar \lambda, \bar \mu)$.
Take an arbitrary accumulation point $\bar x$ of the sequence $x(\lambda_k, \mu_k)$, i.e., $\bar x = \lim_{k \rightarrow +\infty} x(\lambda_{\sigma(k)}, \mu_{\sigma(k)})$
for some subsequence $x(\lambda_{\sigma(k)}, \mu_{\sigma(k)})$ of $x(\lambda_{k}, \mu_{k})$. Then by definition of $x(\lambda_{\sigma(k)}, \mu_{\sigma(k)})$, for every $x \in \mathbb{R}^n$
and every $k \geq 1$ we have
$$
f(x(\lambda_{\sigma(k)}, \mu_{\sigma(k)}) ) + \lambda_{\sigma(k)}^T (A x(\lambda_{\sigma(k)}, \mu_{\sigma(k)}) -b) + 
\mu_{\sigma(k)}^T g( x(\lambda_{\sigma(k)}, \mu_{\sigma(k)})) \leq f(x) + \lambda_{\sigma(k)}^T (Ax - b) + \mu_{\sigma(k)}^T g(x).
$$
Passing to the limit in the inequality above and using the continuity of $f$ and $g_i$ we obtain for all $x \in \mathbb{R}^n$:
$$
f(\bar x) + {\bar \lambda}^T (A \bar x -b) + {\bar \mu}^T g(\bar x) \leq f(x) + {\bar \lambda}^T (Ax - b) + {\bar \mu}^T g(x),
$$
which shows that $\bar x = x(\bar \lambda , \bar \mu)$. Therefore there is only one accumuation point 
$\bar x = x(\bar \lambda , \bar \mu)$ for the sequence $x(\lambda_k, \mu_k)$
which shows that this sequence converges to $x(\bar \lambda , \bar \mu)$.
Consequently, we have shown that $(\lambda,\mu) \rightarrow x(\lambda,\mu)$ is continuous on $\mathbb{R}^q \small{\times} \mathbb{R}^p$.
This implies that there  is a neighborhood 
$
\displaystyle \mathcal{N}(\lambda_*,\mu_*)$
of $(\lambda_* , \mu_*)$ such that for $(\lambda, \mu) \in \mathcal{N}(\lambda_*,\mu_*)$ we have 
$\|x(\lambda,\mu)-x(\lambda_*, \mu_*)\|_2  \leq \varepsilon$.
Moreover, due to (A4), we have $x(\lambda_*, \mu_*)=x_*$.
It follows that for $(\lambda, \mu) \in \mathcal{N}(\lambda_*,\mu_*)$ we have 
$\|x(\lambda,\mu)-x(\lambda_*, \mu_*)\|_2 = \|x(\lambda,\mu)-x_* \|_2 \leq \varepsilon$ which in turn implies that 
the rows of matrix $\left(\begin{array}{c}A \\ J_g(x(\lambda , \mu  ) )\end{array}  \right)$ are independent.
We now show that $\theta$ is strongly concave on $\mathcal{N}(\lambda_*,\mu_*) \cap \mathbb{R}_{+}^{p+q}$.

Take $(\lambda_1, \mu_1)$, $(\lambda_2 , \mu_2 )$ in $\mathcal{N}(\lambda_*,\mu_*) \cap \mathbb{R}_{+}^{p+q}$ and denote 
$x_1 = x(\lambda_1, \mu_1)$ and $x_2 = x(\lambda_2 , \mu_2 )$.
The optimality conditions give 
\begin{equation}\label{optconditionsx1x2}
\begin{array}{l}
\nabla f(x_1 ) + A^T  \lambda_1 + J_g( x_1 )^T \mu_1 = 0,\\
\nabla f(x_2 ) + A^T  \lambda_2 + J_g( x_2 )^T \mu_2 = 0.
\end{array}
\end{equation}
Recall that \eqref{defthetadualgen} has a unique solution 
and therefore $\theta$ is differentiable. The gradient of $\theta$ is given by (see for instance Lemma 2.1 in \cite{guiguessiopt2016}) 
$$
\nabla \theta( \lambda  , \mu ) = \left(
\begin{array}{l}
A x(\lambda, \mu ) - b \\
g( x(\lambda, \mu ) )
\end{array}
\right)
$$
and we obtain, using the notation $\langle x , y \rangle = x^T y$:
\begin{equation} \label{firstsctheta}
- \left \langle \nabla \theta ( \lambda_2, \mu_2 ) - \nabla \theta ( \lambda_1, \mu_1 ) , \left( \begin{array}{c} \lambda_2 - \lambda_1 \\ \mu_2 - \mu_1 \end{array} \right) \right \rangle = 
- \langle A ( x_2 - x_1 ) , \lambda_2 - \lambda_1 \rangle - \langle g(x_2 ) - g(x_1) , \mu_2 - \mu_1 \rangle.
\end{equation}
By convexity of constraint functions we can write for $i=1,\ldots,p$:
\begin{equation} \label{convexgi}
\begin{array}{ll}
g_i( x_2 ) \geq  g_i( x_1 ) + \langle \nabla g_i (x_1 ) , x_2 - x_1 \rangle & (a) \\
g_i( x_1 ) \geq  g_i( x_2 ) + \langle \nabla g_i (x_2 ) , x_1 - x_2 \rangle. & (b)
\end{array}
\end{equation}
Multiplying \eqref{convexgi}-(a) by $\mu_1(i) \geq 0$ and \eqref{convexgi}-(b) by $\mu_2(i) \geq 0$ we obtain
\begin{equation}\label{convexsumm}
-\langle g(x_2 ) - g(x_1) , \mu_2 - \mu_1 \rangle \geq \langle J_g (x_1 )^T \mu_1 - J_g( x_2 )^T \mu_2 , x_2 - x_1 \rangle. 
\end{equation}
Recalling (A1), we can find $0 \leq L(f)<+\infty$ such that for all $x,y \in \mathbb{R}^n$: 
\begin{equation}\label{lipgradf}
\|\nabla f(y) - \nabla f (x)\|_2 \leq L(f) \|y-x\|_2.
\end{equation}
Using \eqref{firstsctheta} and \eqref{convexsumm} and denoting by $\alpha>0$ the constant of strong convexity of $f$ with respect to norm $\|\cdot\|_2$ we get:
{\small{
\begin{equation}\label{secondsctheta}
\begin{array}{lcl}
- \left \langle \nabla \theta ( \lambda_2, \mu_2 ) - \nabla \theta ( \lambda_1, \mu_1 ) , \left( \begin{array}{c} \lambda_2 - \lambda_1 \\ \mu_2 - \mu_1 \end{array} \right) \right \rangle 
& \geq &  - \langle x_2 - x_1 , A^T (\lambda_2 - \lambda_1 )\rangle 
+ \langle J_g (x_1 )^T \mu_1 - J_g( x_2 )^T \mu_2 , x_2 - x_1 \rangle,\\
& \stackrel{\eqref{optconditionsx1x2}}{=} & \langle x_2 - x_1 , \nabla f(x_2 ) - \nabla f(x_1 ) \rangle  \\
& \geq & \alpha \| x_2   -   x_1  \|_2^2\mbox{ by strong convexity of }f,\\
& \geq & \frac{\alpha}{L(f)^2} \|\nabla f( x_2 )  -  \nabla f ( x_1 ) \|_2^2\mbox{ using }\eqref{lipgradf},\\
& \stackrel{\eqref{optconditionsx1x2}}{=} & \frac{\alpha}{L(f)^2} \| 
\underbrace{
\left(  A^T \; J_g( x_2 )^T \right) \left(
\begin{array}{c}
\lambda_2 - \lambda_1 \\
\mu_2 - \mu_1
\end{array}
\right)}_{a}  + \underbrace{(J_g(x_2) - J_g(x_1) )^T \mu_1}_{b}  \|_2^2.\\
\end{array}
\end{equation}
}}
Now recall that for every $x \in \mathcal{V}_\varepsilon( x_*)$
the rows of the matrix $\left(\begin{array}{c}A \\ J_g(x)\end{array}  \right)$ are independent
and therefore the matrix 
$\left(\begin{array}{c}A \\ J_g(x)\end{array}  \right) \left(\begin{array}{c}A \\ J_g(x)\end{array}  \right)^T$
is invertible.
Moreover, the function $x \rightarrow \lambda_{\min}\left(  \left(\begin{array}{c}A \\ J_g(x)\end{array}  \right) \left(\begin{array}{c}A \\ J_g(x)\end{array}  \right)^T   \right)$
is continuous (due to (A2)) and positive on the compact set $\mathcal{V}_\varepsilon( x_*)$.
It follows that we can define 
$$
\underline{\lambda}_{\varepsilon}( x_* ) = \displaystyle  \min_{x \in \mathcal{V}_\varepsilon( x_*)} 
\lambda_{\min}\left(  \left(\begin{array}{c}A \\ J_g(x)\end{array}  \right) \left(\begin{array}{c}A \\ J_g(x)\end{array}  \right)^T   \right),
$$
and $\underline{\lambda}_{\varepsilon}( x_* )>0$.
Since $x_2 \in \mathcal{V}_{\varepsilon}( x_* )$, we deduce that
\begin{equation}\label{bounda}
\|a\|_2 \geq \sqrt{\underline{\lambda}_{\varepsilon}( x_* ) }\left\| \left(  \begin{array}{c} \lambda_2 - \lambda_1 \\ \mu_2 - \mu_1 \end{array} \right) \right\|_2.
\end{equation}
Recalling that $(\lambda_1, \mu_1)$ is in $\mathcal{N}(\lambda_*,\mu_*)$, 
there is $\eta>0$ such that 
\begin{equation}\label{mu1bound}
\|\mu_1 \|_1 \leq U_{\eta}(\mu_*):=\|\mu_*\|_1 + \eta.
\end{equation}
Due to (A2), there is $L(g)\geq 0$ such that for every $x,y \in \mathbb{R}^n$, we have
$$
\|\nabla g_i(y) - \nabla g_i(x)\|_2 \leq L(g) \|y-x\|_2,1,\ldots,p.
$$
Combining  this relation with \eqref{mu1bound}, we get
\begin{equation}\label{boundb}
\|b\|_2 \leq  \| \mu_1 \|_1 L( g ) \| x_2 - x_1 \|_2 \leq L( g)  U_{\eta}(\mu_*) \|x_2 - x_1\|_2. 
\end{equation}
Therefore 
$$
\|a+b\|_2 \geq \|  a \|_2 - \|b\|_2 \geq \sqrt{\underline{\lambda}_{\varepsilon}( x_* ) }  \left\| \left(  \begin{array}{c} \lambda_2 - \lambda_1 \\ \mu_2 - \mu_1 \end{array} \right) \right\|_2 - L( g)  U_{\eta}(\mu_*) \|x_2 - x_1  \|_2
$$
and combining this relation with \eqref{secondsctheta} we obtain
$$
\|x_2 - x_1\|_2 \geq \frac{1}{L(f)} \left[ \sqrt{\underline{\lambda}_{\varepsilon}( x_* ) }  \left\| \left(  \begin{array}{c} \lambda_2 - \lambda_1 \\ \mu_2 - \mu_1 \end{array} \right) \right\|_2 - L( g)  U_{\eta}(\mu_*)  \|x_2 - x_1  \|_2     \right]
$$
which gives
\begin{equation}\label{finalrel}
 \|x_2 - x_1\|_2 \geq \frac{   \sqrt{\underline{\lambda}_{\varepsilon}( x_* )}}  {L(f) +L(g) U_{\eta}(\mu_*) } \left\| \left(  \begin{array}{c} \lambda_2 - \lambda_1 \\ \mu_2 - \mu_1 \end{array} \right) \right\|_2.
\end{equation}
Plugging \eqref{finalrel} into \eqref{secondsctheta} we get 
$$
\begin{array}{lcl}
- \left \langle \nabla \theta ( \lambda_2, \mu_2 ) - \nabla \theta ( \lambda_1, \mu_1 ) , \left( \begin{array}{c} \lambda_2 - \lambda_1 \\ \mu_2 - \mu_1 \end{array} \right) \right \rangle 
& \geq &  \frac{\alpha \underline{\lambda}_{\varepsilon}( x_* )   }{(L(f) + L(g) U_{\eta}(\mu_*)  )^2}
\left\| \left(  \begin{array}{c} \lambda_2 - \lambda_1 \\ \mu_2 - \mu_1 \end{array} \right) \right\|_2^2.
\end{array}
$$
Using Proposition \ref{charc1scfunc}-(iii), the relation above shows that
 $\theta$ is strongly concave on $\mathcal{N}(\lambda_*,\mu_*) \cap \mathbb{R}_{+}^{p+q}$ with constant of strong concavity 
 $\frac{\alpha \underline{\lambda}_{\varepsilon}( x_* )   }{(L(f) + L(g) U_{\eta}(\mu_*)  )^2}$
with respect to norm $\|\cdot\|_2$.
\hfill
\end{proof}
\if{
It is interesting to notice that inequalities similar to \eqref{secondsctheta} can be derived without the assumption of strong convexity of $f$.
For that, we use the fact that a convex function with Lipschitz continuous gradient is co-coercive meaning that it satisfies
$$
\langle x_2 - x_1 , \nabla f(x_2 ) - \nabla f(x_1 ) \rangle  \geq \frac{1}{L(f)}\|\nabla f( x_2 )  -  \nabla f ( x_1 )\|_2^2
$$
for all $x_1,x_2$, where $L(f)$ was defined above.
Therefore, without the assumption of strong convexity of $f$, using \eqref{firstsctheta} and \eqref{convexsumm} we get
{\small{
$$
\begin{array}{lcl}
- \left \langle \nabla \theta ( \lambda_2, \mu_2 ) - \nabla \theta ( \lambda_2, \mu_2 ) , \left( \begin{array}{c} \lambda_2 - \lambda_1 \\ \mu_2 - \mu_1 \end{array} \right) \right \rangle 
& \geq &  - \langle x_2 - x_1 , A^T (\lambda_2 - \lambda_1 )\rangle 
+ \langle J_g (x_1 )^T \mu_1 - J_g( x_2 )^T \mu_2 , x_2 - x_1 \rangle,\\
& \stackrel{\eqref{optconditionsx1x2}}{\geq} & \langle x_2 - x_1 , \nabla f(x_2 ) - \nabla f(x_1 ) \rangle  \\
& \geq & \frac{1}{L(f)} \|\nabla f( x_2 )  -  \nabla f ( x_1 ) \|_2^2\mbox{ by co-coercivity of }f,\\
& \stackrel{\eqref{optconditionsx1x2}}{=} & \frac{1}{L(f)} \| 
\underbrace{
\left(  A^T \; J_g( x_2 )^T \right) \left(
\begin{array}{c}
\lambda_2 - \lambda_1 \\
\mu_2 - \mu_1
\end{array}
\right)}_{a}  + \underbrace{(J_g(x_2) - J_g(x_1) )^T \mu_1}_{b}  \|_2^2.
\end{array}
$$
}}
}\fi
The local strong concavity of the dual function of \eqref{defpblinnonlinstrongconc} was shown recently in
Theorem 10 in \cite{YuNeely2015} assuming (A3), assuming instead of (A1) that $f$ is strongly convex and second-order continuously differentiable
(which is stronger than (A1)), and assuming instead of (A2) that 
$g_i,i=1,\ldots,p$, are convex second-order continuously differentiable, which is stronger than (A2).\footnote{Note that
we used (A4) to ensure that $x(\lambda_*, \mu_*) = x_*$, which is also used in the proof of Theorem 10 in \cite{YuNeely2015}.} 
Therefore Theorem \ref{thmstrongconcloc} gives a new proof of the local strong concavity of the dual function
and improves existing results.

\section{Computing inexact cuts for value functions of convex optimization problems}\label{sec:computing}

\subsection{Preliminaries}

Let $\mathcal{Q}: X\rightarrow \mathbb{R} \cup \{+\infty\}$ be the value function given by
\begin{equation} \label{vfunctionq}
\mathcal{Q}(x)=\left\{
\begin{array}{l}
\inf_{y \in \mathbb{R}^{n}} \;f(y,x)\\
y \in S(x):=\{y\in Y \;:\;Ay+Bx=b,\;g(y,x)\leq 0\}.
\end{array}
\right.
\end{equation}
Here, and in all this section, $X \subseteq \mathbb{R}^m$ and $Y \subseteq \mathbb{R}^n$ are nonempty, compact, and convex
sets, and $A$ and $B$ are respectively $q \small{\times} n$ and $q \small{\times} m$ real matrices.
We will make the following assumptions:\footnote{Note that (H1) and (H2) imply the convexity of $\mathcal{Q}$ given
by \eqref{vfunctionq}. Indeed,
let $x_1, x_2 \in X$, $0 \leq t \leq 1$,
and $y_1 \in S( x_1), y_2 \in S( x_2),$
such that $\mathcal{Q}( x_1)=f(y_1, x_1)$
and $\mathcal{Q}( x_2)=f(y_2, x_2)$.
By convexity of $g$ and $Y$, we have that 
have $ty_1 +(1-t)y_2 \in S(tx_1 + (1-t)x_2)$
and therefore $\mathcal{Q}(tx_1 +(1-t)x_2)
\leq f(ty_1 + (1-t)y_2,t x_1 + (1-t)x_2) \leq t f(y_1,x_1)
+(1-t)f(y_2,x_2)=t\mathcal{Q}(x_1)+(1-t)\mathcal{Q}(x_2)$
where for the last inequality we have used
the convexity of $f$.}
\begin{itemize}
\item[(H1)] $f:\mathbb{R}^n \small{\times} \mathbb{R}^m \rightarrow \mathbb{R} \cup \{+\infty\}$ is 
lower semicontinuous, proper, and convex.
\item[(H2)] For $i=1,\ldots,p$, the $i$-th component of function
$g(y, x)$ is a convex lower semicontinuous function
$g_i:\mathbb{R}^n \small{\times} \mathbb{R}^m \rightarrow \mathbb{R} \cup \{+\infty\}$.
\end{itemize}

In what follows, we say that $\mathcal{C}$ is a cut for $\mathcal{Q}$ on $X$ if $\mathcal{C}$ is an affine
function of $x$ such that $\mathcal{Q}(x) \geq \mathcal{C}(x)$ for all $x \in X$. We say that the cut
is exact at $\bar x$ if $\mathcal{Q}(\bar x) = \mathcal{C}(\bar x)$.
Otherwise, the cut is said to be inexact at $\bar x$.

In this section,  our basic goal is, given ${\bar x} \in X$ and
$\varepsilon$-optimal primal and dual solutions of \eqref{vfunctionq} written for $x=\bar x$, to derive
an inexact cut $\mathcal{C}(x)$ for $\mathcal{Q}$ at $\bar x$, i.e., an affine lower bounding function
for $\mathcal{Q}$ such that the distance
$\mathcal{Q}( \bar x ) - \mathcal{C}( \bar x)$ between the 
values of $\mathcal{Q}$ and of the cut at $\bar x$ is bounded from above by a known function of the problem
parameters. Of course, when $\varepsilon=0$, we will check that 
$\mathcal{Q}( \bar x )= \mathcal{C}( \bar x)$.

We first provide in Proposition \ref{dervaluefunction} below 
a characterization of the subdifferential of value function $\mathcal{Q}$ at $\bar x \in X$
when optimal primal and dual solutions for \eqref{vfunctionq}  written for $x=\bar x$ are available (computation of exact cuts).

Consider for problem \eqref{vfunctionq} the Lagrangian dual problem 
\begin{equation}\label{dualpb}
\displaystyle \sup_{(\lambda, \mu) \in \mathbb{R}^q \small{\times} \mathbb{R}_{+}^{p} }\; \theta_{x}(\lambda, \mu)
\end{equation}
for the dual function
\begin{equation}\label{defdualfunction}
\theta_{x}(\lambda, \mu)=\displaystyle \inf_{y \in Y} \;L_x(y,\lambda,\mu)
\end{equation}
where 
$$
L_x(y,\lambda,\mu)= f(y, x) + \lambda^T (Ay+Bx-b) + \mu^T g(y,x).
$$
We denote by $\Lambda(x)$ the set of optimal solutions of the  dual problem \eqref{dualpb}
and we use the notation
$$
\mbox{Sol}(x):=\{y \in S(x) : f(y,x)=\mathcal{Q}(x)\}
$$
to indicate the solution set to \eqref{vfunctionq}.

\begin{lemma}[Lemma 2.1 in \cite{guiguessiopt2016}]\label{dervaluefunction0}  
Consider the value function $\mathcal{Q}$ given by \eqref{vfunctionq} and take $\bar x \in X$
such that $S(\bar x)\neq \emptyset$. 
Let Assumptions (H1) and (H2) hold and
assume the Slater-type constraint qualification condition:
$$
there \;exists\; (x_* , y_* ) \in X{\small{\times}}\emph{ri}(Y) \mbox{ such that }A y_* + B x_* = b \mbox{ and } (y_* , x_* ) \in \emph{ri}(\{g \leq 0\}).
$$
Then $s \in \partial \mathcal{Q}(\bar x)$ if and only if
\begin{equation}\label{caractsubQ}
\begin{array}{l}
(0, s) \in  \partial f(\bar y, \bar x)+\Big\{[A^T; B^T ] \lambda \;:\;\lambda \in \mathbb{R}^q\Big\}\\
\hspace*{1.2cm}+ \Big\{\displaystyle \sum_{i \in I(\bar y, \bar x)}\; \mu_i \partial g_i(\bar y, \bar x)\;:\;\mu_i \geq 0 \Big\}+\mathcal{N}_{Y}(\bar y)\small{\times}\{0\},
\end{array}
\end{equation}
where $\bar y$ is any element in the solution set \mbox{Sol}($\bar x$) and with
$$I(\bar y, \bar x)=\Big\{i \in \{1,\ldots,p\} \;:\;g_i(\bar y, \bar x) =0\Big\}.$$ 
In particular, if $f$  and $g$ are differentiable, then 
\begin{equation}\label{subdifffordiff}
\partial \mathcal{Q}(\bar x)=\Big\{  \nabla_x f(\bar y, \bar x)+ B^T \lambda + \sum_{i \in I(\bar y, \bar x )}\; \mu_i \nabla_x g_i(\bar y, \bar x)\;:\; (\lambda, \mu) \in \Lambda(\bar x) \Big\}.
\end{equation}
\end{lemma}
The proof of Lemma \ref{dervaluefunction0}  is given in \cite{guiguessiopt2016} using calculus on normal and tangeant cones.
In Proposition \ref{dervaluefunction} below, we show how to obtain an exact cut for $\mathcal{Q}$ at $\bar x \in X$
using convex duality when $f$ and $g$ are differentiable.
\begin{prop}\label{dervaluefunction} 
Consider the value function $\mathcal{Q}$ given by \eqref{vfunctionq} and take $\bar x \in X$
such that $S(\bar x)\neq \emptyset$. 
Let Assumptions (H1) and (H2) hold and assume the following constraint qualification condition:
there exists $y_0 \in \mbox{ri}(Y) \cap \mbox{ri}(\{g(\cdot,\bar x) \leq 0\})$ such that $A y_0 + B \bar x = b$.
Assume that $f$ and $g$ are differentiable on $Y \times X$.
Let $(\bar \lambda, \bar \mu)$ be an optimal  solution of dual problem \eqref{dualpb} written with $x=\bar x$ and let
\begin{equation}\label{caractsubQ}
s( \bar x ) =  \nabla_x f(\bar y, \bar x)+ B^T \bar \lambda +  \displaystyle \sum_{i \in I(\bar y, \bar x)}\; \bar \mu_i \nabla_x g_i(\bar y, \bar x),
\end{equation}
where $\bar y$ is any element in the solution set \mbox{Sol}($\bar x$) and with
$$I(\bar y, \bar x)=\Big\{i \in \{1,\ldots,p\} \;:\;g_i(\bar y, \bar x) =0\Big\}.$$
Then $s( \bar x ) \in \partial \mathcal{Q}(\bar x)$ .
\end{prop}
\begin{proof}
The constraint qualification condition implies that there is no duality gap and therefore 
\begin{equation}\label{firstqxarf}
f(\bar y, \bar x) =\mathcal{Q}(\bar x) = \theta_{\bar x}( \bar \lambda, \bar \mu ).
\end{equation}
Moreover, $\bar y$ is an optimal solution of $\inf \{ L_{\bar x}(y, \bar \lambda, \bar \mu) : y \in Y\}$ which gives
$$
\langle \nabla_y L_{\bar x}(\bar y, \bar \lambda, \bar \mu), y - \bar y \rangle \geq 0 \; \forall y \in Y,
$$
and therefore
\begin{equation}\label{optcondybar}
\min_{y \in Y} \langle \nabla_{y} L_{\bar x}(\bar y, \bar \lambda , \bar \mu ) , y- \bar y \rangle = 0.
\end{equation}
Using the convexity of the function which associates to $(x,y)$ the value 
$L_{x}(y,\bar \lambda , \bar \mu)$ we obtain for every $x\in X$ and $y \in Y$ that
\begin{equation}\label{secondexactcut}
L_x(y,{\bar \lambda}, \bar \mu) \geq 
L_{\bar x}(\bar y,\bar \lambda, \bar \mu)
+ \langle \nabla_{x} L_{\bar x}( \bar y, \bar \lambda , \bar \mu ) , x- \bar x \rangle 
+ \langle \nabla_{y} L_{\bar x}(\bar y, \bar \lambda , \bar \mu ) , y- \bar y \rangle.
\end{equation}
By definition of $\theta_x$, for any $x \in X$ we get
$$
\mathcal{Q}( x ) \geq \theta_x ( \bar \lambda  , \bar \mu ) 
$$
which combined with \eqref{secondexactcut} gives 
$$
\begin{array}{lcl}
\mathcal{Q}( x ) & \geq & L_{\bar x}(\bar y,\bar \lambda, \bar \mu)
+ \langle \nabla_{x} L_{\bar x}( \bar y, \bar \lambda , \bar \mu ) , x- \bar x \rangle 
+ \min_{y \in Y} \langle \nabla_{y} L_{\bar x}(\bar y, \bar \lambda , \bar \mu ) , y- \bar y \rangle \\
& \stackrel{\eqref{optcondybar}}{=} & L_{\bar x}(\bar y,\bar \lambda, \bar \mu) + \langle  \nabla_x f(\bar y, \bar x)+ B^T \bar \lambda +  \displaystyle \sum_{i=1}^p\; \bar \mu_i \nabla_x g_i(\bar y, \bar x) , x - \bar x \rangle,\\
& = & \mathcal{Q}( \bar x ) + \langle s(\bar x ) , x - \bar x \rangle
\end{array}
$$
where the last equality follows from \eqref{firstqxarf}, $A \bar y + B \bar x = b$ (feasibility of $\bar y$), $\langle \bar \mu , g(\bar y , \bar x ) \rangle = 0$, and 
${\bar \mu}_i = 0$ if $i \notin I(\bar y , \bar x)$(complementary slackness for $\bar y$).
\hfill
\end{proof}

\subsection{Inexact cuts with fixed feasible set}

As a special case of \eqref{vfunctionq}, we first consider value
functions where the argument only  appears in the objective of optimization problem \eqref{vfunctionq}:
\begin{equation} \label{vfunction1}
\mathcal{Q}(x)=\left\{
\begin{array}{l}
\inf_{y \in \mathbb{R}^n} \;f(y, x)\\
y \in Y.
\end{array}
\right.
\end{equation}

We fix $\bar x \in X$ and
denote by $\bar y \in Y$ an optimal solution of \eqref{vfunction1} written for $x =\bar x$:
\begin{equation}\label{optfixedset}
\mathcal{Q}( \bar x ) = f(\bar y , \bar x).
\end{equation}

If $f$ is differentiable, using Proposition \ref{dervaluefunction}, we have that $\nabla_x f(\bar y, \bar x) \in \partial \mathcal{Q}( \bar x )$ 
and
$$
\mathcal{C}(x):= \mathcal{Q}(\bar x) + \langle  \nabla_x f(\bar y, \bar x)    ,  x - \bar x \rangle 
$$
is an exact cut for $\mathcal{Q}$ at $\bar x$. If
instead of an optimal solution $\bar y$ of \eqref{vfunction1}, we only have 
at hand an approximate $\varepsilon$-optimal solution $\hat y( \varepsilon )$, 
Proposition \ref{fixedprop1} below
gives an inexact cut for $\mathcal{Q}$ at $\bar x$:
\begin{prop}[Proposition 2.2 in \cite{guigues2016isddp}] \label{fixedprop1} Let $\bar x \in X$ and
let $\hat y(\varepsilon) \in Y$ be an $\epsilon$-optimal solution for problem \eqref{vfunction1}
written for $x= \bar x$ with optimal value $\mathcal{Q}( \bar x )$, i.e., $\mathcal{Q}( \bar x ) \geq f( \hat y(\varepsilon) , \bar x ) - \varepsilon$.
Assume that $f$ is convex and differentiable on $Y \small{\times} X$.
Then setting  $\eta( \varepsilon , \bar x )=\ell_1( \hat y(\varepsilon) , \bar x )$
where $\ell_1 :Y \small{\times} X \rightarrow \mathbb{R}_{+}$ is the function given  by
\begin{equation}\label{defrxy}
\ell_1 (\hat y , \bar x  ) = - \min_{y \in Y} \langle \nabla_y f( \hat y ,  \bar x  ) , y - \hat y  \rangle = \max_{y \in Y} \langle \nabla_y f(    \hat y , \bar x) , \hat y - y  \rangle,
\end{equation}
the affine function
\begin{equation}\label{cutfixed1}
\mathcal{C}(x):= f ( \hat y(\varepsilon) , \bar x ) - \eta(\varepsilon,\bar x)  + \langle \nabla_x f(  \hat y(\varepsilon) , \bar x ) , x - \bar x  \rangle
\end{equation}
is a cut for $\mathcal{Q}$ at $\bar x$, i.e., for every $x \in X$ we have
$\mathcal{Q}(x) \geq \mathcal{C}(x)$ and the quantity $\eta(\varepsilon,\bar x)$
is an upper bound for
the 
distance 
$\mathcal{Q}( \bar x ) - \mathcal{C}( \bar x)$ between the 
values of $\mathcal{Q}$ and of the cut at $\bar x$. 
\end{prop}
\begin{rem} If $\varepsilon=0$ then $\hat y( \varepsilon )$ is an optimal solution of problem \eqref{vfunction1}
written for $x= \bar x$, $\eta(\varepsilon,\bar x) = \ell_1(\hat y( \varepsilon )  ,  \bar x  )=0$ and 
the cut given by Proposition \ref{fixedprop1} is exact. Otherwise it is inexact. 
\end{rem} 

In Proposition \ref{fixedprop3} below, we derive inexact cuts with an additional assumption of strong convexity on $f$: 
\begin{itemize}
\item[(H3)] $f$ is convex and differentiable on $Y \small{\times} X$  and
for every $x \in X$ there  exists $\alpha( x )>0$ such that  the function $f(\cdot , x)$ is strongly convex on $Y$ with constant of strong convexity $\alpha( x )>0$ for $\|\cdot\|_2$:
$$
f(y_2 , x) \geq f(y_1 , x) + (y_2 - y_1)^T \nabla_y f(y_1 , x)   + \frac{\alpha( x )}{2} \|y_2- y_1\|_2^2,\;\forall x \in X,\,\forall \, y_1, y_2 \in Y.
$$
\end{itemize}
We will also need the following assumption, used to control the error on the gradients of $f$:
\begin{itemize}
\item[(H4)] For every $y \in Y$ the function $f(y, \cdot)$ is differentiable on $X$ and for every $x \in X$ there exists $0 \leq M_1(x)<+\infty$ such that
for every
$y_1, y_2 \in Y$, we have
$$
\|\nabla_x f(y_2, x) -  \nabla_x f(y_1, x)  \|_2  \leq  M_1(x) \|y_2   - y_1\|_2. 
$$
\end{itemize}
\begin{prop}\label{fixedprop3} Let $\bar x \in X$ and
let $\hat y( \varepsilon ) \in Y$ be an $\epsilon$-optimal solution for problem \eqref{vfunction1}
written for $x= \bar x$ with optimal value $\mathcal{Q}( \bar x )$, i.e., $\mathcal{Q}( \bar x ) \geq f( \hat y ( \varepsilon ) , \bar x ) - \varepsilon$.
Let Assumptions (H3) and (H4) hold. Then setting
\begin{equation}\label{defeta}
\eta(\varepsilon, \bar x) = \varepsilon   +   M_1( \bar x ) \emph{Diam}(X) \sqrt{ \frac{2 \varepsilon}{\alpha(\bar x)}  },
\end{equation}
the affine function
\begin{equation}\label{definexactcut1}
\mathcal{C}(x):= f (\hat  y ( \varepsilon ) ,  \bar x  ) - \eta(\varepsilon, \bar x)  + \langle  \nabla_x f( \hat  y ( \varepsilon ), \bar x  ) , x - \bar x \rangle
\end{equation}
is a cut for $\mathcal{Q}$ at $\bar x$, i.e., for every $x \in X$ we have
$\mathcal{Q}(x) \geq \mathcal{C}(x)$ and the distance 
$\mathcal{Q}( \bar x ) - \mathcal{C}( \bar x)$ between the 
values of $\mathcal{Q}$ and of the cut at $\bar x$ is at most
$ \eta(\varepsilon , \bar x) $, or, equivalently,
$\nabla_x f( \hat  y ,   \bar x  ) \in \partial_{ \eta(\varepsilon, \bar x) } \mathcal{Q}( \bar x )$.
\end{prop}
\begin{proof} For short, we use the notation $\hat y$ instead of $ \hat y ( \varepsilon )$.
Using the fact that $\hat y \in Y$, the first order optimality conditions for $\bar y$ imply
$(\hat y - \bar y)^T \nabla_y f ( \bar y , \bar x ) \geq 0$, which combined with Assumption (H3), gives
$$
\begin{array}{lll}
f( \hat  y , \bar x ) & \geq &  f( \bar y , \bar x ) + (\hat y - \bar y )^T \nabla_y f ( \bar y , \bar x ) + \frac{\alpha(\bar x)}{2} \|\hat y - \bar y\|_2^2\\
& \geq & \mathcal{Q}( \bar x ) + \frac{\alpha(\bar x)}{2} \|\hat y - \bar y\|_2^2,
\end{array}
$$
yielding
\begin{equation}\label{firstrelationstab}
\|\bar y - \hat y\|_2 \leq \sqrt{   \frac{2}{\alpha(\bar x)}\Big( f(\hat y , \bar x) - \mathcal{Q}( \bar x )  \Big) } \leq  \sqrt{ \frac{2 \varepsilon}{\alpha(\bar x)}}.  
\end{equation}
Now recalling that $\nabla_x f(\bar y ,  \bar x  ) \in \partial \mathcal{Q}(  \bar  x)$, we have
for every $x \in X$,
\begin{equation}\label{firstineqprop23}
\begin{array}{lcl}
\mathcal{Q}(  x)  &  \geq &   \mathcal{Q}(  \bar  x) + ( x - \bar x )^T  \nabla_x f(\bar y ,  \bar x  )  \\
 &  \geq & f( \hat y , \bar x )  - \varepsilon +    ( x - \bar x )^T \nabla_x f(\bar y ,  \bar x  )  \\
 & = & f( \hat y , \bar x )  - \varepsilon + ( x - \bar x )^T  \nabla_x f(\hat  y ,  \bar x )    +  ( x - \bar x )^T \Big(   \nabla_x f(\bar  y ,  \bar x  ) - \nabla_x f(\hat  y ,  \bar x  ) \Big)  \\
 & \geq  & f( \hat y , \bar x )  - \varepsilon + (  x - \bar x )^T  \nabla_x f(\hat  y ,  \bar x  ) -M_1(\bar x) \|\hat y  - \bar y\|_2  \|x  - \bar x\|_2   \\
 & \stackrel{\eqref{firstrelationstab}}{\geq} & f (\hat  y ,  \bar x  ) - \varepsilon   -   M_1( \bar x ) \mbox{Diam}(X) \sqrt{ \frac{2 \varepsilon}{\alpha(\bar x)}  } +   (  x - \bar x  )^T  \nabla_x f(\hat  y ,  \bar x  ),
\end{array}
\end{equation}
where for the third inequality we have used Cauchy-Schwartz inequality and Assumption (H4).
 Finally, observe that $\mathcal{C}( \bar x ) = f (\hat  y ,  \bar x  ) - \eta(\varepsilon, \bar x)  \geq \mathcal{Q}( \bar x ) - \eta(\varepsilon, \bar x) $.\hfill
\end{proof}
\begin{rem} As expected, if $\varepsilon =0$ then $\eta(\varepsilon,\bar x)=0$ and the cut given by Proposition \ref{fixedprop3} is exact. Otherwise it is inexact. 
The error term $\eta(\varepsilon,\bar x)$ is the sum of the upper bound $\varepsilon$ on the error on the optimal value and of the  error term $M_1( \bar x) \emph{Diam}(X) \sqrt{ \frac{2 \varepsilon}{\alpha(\bar x)}  }$
which accounts for the error on the subgradients of $\mathcal{Q}$.
\end{rem}

\subsection{Inexact cuts with variable feasible set}

For $x \in X$, recall that for problem \eqref{vfunctionq} the Lagrangian
function is
$$
L_{x}(y, \lambda, \mu)=f(y,x) + \lambda^T (Bx+Ay-b) + \mu^T g(y,x),
$$ 
and
the 
dual function is given by
\begin{equation}\label{dualfunction}
\theta_{x}(\lambda, \mu)=\displaystyle \inf_{y \in Y} \;L_{x}(y, \lambda, \mu).
\end{equation}
Define
$\ell_2 :Y \small{\times} X  \small{\times}  \mathbb{R}^q      \small{\times} \mathbb{R}_{+}^p \rightarrow \mathbb{R}_{+}$ by
\begin{equation}\label{defrxy}
\ell_2 (\hat y , \bar x , \hat \lambda , \hat \mu ) = - \min_{y \in Y} \langle \nabla_y L_{\bar x} ( \hat y , \hat \lambda , \hat \mu ) , y - \hat y  \rangle = \max_{y \in Y} 
\langle  \nabla_y L_{\bar x} ( \hat y , \hat \lambda , \hat \mu ) , \hat y - y  \rangle.
\end{equation}
We make the following assumption which ensures no duality gap for \eqref{vfunctionq} for any $x \in X$:
\begin{itemize}
\item[(H5)]  if $Y$ is polyhedral then for every $x \in X$ there exists $y_x \in Y$ such that
$Bx+Ay_x=b$ and $g(y_x , x)< 0$ and if $Y$ is not polyhedral then 
for every $x \in X$ there exists $y_x \in \mbox{ri}(Y)$ such that
$Bx+Ay_x=b$ and $g(y_x , x)< 0$.
\end{itemize}
The following proposition, proved in \cite{guigues2016isddp}, provides an inexact cut for $\mathcal{Q}$ given by \eqref{vfunctionq}:
\begin{prop} \label{varprop1} [Proposition 2.7 in \cite{guigues2016isddp}]  Let $\bar x \in X$,
let $\hat y(\epsilon)$ be an $\epsilon$-optimal feasible primal solution for problem \eqref{vfunctionq}
written for $x= \bar x$ 
and let $(\hat \lambda(\epsilon), \hat \mu(\epsilon))$ be an $\epsilon$-optimal feasible solution of the
corresponding dual problem, i.e., of problem
\eqref{dualpb} written for $x=\bar x$.
Let Assumptions (H1), (H2), and (H5) hold. If additionally $f$ and $g$ are differentiable on $Y \small{\times} X$ then
setting $\eta(\varepsilon,\bar x)=\ell_2 (\hat y(\epsilon) , \bar x  , \hat \lambda(\epsilon) , \hat \mu(\epsilon) )$, the affine function
\begin{equation}\label{cutvarprop1}
\mathcal{C}(x):= L_{\bar x} ( \hat y(\epsilon), {\hat \lambda}(\epsilon), \hat \mu(\epsilon) )- \eta(\varepsilon,\bar x)  + 
\langle \nabla_x L_{\bar x} ( \hat y(\epsilon), {\hat \lambda}(\epsilon), \hat \mu(\epsilon) ) , x - \bar x \rangle
\end{equation}
with 
$$
\nabla_x L_{\bar x} ( \hat y(\epsilon), {\hat \lambda}(\epsilon), \hat \mu(\epsilon) )
=\nabla_x f(\hat y(\epsilon), \bar x)+ B^T \hat \lambda(\epsilon) +  \displaystyle \sum_{i=1}^p\; \hat \mu_i(\epsilon) \nabla_x g_i(\hat y(\epsilon), \bar x),
$$
is a cut for $\mathcal{Q}$ at $\bar x$ and the distance 
$\mathcal{Q}( \bar x ) - \mathcal{C}( \bar x)$ between the 
values of $\mathcal{Q}$ and of the cut at $\bar x$ is at most
$\varepsilon + \ell_2 (\hat y(\epsilon),  \bar x  , \hat \lambda(\epsilon) , \hat \mu(\epsilon) )$.
\end{prop}
In Proposition \ref{defcutctk} below, we derive another formula for inexact cuts with an additional assumption of strong convexity:
\begin{itemize}
\item[(H6)] Strong concavity of the dual function: for every $x \in X$ there exists $\alpha_D(x)>0$ and a set $D_x$
containing the set of optimal solutions of dual problem \eqref{dualpb}
such that the dual function $\theta_x$ is strongly concave on $D_x$ with constant of strong concavity $\alpha_D(x)$ 
with respect to $\|\cdot\|_2$.
\end{itemize}

We refer to Section \ref{strongconvdfunction} for conditions on the problem data ensuring Assumption (H6).

If the constants $\alpha(\bar x)$ and $\alpha_D(\bar x)$ in Assumptions (H3) and (H6) are sufficiently large and $n$ is small then the cuts
given by Proposition \ref{defcutctk} are better than the cuts given by Proposition \ref{varprop1}, i.e., $\mathcal{Q}( \bar x) - \mathcal{C}( \bar x)$ is smaller.
We refer to Section \ref{testinginexactcuts} for numerical tests comparing the cuts given by Propositions \ref{varprop1} and
\ref{defcutctk} on quadratic programs.

To proceed, take an optimal primal solution $\bar y$ of
 problem \eqref{vfunctionq} written for $x=\bar x$ and an optimal
 dual solution $(\bar \lambda, \bar \mu)$ of the corresponding dual problem, i.e.,
 problem \eqref{dualpb} written for $x=\bar x$.

With this notation, using Proposition \ref{dervaluefunction}, we have that
$\nabla_x L_{\bar x} ( \bar y, \bar \lambda, \bar \mu  ) \in \partial \mathcal{Q}( \bar x )$. 
Since we only have approximate
primal and dual solutions, $\hat y(\varepsilon)$ and $({\hat \lambda}(\epsilon), \hat \mu(\epsilon) )$ respectively,
we will use the approximate subgradient
$\nabla_x L_{\bar x} ( \hat y(\epsilon), {\hat \lambda}(\epsilon), \hat \mu(\epsilon) )$ instead of $\nabla_x L_{\bar x} ( \bar y, \bar \lambda, \bar \mu  ) $.
To control the error on this subgradient, we assume differentiability of the constraint functions
and that the gradients of these functions are Lipschitz continuous. More precisely, we assume:
\begin{itemize} 
\item[(H7)] $g$ is differentiable on $Y\small{\times}X$  and
 for every $x \in X$ there exists $0 \leq M_2(x) < +\infty$ such that for all $y_1, y_2 \in Y$, we have
$$
\|\nabla_{x} g_{i}(y_1, x) - \nabla_{x} g_{i}(y_2, x)\|_2  \leq M_2( x )  \|y_1 - y_2\|_2,\,i=1,\ldots,p.
$$
\end{itemize}
If Assumptions (H1)-(H7) hold, the following proposition provides an inexact cut 
for $\mathcal{Q}$ at $\bar x$:
\begin{prop}\label{defcutctk} 
Let $\bar x \in X$,
let $\hat y( \varepsilon )$ be an $\epsilon$-optimal feasible primal solution for problem \eqref{vfunctionq}
written for $x= \bar x$ 
and let $( {\hat \lambda}(\epsilon), \hat \mu(\epsilon) )$ be an $\epsilon$-optimal feasible solution of the
corresponding dual problem, i.e., of problem
\eqref{dualpb} written for $x=\bar x$.
Let Assumptions (H1), (H2), (H3), (H4), (H5), (H6), and (H7) hold.
Assume that $( {\hat \lambda}(\epsilon), \hat \mu(\epsilon) ) \in D_{\bar x}$ where $D_{\bar x}$ is defined in (H6) and
let 
\begin{equation}\label{ubgrad}
U = \max_{i=1,\ldots,p} \|\nabla_{x} g_{i}(\hat y(  \varepsilon  ), \bar x)\|_2.
\end{equation}
Let also $\mathcal{L}_{\bar x}$ be any lower bound on $\mathcal{Q}( \bar x) $.
Define
\begin{equation}\label{upperboundmult}
\mathcal{U}_{\bar x}= 
\frac{f(y_{\bar x} , \bar x  )  - \mathcal{L}_{\bar x}  }{\min (-g_{i}(y_{\bar x} ,  \bar x), i=1,\ldots,p )}
\end{equation}
and
$$
\eta(\varepsilon, \bar x)=\varepsilon + \Big( (M_1(  \bar x ) +  M_2(  \bar x ) \mathcal{U}_{\bar x} ) \sqrt{\frac{2}{\alpha( \bar x )}} + \frac{2 \max(   \|B^T\|, \sqrt{p} U) }{\sqrt{\alpha_D(\bar x) }} \Big) \emph{Diam}(X) \sqrt{\varepsilon} .
$$
Then 
$$
\mathcal{C}(x):= f(\hat y( \varepsilon  ) , \bar x) -\eta(\epsilon, \bar x) + \langle \nabla_x L_{\bar x} ( \hat y(\epsilon), {\hat \lambda}(\epsilon), \hat \mu(\epsilon) ) , x - \bar x \rangle
$$
where 
$$
\nabla_x L_{\bar x} ( \hat y(\epsilon), {\hat \lambda}(\epsilon), \hat \mu(\epsilon) )
=\nabla_x f(\hat y(\epsilon), \bar x)+ B^T \hat \lambda(\epsilon) +  \displaystyle \sum_{i=1}^p\; \hat \mu_i(\epsilon) \nabla_x g_i(\hat y(\epsilon), \bar x),
$$
is a cut for $\mathcal{Q}$ at $\bar x$ and the distance 
$\mathcal{Q}( \bar x ) - \mathcal{C}( \bar x)$ between the 
values of $\mathcal{Q}$ and of the cut at $\bar x$ is at most
$\eta(\epsilon, \bar x)$.
\end{prop}

\begin{proof} 
For short, we use the notation $\hat y, \hat \lambda, \hat \mu$ instead of 
$\hat y ( \varepsilon ), \hat \lambda ( \varepsilon ), \hat \mu ( \varepsilon )$.
Since $\nabla_x L_{\bar x} ( \bar y , \bar  \lambda , \bar \mu ) \in \partial \mathcal{Q}( \bar x )$,
we have
\begin{equation}\label{qtcut}
\begin{array}{lcl}
\mathcal{Q}( x ) & \geq &  \mathcal{Q}( \bar x  ) +  \langle \nabla_x L_{\bar x} ( \bar y , \bar  \lambda , \bar \mu ) , x -  \bar x \rangle \geq   f(\hat y , \bar x  ) - \varepsilon +   \langle  \nabla_x L_{\bar x} ( \bar y , \bar  \lambda , \bar \mu ), x -  \bar x \rangle. 
\end{array}
\end{equation}
Next observe that
{\small{
\begin{eqnarray}
\| \nabla_x L_{\bar x} ( \bar y , \bar  \lambda , \bar \mu ) - \nabla_x L_{\bar x} ( \hat y, {\hat \lambda}, \hat \mu ) \| & \leq  & 
M_1(  \bar x ) \| \bar y - \hat y\| + \|B^T\|  \|\bar \lambda  - \hat \lambda\| \nonumber\\
&& + \|\displaystyle \sum_{i=1}^{p}\; {\bar \mu}(i) \Big(\nabla_{x} g_{i}(\bar y , {\bar x}) - \nabla_{x} g_{i}(\hat y, {\bar x})   \Big)  \| \nonumber\\
&& + \|\displaystyle \sum_{i=1}^{p}\; \Big({\bar \mu}(i)- {\hat \mu}(i)\Big) \nabla_{x} g_{i}(\hat y , \bar x) \| \nonumber\\
& \leq & M_1(  \bar x ) \|\bar y - \hat y\|+ \|B^T\|  \|\bar \lambda  - \hat \lambda\| + M_2( \bar x) \|\bar \mu\|_1 \|\bar y - \hat y\| + U \sqrt{p} \|\bar \mu - \hat \mu\| \nonumber\\
& \leq & (M_1(  \bar x )  + M_2( \bar x) \| \bar \mu\|_1 ) \|\bar y - \hat y\| +  \sqrt{2}\max( \|B^T\|, U \sqrt{p} ) \sqrt{\|\hat \lambda - \bar \lambda\|^2  + \|\hat \mu - \bar \mu\|^2}.\label{upperbetadiff}
\end{eqnarray}
}}
Using Remark 2.3.3, p.313 in \cite{hhlem} and Assumption (H5) we have for $\|\bar \mu\|_1$ the upper bound
\begin{equation}\label{boundmult}
\|\bar \mu\|_1 \leq \frac{f(y_{\bar x} , \bar x  )  - \mathcal{Q}({\bar x} ) }{\min (-g_{i}(y_{\bar x} ,  \bar x), i=1,\ldots,p )} \leq \mathcal{U}_{\bar x}.
\end{equation} 
Using Assumptions (H3) and (H6), we also get 
\begin{equation}\label{strongoptpd}
\|\hat y - \bar y\|^2 \leq \frac{2 \varepsilon}{\alpha(\bar x)} \mbox{ and } \|\hat \lambda - \bar \lambda\|^2  + \|\hat \mu - \bar \mu\|^2 \leq \frac{2 \varepsilon}{\alpha_D( \bar x )}.
\end{equation}
Combining \eqref{upperbetadiff}, \eqref{boundmult}, and \eqref{strongoptpd}, we get
\begin{eqnarray}\label{finalbornebeta}
\| \nabla_x L_{\bar x} ( \bar y , \bar  \lambda , \bar \mu ) - \nabla_x L_{\bar x} ( \hat y, {\hat \lambda}, \hat \mu ) \|  \leq  \frac{\eta(\varepsilon,\bar x) - \varepsilon}{\mbox{Diam}(X)}.  
\end{eqnarray}
Plugging the above relation into \eqref{qtcut} and using Cauchy-Schwartz inequality, we get 
\begin{equation}\label{cscwhfinalcutvarset}
\begin{array}{lll}
\mathcal{Q}( x ) & \geq &  f(\hat y,  \bar x  ) - \varepsilon +   \langle \nabla_x L_{\bar x} ( \hat y , {\hat \lambda}, \hat \mu ) , x -  \bar x \rangle  +   
\langle \nabla_x L_{\bar x} ( \bar y , \bar  \lambda , \bar \mu ) - \nabla_x L_{\bar x} ( \hat y, {\hat \lambda}, \hat \mu ), x -  \bar x \rangle \\
 & \geq & f(\hat y , \bar x  ) - \varepsilon - \|\nabla_x L_{\bar x} ( \hat y,  {\hat \lambda}, \hat \mu ) - \nabla_x L_{\bar x} ( \bar y , {\bar \lambda}, \bar \mu )\| \mbox{Diam}(X) + \langle \nabla_x L_{\bar x} ( \hat y ,{\hat \lambda}, \hat \mu ), x -  \bar x \rangle\\
 & \geq & f(\hat y, \bar x) -\eta(\varepsilon, \bar x) + \langle  \nabla_x L_{\bar x} ( \hat y , {\hat \lambda}, \hat \mu ), x - \bar x \rangle.
\end{array}
\end{equation}
Finally, since $\hat y \in \mathcal{S}( \bar x)$ we check that
$
\mathcal{Q}( \bar x ) - \mathcal{C}( \bar x ) = \mathcal{Q}( \bar x ) - f(\hat y, \bar x)  + \eta(\varepsilon, \bar x) \leq \eta(\varepsilon, \bar x),
$
which achieves the proof of the proposition.
\hfill
\end{proof}
Observe that the ``slope" $\nabla_x L_{\bar x} ( \hat y(\epsilon), {\hat \lambda}(\epsilon), \hat \mu(\epsilon) )$
of the cut given by Proposition \ref{varprop1} is the same as the ``slope" 
of the cut given by Proposition \ref{defcutctk}.
\begin{rem} If $\hat y( \varepsilon )$ and $(\hat \lambda( \varepsilon  ), \hat \mu(  \varepsilon ))$
are respectively optimal primal and dual solutions, i.e., $\varepsilon=0$, then Proposition \ref{defcutctk} gives, as expected, an exact cut for 
$\mathcal{Q}$ at $\bar x$.
\end{rem}

As shown in Corollary \ref{corinexactcuts}, the formula for the inexact cuts given in Proposition \ref{defcutctk} can be simplified depending if  
there are nonlinear coupling constraints or not, if $f$ is separable (sum of a function of $x$ and of a function of $y$) or not,
and if $g$ is separable.
\begin{cor}\label{corinexactcuts} Consider the value functions $\mathcal{Q}: X \rightarrow \mathbb{R}$ where $\mathcal{Q}(x)$
is given by the optimal value of the following optimization problems:
\begin{equation}
\begin{array}{lll}
(a)\left\{ 
\begin{array}{l}
\min_{y } f(y,x)\\
Ay + B x=b,\\h(y)+k(x)\leq 0,\\
y \in Y,
\end{array}
\right.
&
(b)\left\{ 
\begin{array}{l}
\min_{y } f_0(y)+f_1(x)\\
Ay + B x=b,\\g(y,x)\leq 0,\\
y \in Y,
\end{array}
\right.
&
(c)\left\{ 
\begin{array}{l}
\min_{y } f_0(y)+f_1(x)\\
Ay + B x=b,\\h(y)+k(x)\leq 0,\\
y \in Y,
\end{array}
\right.\\
(d)\left\{ 
\begin{array}{l}
\min_{y }  f(y,x)\\
g(y,x)\leq 0,\\
y \in Y,
\end{array}
\right.
&
(e)\left\{ 
\begin{array}{l}
\min_{y } f(y,x)\\
h(y)+k(x)\leq 0,\\
y \in Y,
\end{array}
\right.
&
(f)\left\{ 
\begin{array}{l}
\min_{y } f_0(y)+f_1(x)\\
g(y,x)\leq 0,\\
y \in Y,
\end{array}
\right.\\
(g)\left\{ 
\begin{array}{l}
\min_{y } f_0(y)+f_1(x)\\
h(y)+k(x)\leq 0,\\
y \in Y,
\end{array}
\right.
& 
(h)\left\{ 
\begin{array}{l}
\min_{y } f(y,x)\\
A y +B x = b,\\
y \in Y,
\end{array}
\right.
& 
(i)\left\{ 
\begin{array}{l}
\min_{y } f_0(y)+f_1(x)\\
A y +B x = b,\\
y \in Y.
\end{array}
\right.
\end{array}
\end{equation}
For problems (b),(c),(f),(g), (i) above define $f(y,x)=f_0(y)+f_1(x)$ and for problems (a), (c), (e), (g) 
define $g(y,x)=h(y)+k(x)$.
With this notation, assume that (H1), (H2), (H3), (H4), (H5), (H6), and (H7) hold for these problems.
If $g$ is defined, let $L_{x}(y, \lambda, \mu)=f(y,x) + \lambda^T (Bx+Ay-b) + \mu^T g(y,x)$ be the Lagrangian 
and define 
$$
U = \max_{i=1,\ldots,p} \|\nabla_{x} g_{i}(\hat y(  \varepsilon  ), \bar x)\| \mbox{ and }\mathcal{U}_{\bar x}= 
\frac{f(y_{\bar x} , \bar x  )  - \mathcal{L}_{\bar x}  }{\min (-g_{i}(y_{\bar x} ,  \bar x), i=1,\ldots,p )}
$$
where  $\mathcal{L}_{\bar x}$ is any lower bound on $\mathcal{Q}( \bar x) $.
If $g$ is not defined, define $L_{x}(y, \lambda)=f(y,x) + \lambda^T (Bx+Ay-b)$.

Let $\bar x \in X$,
let $\hat y$ be an $\epsilon$-optimal feasible primal solution for problem \eqref{vfunctionq}
written for $x= \bar x$ 
and let $( {\hat \lambda}, \hat \mu )$ be an $\epsilon$-optimal feasible solution of the
corresponding dual problem, i.e., of problem
\eqref{dualpb} written for $x=\bar x$.

Then 
$
\mathcal{C}(x)=f(\hat y, \bar x)-\eta(\varepsilon,\bar x)+ \langle s(\bar x) , x - \bar x \rangle
$
is an inexact cut for $\mathcal{Q}$ at $\bar x$ where  the formulas
for $\eta(\varepsilon,\bar x)$ and $s(\bar x)$ in each of cases (a)-(i) above are the following:
\begin{equation}
\begin{array}{l}
(a)\left\{
\begin{array}{ll}
\eta(\varepsilon, \bar x)=\varepsilon + \Big( M_1(  \bar x ) \frac{1}{\sqrt{\alpha(\bar x)}}  +   \sqrt{2} \max(   \|B^T\|, \sqrt{p} U)  \frac{1}{\sqrt{\alpha_D( \bar x ) }} \Big) \emph{Diam}(X) \sqrt{2 \varepsilon},\\
s(\bar x)=\nabla_x f(\hat y, \bar x ) + B^T \hat \lambda + \sum_{i=1}^p \hat \mu_i \nabla_x k_i( \bar x ),
\end{array}
\right.\\
(b)\left\{
\begin{array}{ll}
\eta(\varepsilon, \bar x)=\varepsilon + \Big( M_2( \bar x) \mathcal{U}_{\bar x} \frac{1}{\sqrt{\alpha(\bar x)}} + \sqrt{2} \max(   \|B^T\|, \sqrt{p} U )  \frac{1}{\sqrt{\alpha_D( \bar x ) }} \Big) \emph{Diam}(X) \sqrt{2 \varepsilon},\\
s(\bar x)=\nabla_x f_1(\bar x ) + B^T \hat \lambda + \sum_{i=1}^p \hat \mu_i \nabla_x g_i(\hat y,  \bar x ),
\end{array}
\right.\\
(c)\left\{
\begin{array}{ll}
\eta(\varepsilon, \bar x)=\varepsilon + 2 \max(   \|B^T\|, \sqrt{p} U)\emph{Diam}(X) \sqrt{\frac{\varepsilon}{ \alpha_D( \bar x )}},\\
s(\bar x)=\nabla_x f_1(\bar x ) + B^T \hat \lambda + \sum_{i=1}^p \hat \mu_i \nabla_x k_i( \bar x ),
\end{array}
\right.\\
(d)\left\{
\begin{array}{ll}
\eta(\varepsilon, \bar x)=\varepsilon + \Big( ( M_1(  \bar x )  + M_2( \bar x) \mathcal{U}_{\bar x} )  \frac{1}{\sqrt{\alpha(\bar x)}}  +  U  \sqrt{\frac{p}{\alpha_D( \bar x )}} \Big) \emph{Diam}(X) \sqrt{2 \varepsilon},\\
s(\bar x)=\nabla_x f(\hat y, \bar x ) + \sum_{i=1}^p \hat \mu_i \nabla_x g_i(\hat y ,  \bar x ),
\end{array}
\right.\\
(e)\left\{
\begin{array}{ll}
\eta(\varepsilon, \bar x)=\varepsilon + \Big( \frac{M_1(  \bar x )}{\sqrt{\alpha(\bar x)}} + \sqrt{  \frac{p}{ \alpha_D(\bar x) }} U  \Big) \emph{Diam}(X) \sqrt{2 \varepsilon},\\
s(\bar x)=\nabla_x f(\hat y, \bar x ) +  \sum_{i=1}^p \hat \mu_i \nabla_x k_i( \bar x ),
\end{array}
\right.\\
(f)\left\{
\begin{array}{ll}
\eta(\varepsilon, \bar x)=\varepsilon + \Big( \frac{ M_2( \bar x)}{\sqrt{\alpha(\bar x)}} \mathcal{U}_{\bar x} +  U \sqrt{   \frac{p}{ \alpha_D(\bar x) }} \Big) \emph{Diam}(X) \sqrt{2 \varepsilon},\\
s(\bar x)=\nabla_x f_1(\bar x )  + \sum_{i=1}^p \hat \mu_i \nabla_x g_i(\hat y ,  \bar x ),
\end{array}
\right.\\
(g)\left\{
\begin{array}{ll}
\eta(\varepsilon, \bar x)=\varepsilon +  \emph{Diam}(X) \sqrt{\frac{2 \varepsilon p}{\alpha_D( \bar x )}}U,\\
s(\bar x)=\nabla_x f_1(\bar x ) + \sum_{i=1}^p \hat \mu_i \nabla_x k_i( \bar x ),
\end{array}
\right.\\
(h)\left\{
\begin{array}{ll}
\eta(\varepsilon, \bar x)=\varepsilon + \Big( \frac{M_1(  \bar x )}{\sqrt{\alpha(\bar x)}} +  \frac{\|B^T \|}{ \sqrt{\alpha_D(\bar x) }}  \Big) \emph{Diam}(X) \sqrt{2 \varepsilon},\\
s(\bar x)=\nabla_x f(\hat y, \bar x ) + B^T \hat \lambda,
\end{array}
\right.\\
(i)\left\{
\begin{array}{ll}
\eta(\varepsilon, \bar x)=\varepsilon +  \|B^T \| \sqrt{ \frac{2 \varepsilon}{\alpha_D( \bar x )}}   \emph{Diam}(X),\\
s(\bar x)=\nabla_x f_1(\bar x ) + B^T \hat \lambda.
\end{array}
\right.
\end{array}
\end{equation}
\end{cor}
\begin{proof} It suffices to follow the proof of Proposition \ref{defcutctk}, specialized to cases (a)-(i). For instance, let us check the formulas in case (g).
For (g), 
$
s(\bar x)=\nabla_x L_{\bar x}(\hat y,\hat \mu)=\nabla_x f_1(\bar x ) + \sum_{i=1}^p \hat \mu_i \nabla_x k_i( \bar x )
$
and
\begin{equation}\label{caseccutinexact}
\begin{array}{lll}
\|\nabla_x L_{\bar x}(\hat y,  \hat \mu) - \nabla_x L_{\bar x}(\bar y, \bar \mu)  \| 
& = &  \| \sum_{i=1}^p (\hat \mu_i - \bar \mu_i) \nabla_x k_i( \bar x )  \| \leq U \| \hat \mu - \bar \mu   \|_1 \\
& \leq & U \sqrt{p}\|\hat \mu - \bar \mu\| \leq U \sqrt{p} \sqrt{\frac{2 \varepsilon}{\alpha_D( \bar x )}}.
\end{array} 
\end{equation}
It then suffices to combine \eqref{qtcut} and \eqref{caseccutinexact}.\hfill
\end{proof}

\if{
As shown in Corollary \ref{corinexactcuts} in the Appendix, the formula for the inexact cuts given in Proposition \ref{defcutctk} can be simplified depending if  
there are nonlinear coupling constraints or not, if $f$ is separable (sum of a function of $x$ and of a function of $y$) or not,
and if $g$ is separable.
}\fi

\subsection{Numerical results} \label{testinginexactcuts}

\subsubsection{Argument of the value function in the objective only}

Let $S=\left(
\begin{array}{cc}
S_1 & S_2 \\
S_2^T & S_3
\end{array}
\right)$
be a positive definite  matrix, let $c_1 \in \mathbb{R}^m, c_2 \in \mathbb{R}^n$ be vectors of ones,
and let 
$\mathcal{Q}$ be the value function given by 
\begin{equation}\label{defqxsimple}
\begin{array}{lll}
\mathcal{Q}(x) &= &\left\{ 
\begin{array}{l}
\min_{y \in \mathbb{R}^n} \;f(y,x)=\frac{1}{2}\left(  \begin{array}{c}x\\y\end{array} \right)^T S \left(  \begin{array}{c}x\\y\end{array} \right)  + 
\left(  
\begin{array}{c}
c_1 \\
c_2
\end{array}
\right)^T \left(  \begin{array}{c}x\\y\end{array} \right)\\
y \in Y:=\{y \in \mathbb{R}^n : y \geq 0,\;\sum_{i=1}^n y_i = 1 \},
\end{array}
\right.\\
&=& 
\left\{ 
\begin{array}{l}
\min_{y \in \mathbb{R}^n} \;c_1^T x + c_2^T y + \frac{1}{2} x^T S_1 x + x^T S_2 y + \frac{1}{2} y^T S_3 y \\
y \geq 0,\;\sum_{i=1}^n y_i = 1.
\end{array}
\right.
\end{array}
\end{equation}
Clearly, Assumption (H3) is satisfied with $\alpha(x )=\lambda_{\min}( S_3 )$,
and
$$
\|\nabla_x f( y_2, x)  - \nabla_x f( y_1, x) \| = \|S_2 ( y_2 - y_1 ) \|_2 \leq \| S_2  \|_2  \| y_2   - y_1 \|_2
$$
implying that Assumption (H4) is satisfied with $M_1(  \bar x )=\|S_2\|_2 = \sigma( S_2 )$ where $\sigma( S_2 )$ is the largest singular value of $S_2$.
We take $X=Y$ with  $\mbox{Diam}(X)=\max_{x_1, x_2 \in X} \|x_2 - x_1\|_2 \leq \sqrt{2}$.
With this notation, if $\hat y$ is an $\epsilon$-optimal solution
of \eqref{defqxsimple} written for $x= \bar x$, we compute at $\bar x$ the cut $\mathcal{C}(x)=f(\hat y , \bar x ) - \eta(\varepsilon,\bar x)+ \langle \nabla_{x}f(\hat y , \bar x ), x - \bar x  \rangle 
= f(\hat y , \bar x ) - \eta(\varepsilon,\bar x)+ \langle c_1 + S_1 \bar x + S_2 {\hat y}, x - \bar x  \rangle$ where 
\begin{itemize}
\item $\eta(\varepsilon,\bar x)=\eta_1(\varepsilon,\bar x)=\varepsilon + 2 M_1(  \bar x )  \sqrt{\frac{\varepsilon}{\alpha(\bar x)}}$ using  Proposition \ref{fixedprop3};
\item $\eta(\varepsilon,\bar x)$ is given by  $$\eta(\varepsilon,\bar x)=\eta_2(\varepsilon,\bar x)=
\left\{ 
\begin{array}{l}
\max \; \langle \nabla_y f( \hat y, \bar x ), \hat y - y \rangle \\
y \geq 0,\;\sum_{i=1}^n y_i = 1,
\end{array}
\right.
=
\left\{ 
\begin{array}{l}
\max \; \langle c_2 + S_2^T \bar x + S_3 \hat y , \hat y - y \rangle \\
y \geq 0,\;\sum_{i=1}^n y_i = 1,
\end{array}
\right.
$$
using Proposition \ref{fixedprop1}.
\end{itemize}

We compare in Table \ref{inexactcutsobjective} the values of $\eta_1(\varepsilon,\bar x)$ and  $\eta_2(\varepsilon,\bar x)$
for several values of $m=n$, $\varepsilon$, and $\alpha(\bar x)$. In these experiments
$S$ is of the form $A A^T + \lambda I_{2 n}$ for some $\lambda>0$ and $A$ has random entries in $[-20,20]$.

Optimization problems were solved using Mosek optimization toolbox \cite{mosek},
setting Mosek parameter MSK\_\\DPAR\_INTPNT\_QO\_TOL \_REL\_GAP which corresponds to the relative
error $\varepsilon_r$  on the optimal value to 0.1, 0.5, and 1. In each run, $\varepsilon$ was estimated computing the duality
gap (the difference between the approximate optimal values of the dual and the primal).
Though $\eta_1(\varepsilon,\bar x)$ does not depend on $\bar x$ (because on this example $\alpha$ and $M_1$
do not depend on $\bar x$), the absolute error $\varepsilon$
depends on the run (for a fixed $\varepsilon_r$, different runs corresponding to different $\bar x$
yield different errors $\varepsilon, \eta_1(\varepsilon,\bar x)$ and $\eta_2(\varepsilon,\bar x)$).
Therefore, for each fixed $(\varepsilon_r, \alpha(\bar x), n)$,
the values $\varepsilon,$ $\eta_1(\varepsilon,\bar x)$, and $\eta_2(\varepsilon,\bar x)$ 
reported in the table
correspond to the mean values of $\varepsilon$,  $\eta_1(\varepsilon,\bar x)$, and $\eta_2(\varepsilon,\bar x)$ obtained taking
randomly 50 points in $X$. We see that the cuts computed by Proposition \ref{fixedprop3} are much more conservative on nearly all 
combinations of parameters, except on three of these combinations when $n=10$ and $\alpha(\bar x)=10^6$ is very large.

\begin{table}
\centering
\begin{tabular}{|c|c|c|c|c||c|c|c|c|c|}
\hline
$\varepsilon$  & $\alpha(\bar x)$  &  $n$  &   $\eta_1$ &  $\eta_2$ & $\varepsilon$  & $\alpha(\bar x)$  &  $n$  &   $\eta_1$ &  $\eta_2$ \\
\hline
 0.0024 & 102.9         &  10   &  1.76   &         0.025  & 0.0061  &   190.2       &  10   &  2.73   &         0.026  \\
\hline
0.0080  &    10 087      &  10   &  0.86   &         0.054  & 0.0024  &     $10^6$     &  10   &  0.076   &         0.354  \\
\hline
 0.016 &   129.0       &  10   &  9.81   &         0.047  &  0.0084 &  174.5        &  10   &  4.85   &         0.037  \\
\hline
 0.029 &   10054       &  10   &  2.49   &         0.128  & 0.002  &   $10^6$       &  10   &  0.09   &         0.342  \\
\hline
0.008  &   112.3       &  10   &  8.07   &         0.043  & 0.008  &  150.0        &  10   &  6.36   &         0.022  \\
\hline
0.018  &   10 090       &  10   &  1.29   &         0.078  & 0.0019  &  $10^6$        &  10   &  0.06   &         0.442  \\
\hline
\hline
0.15  &   531.9       &  100   &   175.6 &   0.3        & 0.18  &      665.3    &  100   &  183.5   &  0.3        \\
\hline
0.23 &    10 687 &  100   &  44.5  &    0.2       & 0.03  &   $10^6$       &  100   & 2.1    &   0.9       \\
\hline
0.17 &   676.2       &  100   &  185.7  &   0.2        & 0.09  &  734.3       &  100   & 106.5   &  0.2        \\
\hline
0.11 &    10 638      &  100   & 37.9   &    0.2       & 0.02  &  $10^6$        &  100   & 1.7    &  0.3        \\
\hline
0.05 &    660      &  100   & 106.7   &     0.2      & 0.40  &  777        &  100   &   253.8 & 0.4         \\
\hline
0.07 &  10 585        &  100   &  32.6  &     0.2      & 0.02  & $10^6$         &  100   &  1.3   &  0.4        \\
\hline
\hline
6.78  &    6017.9      &  1000   & 4177.8   &   9.5        & 2.69  & 5991.4         &  1000   & 2778.8    &   6.8       \\
\hline
 8.12&    15 722      &  1000   &  3059.5  &     11.1      & 0.99  &  $10^6$        &  1000   & 132.1    &   3.2      \\
\hline
7.40 &    5799      &  1000   & 4160.2   &    9.8       & 7.83  &  6020        &  1000   & 4590.7   &     9.3     \\
\hline
12.5 &   15860       &  1000   &  4001.6  &    14.6       & 1.3  &   $10^6$       &  1000   &  153.6   &   3.47      \\
\hline
 9.9&   6065       &  1000   &  4996.4  &    11.8      & 8.3  &  5955        &  1000   & 4034.9   &  8.3        \\
\hline
7.2 &   15 895       &  1000   &  2564.3  &    3.4       &  9.7 &  $10^6$       &  1000   & 117.2    &   1.8       \\
\hline
\end{tabular}
\caption{Values of $\eta(\varepsilon,\bar x)=\eta_1(\varepsilon,\bar x)$ (resp. $\eta(\varepsilon,\bar x)=\eta_2(\varepsilon,\bar x)$) for the inexact cuts 
given by Proposition \ref{fixedprop3} (resp. Proposition \ref{fixedprop1}) for value function 
\eqref{defqxsimple} for various values of $n$ (problem dimension), $\alpha(\bar x)=\lambda_{\min}(S_3)$, and
$\varepsilon$.}
\label{inexactcutsobjective}
\end{table}

\subsubsection{Argument of the value function in the objective and constraints}

We close this section comparing the error terms in the cuts given by Propositions \ref{varprop1} and \ref{defcutctk}
on a very simple problem with a quadratic objective and a quadratic constraint. 

Let $S=\left(
\begin{array}{cc}
S_1 & S_2 \\
S_2^T & S_3
\end{array}
\right)$
be a positive definite matrix, let $c_1, c_2 \in \mathbb{R}^n$,
and let  $\mathcal{Q}:X \rightarrow \mathbb{R}$ be the value function given by 
\begin{equation}\label{defqxsimple2}
\mathcal{Q}(x)  = 
\min_{y \in \mathbb{R}^n} \;\{ f(y,x) :  g_1(y,x) \leq 0 \},
\end{equation}
where
\begin{equation}\label{defqxsimple3}
\begin{array}{lll}
f(y,x)&=&\frac{1}{2}\left(  \begin{array}{c}x\\y\end{array} \right)^T S \left(  \begin{array}{c}x\\y\end{array} \right)  + 
\left(  
\begin{array}{c}
c_1 \\
c_2
\end{array}
\right)^T \left(  \begin{array}{c}x\\y\end{array} \right)\\
& = & c_1^T x + c_2^T y + \frac{1}{2} x^T S_1 x + x^T S_2 y + \frac{1}{2} y^T S_3 y, \\
g_1(y,x)&=&\frac{1}{2} \|y - y_0\|_2^2 + \frac{1}{2}\|x - x_0\|_2^2 - \frac{R^2}{2},\\
X&=&\{x \in \mathbb{R}^n : \|x-x_0\|_2 \leq 1\}.
\end{array}
\end{equation}
In what follows, we take $R=5$ and $x_0, y_0 \in \mathbb{R}^n$ given by 
$x_0(i)=y_0(i)=10, i=1,\ldots,n$.
Clearly, for fixed $\bar x \in X$ and any feasible $y$ for \eqref{defqxsimple2}, \eqref{defqxsimple3} written for $x=\bar x$, we have 
$$
\left\| \left( \begin{array}{c}x_0\\y_0\end{array} \right) \right\| + R \geq \left\| \left( \begin{array}{c}\bar x \\y\end{array} \right) \right\| \geq  \left\| \left( \begin{array}{c}x_0\\y_0\end{array} \right) \right\| - R.
$$
Knowing that with our problem data $\left\| \left( \begin{array}{c}x_0\\y_0\end{array} \right) \right\| - R>0$, we obtain the bound 
$
\mathcal{Q}(\bar x) \geq \mathcal{L}_{\bar x}$
where
$$ 
\mathcal{L}_{\bar x} =  \frac{1}{2} \lambda_{\min}( S ) \Big(  \left\| \left( \begin{array}{c}x_0\\y_0\end{array} \right) \right\|-R\Big)^2  - 
\Big( \left\| \left( \begin{array}{c}x_0\\y_0\end{array} \right) \right\| + R  \Big) \left\|\left( \begin{array}{c}c_1\\c_2\end{array} \right) \right\|_2.
$$
Next, for every $\bar x \in X$ we have  $g_1(y_0, \bar x)<0$ which gives the upper bound 
\begin{equation}\label{formulaupperboundmu}
\mathcal{U}_{\bar x} = \frac{\mathcal{L}_{\bar x}-f(y_0,\bar x)}{g_1(y_0,\bar x)}
\end{equation}
for any optimal dual solution $\bar \mu \geq 0$ of the dual of \eqref{defqxsimple2}, \eqref{defqxsimple3} written for $x=\bar x$.
Making the change of variable $z=y-y_0$, we can express \eqref{defqxsimple2} under the form  
\eqref{pbquadratic} where 
\begin{equation}\label{varreformulation}
\begin{array}{l}
Q_0=S_3, a_0=a_0(x)=c_2 + S_2^T x + S_3 y_0, b_0=b_0(x)=\frac{1}{2}x^T S_1 x + c_1^T x + y_0^T(c_2+S_2^T x) + \frac{1}{2}y_0^T S_3 y_0,\\ 
Q_1=I_n, a_1=0, b_1=b_1(x)=\frac{1}{2}(\|x - x_0\|_2^2 - R^2 ).
\end{array}
\end{equation}
Therefore, using Proposition \ref{strongconcquad}, we have that dual function $\theta_{\bar x}$ for \eqref{defqxsimple2} is given by
\begin{equation}\label{formuladualfunction}
\theta_{\bar x}(\mu)=-\frac{1}{2}a_0(\bar x)^T (S_3+\mu I_n )^{-1} a_0(\bar x) +b_0( \bar x ) + \mu b_1(\bar x )
\end{equation}
with $a_0, b_0, b_1$ given by \eqref{varreformulation}
and setting
$$
\alpha_D( \bar x )=a_0(\bar x)^T(S_3+ \mathcal{U}_{\bar x} I_n )^{-3} a_0(\bar x),
$$
if $a_0(\bar x) \neq 0$ then $\theta_{\bar x}$ is strongly concave 
on the interval $[0,\mathcal{U}_{\bar x}]$ with constant of strong concavity $\alpha_D( \bar x )$
where 
$\mathcal{U}_{\bar x}$ is given by \eqref{formulaupperboundmu}.
Let $\hat y$ be an $\varepsilon$-optimal primal solution of \eqref{defqxsimple2} written for $x = \bar x$ and let $\hat \mu$ be an $\varepsilon$-optimal solution of its dual.
If $a_0(\bar x) \neq 0$, 
we obtain for  $\mathcal{Q}$ the cut 
\begin{equation}\label{inexcutex21}
\left\{
\begin{array}{lll} 
\mathcal{C}_1(x)&=&f(\hat y, \bar x ) - \eta_1(\varepsilon, \bar x ) + \langle \nabla_x L_{\bar x}( \hat y, \hat \mu ) , x - \bar x \rangle \mbox{ where }\\ 
\eta_1(\varepsilon, \bar x ) & = & \varepsilon + D(X) \sqrt{2 \varepsilon}\Big( \frac{M_1(\bar x)}{\sqrt{ \alpha(\bar x)}} +   \frac{\|\bar x - x_0\|}{\sqrt{\alpha_D(\bar x)}}  \Big)\mbox{ with }D(X)=2, M_1(\bar x)=\|S_2\|_2, \alpha(\bar x)=\lambda_{\min}(S_3),\\
\nabla_x L_{\bar x}(\hat y, \hat \mu ) & = & S_1 \bar x + c_1 + S_2 \hat y + \hat \mu ( \bar x - x_0).
\end{array}
\right.
\end{equation}

We now apply Proposition \ref{varprop1} to obtain another inexact cut for  $\mathcal{Q}$ at $\bar x \in X$ rewriting 
\eqref{defqxsimple2} under the form \eqref{vfunctionq} with $Y$ the compact set $Y=\{y \in \mathbb{R}^n : \|y - y_0\|_2 \leq R \}$:
\begin{equation}\label{reformdefqxsimple}
\mathcal{Q}(x)  = 
\min_{y \in \mathbb{R}^n} \;\left\{ f(y,x) :  g_1(y,x) \leq 0, \|y-y_0\|_2 \leq R \right\}.
\end{equation}

Applying Proposition \ref{varprop1} to reformulation \eqref{reformdefqxsimple} of \eqref{defqxsimple2}, we obtain for $\mathcal{Q}$ the inexact cut $\mathcal{C}_2$  at $\bar x$ where
\begin{equation}\label{inexcutex22}
\left\{
\begin{array}{lll} 
\mathcal{C}_2(x)&=&f(\hat y, \bar x ) - \eta_2(\varepsilon, \bar x ) + \langle \nabla_x L_{\bar x}( \hat y, \hat \mu ) , x - \bar x \rangle \mbox{ with }\\ 
\eta_2(\varepsilon, \bar x ) & = & -\min\{\langle \nabla_y L_{\bar x}(\hat y , \hat \mu) , y - \hat y \rangle : \|y-y_0\|_2 \leq R \},\\
& =& \langle \nabla_y L_{\bar x}(\hat y , \hat \mu) , \hat y - y_0  \rangle  + R \| \nabla_y L_{\bar x}(\hat y , \hat \mu) \|_2,\\
\nabla_x L_{\bar x}(\hat y, \hat \mu ) & = & S_1 \bar x + c_1 + S_2 \hat y + \hat \mu ( \bar x - x_0),\\
\nabla_y L_{\bar x}(\hat y, \hat \mu ) & = & S_3 \hat y + S_2^T \bar x +   c_2 + \hat \mu ( \hat y  - y_0).
\end{array}
\right.
\end{equation}
\if{
\par {\textbf{Quadratic objective and linear constraint.} 
Consider the value function $\mathcal{Q}:X \rightarrow \mathbb{R}$ given by\footnote{In this very simple example,
we can compute directly $\theta_x(\lambda)=-\frac{n}{2} \lambda^2 +\lambda (-1+\sum_{i=1}^n x_i) + \frac{1}{2}\|x\|_2^2$, the optimal dual
solution $\bar \lambda = \frac{1}{n}(-1+\sum_{i=1}^n x_i)$, and the optimal primal solution $\bar y = \Big(\frac{1-\sum_{i=1}^n x_i}{n}\Big) e$
where e is a vector in $\mathbb{R}^n$ with all components equal to one.}
\begin{equation}\label{defQxex22}
\mathcal{Q}(x)= \min_{y \in \mathbb{R}^n} \left\{ \frac{1}{2}\|x\|_2^2 + \frac{1}{2}\|y\|_2^2 : \sum_{i=1}^n (x_i + y_i ) = 1 \right\}  
\end{equation}
where $X=\{x \in \mathbb{R}^n : x\geq 0 , \sum_{i=1}^n x_i \leq 1/2\}$ with $D(X)=\frac{1}{\sqrt{2}}$.
Using proposition , dual function $\theta_{\bar x}$ of \eqref{defQxex22} is strongly concave with constant
of strong concavity $\alpha_D( \bar x )=\lambda_{\min}(A A^T)$ where $A$  has one line with all components
equal to one. Therefore $\alpha_D( \bar x )=n$. Using Proposition we obtain the cut
\begin{equation}\label{inexcutex22f}
\left\{
\begin{array}{lll} 
\mathcal{C}_1(x)&=& \frac{1}{2}( \| \bar x \|_2^2  + \|\hat y\|_2^2 ) - \eta_1(\varepsilon, \bar x ) + \langle \nabla_x L_{\bar x}( \hat y, \hat \lambda ) , x - \bar x \rangle \mbox{ where }\\ 
\eta_1(\varepsilon, \bar x ) & = & \varepsilon + D(X) \sqrt{n} \sqrt{\frac{2 \varepsilon}{\alpha_D( \bar x )}} = \varepsilon + 2 \sqrt{2 \varepsilon},\\
\nabla_x L_{\bar x}(\hat y, \hat \lambda  ) & = & \bar x + \hat \lambda e \mbox{ with }e_i=1,i=1,\ldots,n.
\end{array}
\right.
\end{equation}
On the other hand, Proposition \ref{varprop1} gives the cut 
\begin{equation}\label{inexcutex234f}
\left\{
\begin{array}{lll} 
\mathcal{C}_2(x)&=& \frac{1}{2}( \| \bar x \|_2^2  + \|\hat y\|_2^2 ) - \eta_2(\varepsilon, \bar x ) + \langle \nabla_x L_{\bar x}( \hat y, \hat \lambda ) , x - \bar x \rangle \mbox{ where }\\ 
\eta_1(\varepsilon, \bar x ) & = & -\min \{\langle \hat y + \hat \lambda e , \hat y - y \rangle : \frac{1}{2} \leq \sum_{i=1}^n y_i \leq 1 \},\\
\nabla_x L_{\bar x}(\hat y, \hat \lambda  ) & = & \bar x + \hat \lambda e \mbox{ with }e_i=1,i=1,\ldots,n.
\end{array}
\right.
\end{equation}
}\fi

As in the previous example, we take $S$ of form 
$S=AA^T + \lambda I_{2 n}$ where the entries of $A$ are randomly selected
in the range $[-20,20]$. We also take $c_1(i)=c_2(i)=1,i=1,\ldots,n$.
For 8 values of the pair $(n,\lambda)$, namely 
$(n,\lambda) \in \{(1,1), (10,1),(100,1),(1000,1),(1,100),(10,100),(100,100),(1000,100)\}$, we generate a matrix $S$ of form $AA^T + \lambda I_{2 n}$
where the entries of $A$ are realizations of independent random variables with uniform distribution in $[-20,20]$.
In each case, we select randomly $\bar x \in X$ and solve \eqref{defqxsimple2}, \eqref{defqxsimple3} and its dual 
written for $x= \bar x$ using Mosek interior point solver. The value of $\alpha(\bar x)=\lambda_{\min}(S_3)$, the 
dual function $\theta_{\bar x}(\cdot)$, and 
the dual iterates computed along the iterations are reported in Figure \ref{fig:f_1} in the Appendix.
Figure \ref{fig:f_2} shows the plots of $\eta_{1}(\varepsilon_k, \bar x)$ and $\eta_{2}(\varepsilon_k, \bar x)$
as a function of iteration $k$ where $\varepsilon_k$ is the duality gap at iteration $k$.

The cuts computed by Proposition \ref{defcutctk} are more conservative than cuts given by Proposition \ref{varprop1}
on nearly all instances and iterations.
We also see that, as expected, the error terms  $\eta_1(\varepsilon_k, \bar x)$ and $\eta_2(\varepsilon_k, \bar x)$
go to zero when $\varepsilon_k$ goes to zero (see the proof of Theorem \ref{ismdrn1prop} for a proof of this statement).

\section{Inexact Stochastic Mirror Descent for two-stage nonlinear stochastic programs}\label{sec:ismd}

The algorithm to be described in this section is an inexact extension of  SMD \cite{nemjudlannem09}
to solve 
\begin{equation}\label{defpb4}
\left\{
\begin{array}{l}
\min \; f(x_1):=f_1( x_1 ) + \mathcal{Q}(x_1) \\
x_1 \in X_1
\end{array}
\right.
\end{equation}
with $X_1 \subset \mathbb{R}^n$ a convex, nonempty, and compact set, and $\mathcal{Q}(x_1)=\mathbb{E}_{\xi_2}[ \mathfrak{Q}( x_1 , \xi_2 ) ]$, $\xi_2$ is a random vector
with probability distribution $P$ on  $\Xi \subset \mathbb{R}^k$, and
\begin{equation}\label{pbsecondstage4}
\mathfrak{Q}( x_1 , \xi_2 ) = 
\left\{
\begin{array}{l}
\min_{x_2} \; f_2( x_2 , x_1 , \xi_2) \\
x_2 \in X_2(x_1 , \xi_2 ):=\{x_2 \in \mathcal{X}_2 : \;A x_2 + B x_1 = b, \;g(x_2, x_1 , \xi_2) \leq 0\}.
\end{array}
\right.
\end{equation}
Recall that $\xi_2$ contains the random variables in $(A,B,b)$ and eventually other sources
of randomness.
Let $\|\cdot\|$ be a norm on $\mathbb{R}^n$ and let $\omega: X_1\to \mathbb{R}$ be a {\sl distance-generating function}.
This function should
\begin{itemize}
\item be convex and continuous on $X_1$,
\item admit on $X_1^o=\{x\in X_1:\partial \omega(x)\neq\emptyset\}$ a selection $\omega'(x)$ of subgradients, and
\item be compatible with $\|\cdot\|$, meaning that $\omega(\cdot)$ is strongly convex with constant of strong convexity $\mu(\omega)>0$ with respect to
the norm $\|\cdot\|$:
$$
( \omega'(x)-\omega'(y))^T (x-y) \geq \mu( \omega) \|x-y\|^2\,\,\forall x,y\in X_1^o.
$$
\end{itemize}
We also define
\begin{enumerate}
\item the {\sl $\omega$-center of $X_1$} given by $x_{1 \omega} =\displaystyle \argmin_{x_1 \in X_1} \, \omega(x_1)\in X_1^o$;
\item the {\sl Bregman distance} or prox-function
\begin{equation}\label{strong}
V_x(y)=\omega(y)-\omega(x)- (y-x)^T \omega'(x),
\end{equation}
for $x\in X_1^o$, $y\in X_1$;
\item the {\sl $\omega$-radius of $X_1$} defined as
\begin{equation}\label{defDomega}
D_{\omega, X_1}=\sqrt{2 \Big[\max_{x \in X_1}\omega(x)-\min_{x \in X_1} \omega(x)\Big]}.
\end{equation}
\item {\sl The proximal mapping}
\begin{equation} \label{defprox}
\Prox_x(\zeta)=\argmin_{y\in X_1} \{\omega(y)+y^T (\zeta-\omega'(x)) \} \;\;\; [x\in X_1^o,\zeta\in \mathbb{R}^n],
\end{equation}
taking values in $X_1^o$.\\
\end{enumerate}

We describe below ISMD, an inexact variant of SMD for solving problem \eqref{defpb4} in which primal and dual
second stage problems are solved approximately.

For $x_1 \in X_1$, $\xi_2 \in \Xi$, and $\varepsilon \geq 0$, we denote by
$x_2(x_1,\xi_2,\varepsilon)$ an  
$\varepsilon$-optimal feasible primal solution of
\eqref{pbsecondstage4}, i.e., $x_2(x_1,\xi_2,\varepsilon) \in X_2(x_1,\xi_2)$ and
$$
\mathfrak{Q}(x_1,\xi_2) \leq f_2(x_2,x_1,\xi_2) \leq \mathfrak{Q}(x_1,\xi_2) + \varepsilon.
$$
We now define $\varepsilon$-optimal dual second stage solutions.
For $x_1 \in X_1$ and $\xi_2 \in \Xi$ 
let 
$$
L_{x_1, \xi_2}(x_2 , \lambda , \mu ) =f_2( x_2 , x_1, \xi_2) + \langle \lambda , A x_2 + B x_1 - b \rangle  + \langle   \mu ,  g(x_2, x_1, \xi_2) \rangle,
$$
and let $\theta_{x_1, \xi_2}$ be the dual function given by
\begin{equation}\label{defdualfunctionsstage}
\theta_{x_1, \xi_2}(\lambda, \mu) = \left\{ 
\begin{array}{l}
\min \; L_{x_1, \xi_2}(x_2 , \lambda , \mu )\\
x_2 \in \mathcal{X}_2.
\end{array}
\right.
\end{equation}
For $x_1 \in X_1$, $\xi_2 \in \Xi$, and $\varepsilon \geq 0$, we denote by
$(\lambda(x_1,\xi_2,\varepsilon), \mu(x_1,\xi_2,\varepsilon))$
an $\varepsilon$-optimal feasible solution of the dual problem 
\begin{equation}\label{dualproblemsecondstage}
\left\{ 
\begin{array}{l}
\max \; \theta_{x_1, \xi_2}( \lambda , \mu ) \\
\mu \geq 0, \lambda = A x_2 + B x_1 -b, x_2 \in \mbox{Aff}(\mathcal{X}_2).
\end{array}
\right.
\end{equation}
Under Slater-type constraint qualification conditions to be specified in Theorems \ref{ismdrn1prop}
and \ref{ismdrn2}, the optimal values of primal second stage problem \eqref{pbsecondstage4}
and dual second stage problem \eqref{dualproblemsecondstage} are the same and 
$(\lambda(x_1,\xi_2,\varepsilon), \mu(x_1,\xi_2,\varepsilon))$ satisfies:
$$
\mu(x_1,\xi_2,\varepsilon) \geq 0, \lambda(x_1,\xi_2,\varepsilon) = A x_2 + B x_1 -b,
$$
for some $x_2 \in \mbox{Aff}(\mathcal{X}_2)$ and
$$
\mathfrak{Q}(x_1, \xi_2) - \varepsilon \leq \theta_{x_1, \xi_2}( \lambda(x_1,\xi_2,\varepsilon), \mu(x_1,\xi_2,\varepsilon) ) \leq \mathfrak{Q}(x_1, \xi_2).
$$
We also denote by
$D_{X_1}=\max_{x, y \in X_1} \|y-x\|$
the diameter of $X_1$, by $s_{f_1}( x_1)$ a subgradient of $f_1$ at $x_1$, 
and we define
\begin{equation}\label{defGH}
\begin{array}{lll}
H(x_1 , \xi_2, \varepsilon ) =\nabla_{x_1} f_2( x_2(x_1,\xi_2,\varepsilon), x_1, \xi_2) + B^T \lambda(x_1,\xi_2,\varepsilon) +
\sum_{i=1}^p \mu_i(x_1,\xi_2,\varepsilon) \nabla_{x_1} g_i(x_2(x_1,\xi_2,\varepsilon), x_1, \xi_2),\\
G(x_1 , \xi_2, \varepsilon ) =s_{f_1}( x_1 ) + H(x_1 , \xi_2, \varepsilon ).
\end{array}
\end{equation}
\rule{\linewidth}{1pt}
\par {\textbf{Inexact Stochastic Mirror Descent (ISMD) for risk-neutral two-stage nonlinear stochastic problems.}} \\
\par {\textbf{Parameters:}} Sequence $(\varepsilon_t)$ and $\theta>0$.\\
\par {\textbf{For }}$N=2,3,\ldots,$  
\par \hspace*{0.7 cm}Take $x_1^{N,1} = x_{1 \omega}$. 
\par \hspace*{0.7cm}{\textbf{For }}$t=1,\ldots,N-1$, sample a realization $\xi_2^{N,t}$ of $\xi_2$ (with corresponding realizations
$A^{N,t}$ of $A$, $B^{N,t}$ of $B$, and $b^{N,t}$ of $b$), compute an
$\varepsilon_t$-optimal solution $x_2^{N,t}$ of the problem
\begin{equation}\label{primalproblem}
\mathfrak{Q}( x_1^{N,t} , \xi_2^{N,t} ) = \left\{ 
\begin{array}{l}
\min_{x_2} \; f_2( x_2, x_1^{N,t}, \xi_2^{N,t}) \\
A^{N,t} x_2 + B^{N,t} x_1^{N,t} = b^{N,t},\\
g(x_2, x_1^{N,t}, \xi_2^{N,t}) \leq 0,\\
x_2 \in \mathcal{X}_2,
\end{array}
\right. 
\end{equation}
and an $\varepsilon_t$-optimal solution $(\lambda^{N,t}, \mu^{N,t})=(\lambda(x_1^{N,t},\xi_2^{N,t},\varepsilon_t), \mu(x_1^{N,t},\xi_2^{N,t},\varepsilon_t))$  of the dual problem 
\begin{equation}\label{dualpbmodif}
\left\{ 
\begin{array}{l}
\max \; \theta_{x_1^{N,t}, \xi_2^{N,t}}( \lambda , \mu ) \\
\mu \geq 0, \lambda = A^{N,t} x_2 + B^{N,t} x_1^{N,t} -b^{N,t}, x_2 \in \mbox{Aff}(\mathcal{X}_2)
\end{array}
\right.
\end{equation}
used 
to compute $G(x_{1}^{N,t} ,\xi_2^{N,t}, \varepsilon_t )$ given by \eqref{defGH} replacing 
$(x_1, \xi_2, \varepsilon)$
by $(x_1^t, \xi_2^t, \varepsilon_t)$.\footnote{Any optimization solver for convex nonlinear programs able to provide 
$\varepsilon_t$-optimal solutions can be used (for instance an interior point solver).}
\par \hspace*{0.7cm}Compute $\gamma_t(N)=\frac{\theta}{\sqrt{N}}$ and
\begin{equation}\label{rsaiterations}
\begin{array}{rcl}
x_1^{N, t+1}&=&\Prox_{x_1^{N, t}}(\gamma_t(N) G(x_1^{N, t} ,\xi_2^{N,t}, \varepsilon_t )).
\end{array}
\end{equation}
\par \hspace*{0.7cm}Compute
\begin{equation}\label{outputsrsa}
\begin{array}{l}
x_1( N  ) = \displaystyle \frac{1}{\Gamma_N} \sum_{\tau=1}^N \gamma_\tau( N) x_1^{N,\tau}  \mbox{ and }\\
{\hat f}_N =\displaystyle \frac{1}{\Gamma_N}\left[ \sum_{\tau=1}^N  \gamma_\tau( N) \Big(  f_1(x_1^{N,\tau}) + f_2( x_2^{N,\tau }, x_1^{N,\tau}, \xi_2^{N,\tau}) \Big)  \right] \mbox{ with }\Gamma_N=\displaystyle \sum_{\tau=1}^N \gamma_\tau( N ).
\end{array}
\end{equation}
\par \hspace*{0.7cm}{\textbf{End For}}
\par {\textbf{End For}}\\
\rule{\linewidth}{1pt}
\begin{rem} In practise ISMD is run fixing the number $N$ of inner iterations, i.e., we fix $N$
and compute $x_1(N)$ and $\hat f_N$. 
\end{rem}
Convergence of Inexact Stochastic Mirror Descent for solving \eqref{defpb4} can be shown when 
error terms $(\varepsilon_t)$ asymptotically vanish:
\begin{thm}[Convergence of ISMD]\label{ismdrn1prop} Consider problem \eqref{defpb4} and assume that (i) $X_1$ and $\mathcal{X}_2$
are nonempty, convex, and compact, (ii)
$f_1$ is convex, finite-valued, and has bounded subgradients on $X_1$, (iii) for every $x_1 \in X_1$ and $x_2 \in \mathcal{X}_2$,
$f_2(x_2,x_1,\cdot)$ and $g_i(x_2,x_1,\cdot),i=1,\ldots,p$, are measurable, (iv) for every $\xi_2 \in \Xi$ the functions 
$f_2(\cdot,\cdot,\xi_2)$ and $g_i(\cdot,\cdot,\xi_2),i=1,\ldots,p$, are convex and continuously differentiable on $\mathcal{X}_2 \times X_1$,
(v) $\exists \kappa >0$ and $r>0$ such that for all $x_1 \in X_1$, for all ${\tilde \xi}_2 \in \Xi$,
there exists $x_2 \in \mathcal{X}_2$
such that $\mathbb{B}(x_2 ,r) \cap \mbox{Aff}(\mathcal{X}_2) \neq \emptyset$,
${\tilde A} x_2 + {\tilde B} x_1 ={\tilde b}$, and $g(x_2,x_1,{\tilde \xi}_2)<-\kappa e$ where $e$
is a vector of ones. If $\gamma_t=\frac{\theta}{\sqrt{N}}$ for some $\theta>0$, if the support $\Xi$
of $\xi_2$ is compact, and if $\lim_{t \rightarrow \infty} \varepsilon_t = 0$, then
$$
\lim_{N \rightarrow +\infty} \mathbb{E}[ f(x_1(N)) ]   = \lim_{N \rightarrow +\infty} \mathbb{E}[ {\hat f}_N ] = f_{1 *} 
$$
where $f_{1 *}$ is the optimal value of \eqref{defpb4}.
\end{thm}
\begin{proof} 
For fixed $N$, to alleviate notation, we denote vectors $x_1^{N,t}, x_{2}^{N,t}, \xi_2^{N,t}, 
A^{N,t}, B^{N,t}, b^{N,t}, \gamma_t(N), \lambda^{N,t}, \mu^{N,t}$
 used to compute $x_1(N)$ and $\hat f_N$ by  $x_1^{t}, x_{2}^{t}, \xi_2^{t}, 
A^{t}, B^{t}, b^{t}, \gamma_t, \lambda^{t}, \mu^{t}$, respectively.
Let $x_1^*$ be an optimal solution of \eqref{defpb4}. 
Standard computations on the proximal mapping give
\begin{equation}\label{proxmapgen}
\sum_{\tau=1}^N \gamma_\tau  G(x_1^\tau , \xi_2^\tau , \varepsilon_\tau)^T (x_1^\tau - x_1^*) \leq 
{1\over 2} D_{\omega, X_1}^2+\frac{1}{2 \mu( \omega )}\sum_{\tau=1}^N\gamma_\tau^2\|G(x_1^\tau , \xi_2^\tau , \varepsilon_\tau)\|_*^2.
\end{equation}
\if{
Setting $f_{\varepsilon}'(x)= \mathbb{E}_{\xi}[G(x,\xi,\varepsilon)]$ and
$\Delta_\tau=G(x_1^\tau, \xi_2^\tau, \varepsilon_\tau) - f_{\varepsilon_\tau}'( x_1^\tau )$ the above inequality can be re-written
\begin{equation}\label{projectionfirst} 
\sum_{\tau=1}^N \gamma_{\tau} 
  f_{\varepsilon_\tau}'( x_1^\tau )^T ( x_1^\tau - x_1^* )
\leq \frac{1}{2} D_{\omega, X_1}^2 + \frac{1}{2 \mu( \omega ) } \sum_{\tau=1}^N  \gamma_{\tau}^2 \|G(x_1^{\tau}, \xi_2^{\tau}, \varepsilon_{\tau})\|_*^2
+ \sum_{\tau=1}^N \gamma_{\tau} 
\Delta_{\tau}^T ( x_1^* - x_1^{\tau}).
\end{equation}
Setting $\xi_2^{1:\tau-1}=(\xi_2^1,\ldots,\xi_2^{\tau-1})$, 
observe that 
\begin{equation}\label{deltaeq} 
\begin{array}{lll}
\mathbb{E}[\langle \Delta_\tau , x_1^*  - x_1^\tau \rangle ] & = & 
\mathbb{E}[ \mathbb{E}[ \langle \Delta_\tau , x_1^*  - x_1^\tau \rangle |   \xi_2^{1:\tau-1}   ] ] = 
\mathbb{E}[  \langle \mathbb{E} [ \Delta_\tau | \xi_2^{1:\tau-1} ] , x_1^*  - x_1^\tau \rangle  ]\\
& = & \mathbb{E}[  \langle \mathbb{E} [ G(x_1^\tau, \xi_2^\tau, \varepsilon_\tau) - f_{\varepsilon_\tau}'( x_1^\tau ) | \xi_2^{1:\tau-1} ] , x_1^*  - x_1^\tau \rangle  ]=0.
\end{array}
\end{equation}
}\fi
Next using Proposition \ref{varprop1} we have 
\begin{equation}\label{ineqqx1tau} 
\mathfrak{Q}(x_1^* , \xi_2^\tau ) 
\geq \mathfrak{Q}(x_1^\tau , \xi_2^\tau ) - \eta_{\xi_2^{\tau}}(\varepsilon_\tau , x_1^\tau  ) 
+ \langle H(x_1^{\tau}, \xi_2^{\tau}, \varepsilon_{\tau} ) , x_1^* - x_1^{\tau} \rangle
\end{equation}
where 
\begin{equation}\label{defeta1}
\begin{array}{lll}
\eta_{\xi_2^{\tau}}(\varepsilon_\tau , x_1^\tau  )& =& 
\left\{ 
\begin{array}{l}
\max \;\langle \nabla_{x_2} L_{x_1^\tau , \xi_2^{\tau}}( x_2^\tau , \lambda^{\tau}, \mu^{\tau} ) , x_2^{\tau} - x_2 \rangle \\
x_2 \in \mathcal{X}_2
\end{array}
\right. \\
& = &  
\left\{ 
\begin{array}{l}
\max \;\langle \nabla_{x_2} f_2(x_2^\tau , x_1^\tau , \xi_2^\tau ) + (A^\tau)^T  \lambda^{\tau} + \sum_{i=1}^p \mu_i^{\tau} \nabla_{x_2} g_i(x_2^\tau , x_1^{\tau}, \xi_2^{\tau} )     , x_2^{\tau} - x_2 \rangle  \\
x_2 \in \mathcal{X}_2.
\end{array}
\right.
\end{array}
\end{equation}
Setting $\xi_2^{1:\tau-1}=(\xi_2^1,\ldots,\xi_2^{\tau-1})$ and
taking the conditional expectation $\mathbb{E}_{\xi_2^\tau}[\cdot|\xi_2^{1:\tau-1}]$ on each side of \eqref{ineqqx1tau}  we obtain almost surely
\begin{equation}\label{subgradientf}
\mathcal{Q}(x_1^* ) \geq \mathcal{Q}( x_1^\tau ) -  \mathbb{E}_{\xi_2^\tau}[ \eta_{\xi_2^{\tau}}(\varepsilon_\tau , x_1^\tau  ) |\xi_2^{1:\tau-1}]
+ ( \mathbb{E}_{\xi_2^\tau}[H(x_1^\tau , \xi_2^\tau ,\varepsilon_\tau )   | \xi_2^{1:\tau-1} ]   )^T  ( x_1^* - x_1^{\tau} ).
\end{equation}
Combining \eqref{proxmapgen}, \eqref{subgradientf}, and using the convexity of $f$ we get
\begin{equation}\label{ismdrn1}
\begin{array}{l}
0 \leq \mathbb{E}[ f(x_1(N))-f(x_1^*) ]   \leq   \displaystyle \frac{1}{\Gamma_N} \sum_{\tau=1}^N \gamma_\tau \mathbb{E}[ f( x_1^\tau )    -   f ( x_1^* ) ] \\
 \leq  \displaystyle \frac{1}{\Gamma_N} \sum_{\tau=1}^N \gamma_\tau \mathbb{E}[  \eta_{\xi_2^{\tau}}(\varepsilon_\tau , x_1^\tau  )  ]
+ \frac{1}{2 \Gamma_N}\Big[ D_{\omega, X_1}^2 +  \frac{1}{\mu(\omega)} \sum_{\tau=1}^N  \gamma_{\tau}^2 \mathbb{E}[ \|G(x_1^{\tau}, \xi_2^{\tau}, \varepsilon_{\tau})\|_*^2 ]  \Big].
 \end{array}
\end{equation}
We now show by contradiction that\footnote{The proof is similar to the proof of Proposition 4.6 in \cite{guigues2016isddp}.} 
\begin{equation}\label{etaconv0}
\lim_{\tau \rightarrow +\infty} \eta_{\xi_2^{\tau}}(\varepsilon_\tau , x_1^\tau  ) = 0 \mbox{ almost surely.}
\end{equation}
Take an arbitrary realization of ISMD. We want to show that 
\begin{equation}\label{etaconv}
\lim_{\tau \rightarrow +\infty} \eta_{\xi_2^{\tau}}(\varepsilon_\tau , x_1^\tau  ) = 0 
\end{equation}
for that realization.
Assume that \eqref{etaconv} does not hold. 
Let $x_{2 *}^t$ (resp. ${\tilde x}_{2}^{\tau}$) be an optimal solution of \eqref{primalproblem} (resp. \eqref{defeta1}).
Then there is $\varepsilon_0>0$ and $\sigma_1: \mathbb{N} \rightarrow \mathbb{N}$
increasing such that for every $\tau \in \mathbb{N}$, we have 
\begin{equation}\label{etaconv2}
 \langle \nabla_{x_2} 
 f_2({x}_{2}^{\sigma_1(\tau)} , x_1^{\sigma_1(\tau)} , 
 \xi_2^{\sigma_1(\tau)} ) + (A^{\sigma_1(\tau)})^T  \lambda^{\sigma_1(\tau)} + 
 \sum_{i=1}^p \mu_i^{\sigma_1(\tau)} \nabla_{x_2} g_i( x_{2}^{\sigma_1( \tau )}, x_1^{\sigma_1 ( \tau ) }, \xi_2^{\sigma_1 (  \tau ) } )     , x_2^{\sigma_1 ( \tau ) } -  {\tilde x}_{2}^{ \sigma_1 ( \tau )} \rangle
 \geq \varepsilon_0.
\end{equation}
By  $\varepsilon_t$-optimality of $x_{2}^t$ we obtain
\begin{equation}\label{optx2*t}
f_2(x_{2 *}^t , x_1^t ,\xi_2^t) \leq  
f_2(x_{2}^t , x_1^t ,\xi_2^t) \leq f_2(x_{2 *}^t , x_1^t ,\xi_2^t) + \varepsilon_t.
\end{equation}
Using Assumptions (i), (iii), (iv), and Proposition 3.1 in \cite{guigues2016isddp} we get that 
the sequence $(\lambda^{\tau}, \mu^{\tau})_{\tau}$ is almost surely bounded.
Let $\mathcal{D}$
be a compact set to which  this sequence belongs.
By compacity, we can find $\sigma_2:\mathbb{N} \rightarrow \mathbb{N}$ increasing such that
setting $\sigma= \sigma_1 \circ \sigma_2$ the sequence
$(x_2^{\sigma(\tau)}, x_1^{\sigma(\tau)}, \lambda^{\sigma(\tau)}, \mu^{\sigma(\tau)}, \xi_2^{\sigma(\tau)}  )$ converges 
to some $(\bar x_2, x_{1 *}, \lambda_*, \mu_*, \xi_{2 *})
\in \mathcal{X}_2 \times X_1 \times \mathcal{D} \times \Xi$.
We will denote by $A_*, B_*, b_*$ the values of $A, B,$ and $b$ in $\xi_{2 *}$.
By continuity arguments there is $\tau_0 \in \mathbb{N}$ such that for every $\tau \geq \tau_0$:
\begin{equation}\label{contradeta23}
\begin{array}{l}
\left|\langle \nabla_{x_2} 
 f_2(x_{2}^{\sigma(\tau)} , x_1^{\sigma(\tau)} , 
 \xi_2^{\sigma(\tau)} ) + (A^{\sigma(\tau)})^T  \lambda^{\sigma(\tau)} + 
 \sum_{i=1}^p \mu_i^{\sigma(\tau)} \nabla_{x_2} g_i({x}_{2}^{\sigma( \tau )}, x_1^{\sigma ( \tau ) }, \xi_2^{\sigma (  \tau ) } )     , x_2^{\sigma ( \tau ) } -  {\tilde x}_{2}^{ \sigma ( \tau )} \rangle \right.\\
- \left. \langle \nabla_{x_2} 
 f_2(\bar x_2 , x_{1 *}, 
 \xi_{2 *}) + A_*^T  \lambda_*  + 
 \sum_{i=1}^p \mu_*(i) \nabla_{x_2} g_i(\bar x_2, x_{1 *}, \xi_{2 *}  )   , \bar x_2 -  {\tilde x}_{2}^{ \sigma ( \tau )} \rangle \right|
 \leq \varepsilon_{0}/2.
\end{array}
\end{equation}
We deduce from \eqref{etaconv2} and \eqref{contradeta23} that for all $\tau \geq \tau_0$
\begin{equation}\label{caseBposcontrad}
\left \langle \nabla_{x_2} 
 f_2(\bar x_2 , x_{1 *}, 
 \xi_{2 *}) + A_*^T  \lambda_*  + 
 \sum_{i=1}^p \mu_*(i) \nabla_{x_2} g_i(\bar x_2, x_{1 *}, \xi_{2 *} )    , \bar x_2 -  {\tilde x}_{2}^{ \sigma ( \tau )} \right  \rangle \geq \varepsilon_0/2 >0.
\end{equation}
Assumptions (i)-(iv) imply that
primal problem \eqref{primalproblem} and dual problem \eqref{dualpbmodif} have the same
optimal value and for every $x_2 \in \mathcal{X}_2$ and $\tau \geq \tau_0$ we have:
{\small{
$$
\begin{array}{l}
f_2(x_2^{\sigma(\tau)},x_1^{\sigma(\tau)},\xi_2^{\sigma(\tau)}) + \langle A^{\sigma(\tau)} x_2^{\sigma(\tau)} + B^{\sigma(\tau)} x_1^{\sigma(\tau)} - b^{\sigma(\tau)} ,   \lambda^{\sigma(\tau)} \rangle + \langle \mu^{\sigma(\tau)} ,  g( x_2^{\sigma(\tau)},x_1^{\sigma(\tau)},\xi_2^{\sigma(\tau)}) \rangle  \\
\leq  f_2(x_{2 *}^{\sigma(\tau)},x_1^{\sigma(\tau)},\xi_2^{\sigma(\tau)}) + \varepsilon_{\sigma(\tau)} \;\;\mbox{by definition of }x_{2 *}^{\sigma(\tau)}, x_2^{\sigma(\tau)}\mbox{ and since }\mu^{\sigma(\tau)} \geq 0, x_2^{\sigma(\tau)} \in X_2( x_1^{\sigma(\tau)  } , \xi_2^{\sigma(\tau)  } ),\\
\leq \theta_{x_1^{\sigma(\tau)}, \xi_2^{\sigma(\tau)}}  (\lambda^{\sigma(\tau)} , \mu^{\sigma(\tau)} ) + 2 \varepsilon_{\sigma(\tau)}, [(\lambda^{\sigma(\tau)} , \mu^{\sigma(\tau)})\mbox{ is an }\epsilon_{\sigma(\tau)}\mbox{-optimal dual solution and there is no duality gap}],\\
\leq f_2(x_2, x_1^{{\sigma(\tau)}}, \xi_2^{\sigma(\tau)} )  + \langle A^{\sigma(\tau)} x_2 + B^{\sigma(\tau)} x_1^{\sigma(\tau)} - b^{\sigma(\tau)} , \lambda^{\sigma(\tau)} \rangle 
+ \langle \mu^{\sigma(\tau)} ,  g( x_2 , x_1^{\sigma(\tau)} , \xi_2^{\sigma(\tau)} ) \rangle  +   2 \varepsilon_{\sigma(\tau)}
\end{array}
$$
}}
where in the last relation we have used the definition of $\theta_{x_1^{\sigma(\tau)}, \xi_2^{\sigma(\tau)}}$.
Taking the limit in the above relation as $\tau \rightarrow +\infty$, we get  for every $x_2 \in \mathcal{X}_2$:
$$
\begin{array}{l}
f_2(\bar x_2, x_{1 *}, \xi_{2 *} )  + \langle A_* \bar x_2 + B_* x_{1 *} - b_* , \lambda_{*} \rangle + 
\langle \mu_{*} , g(\bar x_2 , x_{1 *}, \xi_{2 *} ) \rangle  \\
\leq  f_2(x_{2}, x_{1 *}, \xi_{2 *} )  + \langle A_* x_{2} + B_* x_{1 *} - b_* , \lambda_{*} \rangle + 
\langle \mu_{*} , g(x_{2} , x_{1 *}, \xi_{2 *} ) \rangle.
\end{array}
$$
Recalling that $\bar x_{2} \in \mathcal{X}_2$ this shows that $\bar x_{2}$ is an optimal solution of
\begin{equation}
\left\{
\begin{array}{l}
\min  f_2(x_{2}, x_{1 *}, \xi_{2 *} )  + \langle A_* x_{2} + B_* x_{1 *} - b_* , \lambda_{*} \rangle +  \langle \mu_{*} , g(x_{2} , x_{1 *}, \xi_{2 *} ) \rangle\\
x_2 \in \mathcal{X}_2.
\end{array}
\right.
\end{equation}
The first order optimality conditions for $\bar x_{2}$ can be written
\begin{equation}\label{caseBposcontradfinal}
\left \langle \nabla_{x_2} f_2( \bar x_{2}, x_{1 *}, \xi_{2 *} )  + A_*^T \lambda_{*} + \sum_{i=1}^p \mu_{*} ( i) \nabla_{x_2} g_{i}( \bar x_{2} , x_{1 *} , \xi_{2 *}), x_2  - \bar x_{2}  \right  \rangle \geq 0
\end{equation}
for all $x_2 \in \mathcal{X}_2$. Specializing the above relation for $x_2 = {\tilde x}_2^{ \sigma(\tau_0 ) } \in \mathcal{X}_2$, we get 
$$
\left \langle \nabla_{x_2} f_2( \bar x_{2}, x_{1 *}, \xi_{2 *} )  + A_*^T \lambda_{*} + \sum_{i=1}^p \mu_{*} ( i) \nabla_{x_2} g_{i}( \bar x_{2} , x_{1 *} , \xi_{2 *}),  {\tilde x}_2^{ \sigma(\tau_0 ) }  - \bar x_{2}  \right  \rangle \geq 0,
$$
but the left-hand side of the above inequality is $\leq -\varepsilon_0/2<0$ due to \eqref{caseBposcontrad} which yields the desired contradiction.
Therefore we have shown \eqref{etaconv0} and since the sequence $\eta_{\xi_2^{\tau}}(\varepsilon_\tau , x_1^\tau  )$ is almost surely bounded,
this implies $\lim_{\tau \rightarrow +\infty} \mathbb{E}[ \eta_{\xi_2^{\tau}}(\varepsilon_\tau , x_1^\tau  )]=0$
and consequently
$ \lim_{N \rightarrow +\infty}\frac{1}{\Gamma_N} \sum_{\tau=1}^N \gamma_\tau \mathbb{E}[ \eta_{\xi_2^{\tau}}(\varepsilon_\tau , x_1^\tau  )]=0$.
Using the boundedness of the sequence $(\lambda^t, \mu^t)$ and
Assumption (ii) we get that
$\|G(x_1^{\tau}, \xi_2^{\tau}, \varepsilon_{\tau})\|_*^2$ is almost surely bounded.
Combining these observations with relation \eqref{ismdrn1} and using the definition of $\gamma_t$ we have $\lim_{N \rightarrow +\infty} \mathbb{E}[ f(x_1(N)) ] = f_{1 *} $.
Finally, recalling relation \eqref{ismdrn1}, to show $\lim_{N \rightarrow +\infty} \mathbb{E}[ {\hat f}_N ] = f_{1 *}$ all we have to show is 
\begin{equation}\label{f2qzero}
\lim_{N \rightarrow +\infty} \frac{1}{\Gamma_N}
\sum_{\tau=1}^N \gamma_{\tau} \mathbb{E}[ \mathcal{Q}( x_1^\tau ) - f_2(x_2^\tau , x_1^\tau , \xi_2^\tau )  ]=0.
\end{equation}
The above relation immediately follows from
\begin{equation}\label{immfollowsendfirt}
\mathbb{E}[\mathcal{Q}(x_1^\tau ) ] 
=\mathbb{E}_{\xi_2^{1:\tau-1}}[\mathcal{Q}(x_1^\tau ) ]
= \mathbb{E}_{\xi_2^{1:\tau-1}}[ \mathbb{E}_{\xi_2^\tau}[ \mathfrak{Q}(x_1^\tau , \xi_2^\tau )  | \xi_2^{1:\tau-1} ]    ]
\leq \mathbb{E}_{\xi_2^{1:\tau}}[ f_2(x_2^\tau , x_1^\tau, \xi_2^\tau ) ] \leq \mathbb{E}[\mathcal{Q}(x_1^\tau ) ]  + \varepsilon_\tau
\end{equation}
which holds since $\mathfrak{Q}(x_1^\tau , \xi_2^\tau ) \leq f_2(x_2^\tau , x_1^\tau, \xi_2^\tau ) \leq \mathfrak{Q}(x_1^\tau , \xi_2^\tau )+ \varepsilon_\tau$ by definition of $x_2^\tau$.\hfill
\end{proof}
\begin{rem} Output $\hat f_N$ of ISMD is a computable approximation of the optimal value $f_{1 *}$
of optimization problem \eqref{defpb4}.
\end{rem}

\begin{thm}\label{ismdrn2}[Convergence rate for ISMD]  Consider problem \eqref{defpb4} and assume that 
Assumptions (i)-(iv) of Theorem \ref{ismdrn1prop} are satisfied.
We alse make the following assumptions:
\begin{itemize}
\item[(a)] $\exists \alpha>0$ such that for every $\xi_2 \in \Xi$, for every $x_1 \in X_1$,
for every $y_1, y_2 \in \mathcal{X}_2$ we have
$$
f_2 (y_2 , x_1,  \xi_2 ) \geq f_2 (y_1 , x_1, \xi_2 ) + (y_2 - y_1 )^T \nabla_{x_2} f_2(y_1 , x_1 , \xi_2 )   + \frac{\alpha}{2} \|y_2- y_1\|_2^2;
$$
\item[(b)] there is $0 < M_1 < +\infty$ such that for every $\xi_2 \in \Xi$, for every $x_1 \in X_1$,
for every $y_1, y_2 \in \mathcal{X}_2$ we have
$$
\|\nabla_{x_1} f_2(y_2, x_1, \xi_2) -  \nabla_{x_1} f_2(y_1, x_1, \xi_2)  \|_2  \leq  M_1 \|y_2   - y_1\|_2;
$$
\item[(c)] there is $0 < M_2 < +\infty$ such that for every $\xi_2 \in \Xi$, for every $x_1 \in X_1$, for every $i=1,\ldots,p$,
for every $y_1, y_2 \in \mathcal{X}_2$, we have
$$
\|\nabla_{x_1} g_i(y_2, x_1, \xi_2) -  \nabla_{x_1} g_i (y_1, x_1, \xi_2)  \|_2  \leq  M_2 \|y_2   - y_1\|_2;
$$
\item[(d)] $\exists \alpha_D>0$ such that for every $x_1 \in X_1$, for every $\xi_2 \in \Xi$, dual function $\theta_{x_1, \xi_2}$
given by \eqref{defdualfunctionsstage} is strongly concave on $D_{x_1, \xi_2}$ with constant of strong concavity $\alpha_D$
where $D_{x_1, \xi_2}$ is a set containing the set of solutions of  second stage dual problem 
\eqref{dualproblemsecondstage} such that $(\lambda^t, \mu^t) \in D_{x_1^t , \xi_2^t }$.
\item[(e)] There are functions  $G_0, M_0$ such that for every $x_1 \in X_1$, for every $x_2 \in \mathcal{X}_2$
we have $$
\begin{array}{l}
\max(\|B^T\| , \sqrt{p} \max_{i=1,\ldots,p} \|\nabla_{x_1} g_i(x_2, x_1, \xi_2)\|_2 )  \leq G_0(\xi_2) \mbox{ and }\|\nabla_{x_1} f_2(x_2, x_1, \xi_2)\|_2 \leq M_0(\xi_2)
\end{array}
$$
with $\mathbb{E}[G_0(\xi_2)]$ and $\mathbb{E}[M_0(\xi_2)]$ finite;
\item[(f)] There are functions $\overline{f}_2, \underline{f}_2$ such that for all $x_1 \in X_1, x_2 \in \mathcal{X}_2$ we have
$$
\underline{f}_2(  \xi_2 ) \leq  f_2(x_2, x_1, \xi_2) \leq \overline{f}_2(\xi_2)
$$
with $\mathbb{E}[\overline{f}_2(  \xi_2 )]$ and $\mathbb{E}[\underline{f}_2 (  \xi_2 ) ]$ finite.
\item[(g)] There exists $0<L(f_2)<+\infty$ such that for every $\xi_2 \in \Xi$, for every $x_1 \in X_1$, function $f_2(\cdot,x_1, \xi_2)$ is 
Lipschitz continuous with Lipschitz constant $L(f_2)$.
\end{itemize}
Let $\mathcal{A}$ be a compact set such that matrix $A$ in $\xi_2$ almost surely belongs to $\mathcal{A}$
and let $M_3<+\infty$ such that $\|s_{f_1}( x_1 ) \|_2 \leq M_3$ for all $x_1 \in X_1$.
Let $V_{\mathcal{X}_2}$ be the vector space 
$V_{\mathcal{X}_2} = \{x-y : x,y \in \emph{Aff}(\mathcal{X}_2 )\}$.
Define the functions $\rho$ and $\rho_*$ by
$$
\begin{array}{ll} 
\rho(A,z)=\left\{ 
\begin{array}{l}
\max \; t \|z\| \\
t \geq 0,\;tz \in A(\mathbb{B}(0,r) \cap V_{\mathcal{X}_2}),
\end{array}
\right.
&
\rho_* (A)=\left\{ 
\begin{array}{l}
\min \; \rho(A,z) \\
\|z\|=1,\;z \in A V_{\mathcal{X}_2}.
\end{array}
\right.
\end{array}
$$
Assume that $\gamma_t=\frac{\theta_1}{\sqrt{N}}$ and $\varepsilon_t=\frac{\theta_2}{t^2}$
for some $\theta_1, \theta_2 >0$.
Let
$$
\left\{ 
\begin{array}{rll}
\mathcal{U}_1&=&( \mathbb{E}[\overline{f}_2( \xi_2 ) ] - \mathbb{E}[\underline{f}_2( \xi_2 )]) / \kappa,\\
\mathcal{U}_2( r , \xi_2 )&=& \frac{\overline{f}_2( \xi_2 )  - \underline{f}_2( \xi_2 ) + \theta_2 + L(f_2) r}{\min(\rho_* , \kappa/2)} \emph{ with }\displaystyle \rho_* = \min_{A \in \mathcal{A}} \rho_{*}( A ),\\
{\overline{\mathcal{U}}}&=& \Big( (M_1 +  M_2 \mathcal{U}_1 ) \sqrt{\frac{2}{\alpha}} + \frac{2 \mathbb{E}[G_0(\xi_2)]  }{\sqrt{\alpha_D }} \Big) \emph{Diam}(\mathcal{X}_2),\\
M_*(r)& =& \sqrt{  \mathbb{E}( M_3 + M_0( \xi_2 )  + \sqrt{2} \mathcal{U}_2(r , \xi_2 )G_0(\xi_2))^2 }.
\end{array}
\right.
$$
Let ${\hat f}_N$ computed by ISMD.
Then there is  $r_0>0$ such that 
\begin{equation}\label{ismd2ineq}
f_{1 *} \leq \mathbb{E}[ {\hat f}_N ]  \leq f_{1 *} + \frac{2\theta_2 +{\overline{\mathcal{U}}} \sqrt{\theta_2}}{N} + \frac{{\overline{\mathcal{U}}} \sqrt{\theta_2} \ln(N)}{N} + 
\frac{\frac{D_{\omega, X_1}^2}{\theta_1} + \frac{\theta_1 M_*^2(r_0) }{\mu( \omega ) }   }{2\sqrt{N}}
\end{equation}
where $f_{1 *}$ is the optimal value of \eqref{defpb4}.
\end{thm}
\begin{proof} Let $x_1^*$ be an optimal solution of \eqref{defpb4}.
Under our assumptions, we can apply  Proposition \ref{defcutctk} to value function $\mathfrak{Q}(\cdot,\xi_2^t)$ and $\bar x =  x_1^t$, which gives
\begin{equation}\label{firstrelconvth21}
\mathfrak{Q}( x_1^*  , \xi_2^t) \geq  f_2(  x_2^t, x_1^t ,   \xi_2^t ) +  \langle H(x_1^{t}, \xi_2^{t}, \varepsilon_{t} ) , x_1^* - x_1^{t} \rangle - \eta_{\xi_2^t}(\varepsilon_t, x_1^t   ),
\end{equation}
where 
$$
\begin{array}{lll}
\eta_{\xi_2^t}(\varepsilon_t, x_1^t   ) &= & \varepsilon_t + \Big( M_1 + 
\frac{M_2}{\kappa}(f_2(\bar x_2^t, x_1^t, \xi_2^t) - {\underline{f}}_2(\xi_2^t)  ) \Big) \sqrt{\frac{2 \varepsilon_t }{\alpha}}\mbox{Diam}(\mathcal{X}_2),\\
&&+ 2  \max\Big(  \|(B^t)^T \|, \sqrt{p}\, \max_{i=1,\ldots,p} \|\nabla_{x_1} g_i(x_2^t, x_1^t, \xi_2^t)\|_2  \Big) \mbox{Diam}(\mathcal{X}_2) \sqrt{ \frac{  \varepsilon_t } { \alpha_D }},
\end{array}
$$
for some $\bar x_2^t \in \mathcal{X}_2$ depending on $\xi_2^{1:t}$.
Taking the conditional expectation $\mathbb{E}_{\xi_2^t}[\cdot|\xi_2^{1:t-1}]$ in \eqref{firstrelconvth21} and using (e)-(f), we get 
\begin{equation}\label{convproofexpcase21}
\mathcal{Q}( x_1^*  ) \geq \mathbb{E}_{\xi_2^t}[ f_2(  x_2^t, x_1^t ,   \xi_2^t )|\xi_2^{1:t-1}] +  \mathbb{E}_{\xi_2^t}[\langle H(x_1^{t}, \xi_2^{t}, \varepsilon_{t} ) , x_1^* - x_1^{t} \rangle |\xi_2^{1:t-1}]
- (\varepsilon_t + {\overline{\mathcal{U}}} \sqrt{ \varepsilon_t }).
\end{equation}
Summing \eqref{convproofexpcase21} with the relation
$$
f_1(x_1^* ) \geq f_1(x_1^t ) + \langle s_{f_1}(x_1^t ) , x_1^* - x_1^t \rangle
$$
and taking the expectation operator $\mathbb{E}_{\xi_2^{1:t-1}}[\cdot]$ on each side of the resulting inequality gives
\begin{equation}\label{convproofexpcase22}
f(x_1^*) \geq \mathbb{E}[f_2(  x_2^t, x_1^t ,   \xi_2^t ) + f_1(x_1^t)] + \mathbb{E}[\langle G(x_1^{t}, \xi_2^{t}, \varepsilon_{t} ) , x_1^* - x_1^{t} \rangle ]-(\varepsilon_t + {\overline{\mathcal{U}}} \sqrt{ \varepsilon_t }).
\end{equation}
From \eqref{convproofexpcase22}, we deduce
\begin{equation}\label{smdifinal}
\mathbb{E}[\hat f_N - f_{1 *}  ]
\leq \frac{1}{\Gamma_N} \sum_{t=1}^N \gamma_t (\varepsilon_t + {\overline{\mathcal{U}}} \sqrt{\varepsilon_t } ) 
+ \frac{1}{\Gamma_N} \sum_{t=1}^N \gamma_t \mathbb{E}[ \langle   G(x_1^t , \xi_2^t , \varepsilon_t) ,  x_1^t - x_1^* \rangle   ]. 
\end{equation}
Using Proposition 3.1 in \cite{guigues2016isddp} and our assumptions, we can find $r_0>0$ such  that $M_*^2(r_0)$ is an upper bound for 
$\mathbb{E}[\| G(x_1^t , \xi_2^t , \varepsilon_t)\|_*^2 ]$. 
Using this observation, \eqref{smdifinal}, and \eqref{firstrelconvth21} (which still holds), we get
\begin{equation}\label{nearlyendsecond}
\begin{array}{lll}
\mathbb{E}[\hat f_N - f_{1 *}  ]& \leq &  
 \frac{1}{N}\Big(\theta_2 \Big( 1  + \displaystyle \int_{1}^N \frac{dx}{x^2} \Big) +{\overline{\mathcal{U}}} \sqrt{\theta_2} \Big(1 + \displaystyle \int_{1}^N \frac{dx}{x}  \Big) \Big) 
 + \frac{1}{2 \theta_1 \sqrt{N}}\Big(D_{\omega, X_1}^2   + \frac{M_*^2(r_0) \theta_1^2}{\mu(\omega)}  \Big)\\
 & \leq & \frac{2\theta_2 +{\overline{\mathcal{U}}} \sqrt{\theta_2}}{N} + \frac{{\overline{\mathcal{U}}} \sqrt{\theta_2} \ln(N)}{N} + 
\frac{\frac{D_{\omega, X_1}^2}{\theta_1} + \frac{\theta_1 M_*^2(r_0) }{\mu( \omega ) }   }{2\sqrt{N}}.
\end{array}
\end{equation}
Finally
\begin{equation}\label{nearlyendthird}
\begin{array}{lcl}
\displaystyle 0 & \stackrel{\eqref{ismdrn1}}{\leq} &\displaystyle \frac{1}{\Gamma_N} \sum_{\tau=1}^N \gamma_\tau \mathbb{E}[f(x_1^\tau )] - f_{1 *}  \\
& = & \displaystyle \frac{1}{\Gamma_N} \sum_{\tau=1}^N \gamma_\tau \mathbb{E}[f_1(x_1^\tau ) + \mathcal{Q}(x_1^\tau )] - f_{1 *}  \\
& \stackrel{\eqref{immfollowsendfirt}}{\leq } &\displaystyle  \frac{1}{\Gamma_N} \sum_{\tau=1}^N \gamma_\tau \mathbb{E}[f_1(x_1^\tau ) + f_2(x_2^\tau , x_1^\tau , \xi_2^\tau )] - f_{1 *} = \mathbb{E}[\hat f_N - f_{1 *}].
\end{array}
\end{equation}
Combining \eqref{nearlyendsecond} and  \eqref{nearlyendthird} we obtain \eqref{ismd2ineq}.\hfill
\end{proof}


\section{Numerical experiments} \label{sec:simualgorithms}

We compare the performances of SMD, ISMD, SAA (Sample Average Approximation, see \cite{shadenrbook}),
and the L-shaped method (see \cite{birge-louv-book}) on two simple two-stage quadratic stochastic programs
which satisfy the assumptions of Theorems \ref{ismdrn1prop} and \ref{ismdrn2}.

The first two-stage program is
\begin{equation}\label{smdmodel11}
\left\{ 
\begin{array}{l}
\min \;c^T x_1 + \mathbb{E}[\mathfrak{Q}(x_1,\xi_2)]\\
x_1 \in \{x_1 \in \mathbb{R}^n : x_1 \geq 0, \sum_{i=1}^n x_1(i) = 1\}
\end{array}
\right.
\end{equation}
where the second stage recourse function is given by
\begin{equation}\label{smdmodel12}
\mathfrak{Q}(x_1 ,\xi_2)=\left\{ 
\begin{array}{l}
\displaystyle  \min_{x_2 \in \mathbb{R}^n} \;\frac{1}{2}\left( \begin{array}{c}x_1\\x_2\end{array} \right)^T \Big( \xi_2 \xi_2^T + \lambda I_{2 n} \Big)\left( \begin{array}{c}x_1\\x_2\end{array} \right) + \xi_2^T \left( \begin{array}{c}x_1\\x_2\end{array} \right)  \\
x_2 \geq 0, \displaystyle \sum_{i=1}^n x_2(i) = 1.
\end{array}
\right.
\end{equation}

The second two-stage program is 
\begin{equation}\label{smdmodel21}
\left\{
\begin{array}{l}
\min \; c^T x_1 + \mathbb{E}[\mathfrak{Q}(x_1,\xi_2)]\\
x_1 \in \{x_1 \in \mathbb{R}^n : \|x_1-x_0\|_2 \leq 1  \}
\end{array}
\right. 
\end{equation}
where cost-to-go function $\mathfrak{Q}(x_1,\xi_2)$ has nonlinear objective and constraint coupling functions and is given by
\begin{equation}\label{smdmodel22}
\mathfrak{Q}(x_1,\xi_2)=\left\{ 
\begin{array}{l}
\displaystyle  \min_{x_2 \in \mathbb{R}^n} \;\frac{1}{2}\left( \begin{array}{c}x_1\\x_2\end{array} \right)^T \Big( \xi_2 \xi_2^T + \lambda I_{2 n} \Big) \left( \begin{array}{c}x_1\\x_2\end{array} \right) + \xi_2^T \left( \begin{array}{c}x_1\\x_2\end{array} \right)  \\
\frac{1}{2} \|x_2-y_0\|_2^2 + \frac{1}{2}\|x_1-x_0\|_2^2 - \frac{R^2}{2} \leq 0.
\end{array}
\right.
\end{equation}
For both problems, $\xi_2$ is a Gaussian random vector in $\mathbb{R}^{2n}$ and $\lambda>0$.
We consider several instances of these problem
with $n=5$, $10$, $200$, $400$, and $n=600$.
For each instance, the components of $\xi_2$ are independent with means 
and standard deviations 
randomly generated in respectively intervals 
$[5,25]$ and $[5,15]$.
We fix $\lambda=2$ while the components of $c$ are generated randomly in interval $[1,3]$.
For problem \eqref{smdmodel21}-\eqref{smdmodel22} we also take $R=5$ and
$x_0(i)=y_0(i)=10,i=1,\ldots,n$.

In SMD and ISMD, we take $\omega(x)=\sum_{i=1}^n x_i \ln(x_i)$ for problem \eqref{smdmodel11}-\eqref{smdmodel12}. 
For this distance generating function, $x_+=$Prox$_{x}( \zeta)$ can be computed analytically for $x \in \mathbb{R}^n$ with $x>0$ (see \cite{nemjudlannem09, guiguesmstep17} for details): 
defining $z \in \mathbb{R}^n$ by
$z(i) = \ln(x(i))$ we have $x_+(i)=\exp(z_{+}(i))$ where 
$$
z_{+}=w-\ln\left(\sum_{i=1}^n e^{w(i)}\right) \mathbf{1} \mbox{ with }w=z-\zeta-\max_{i} [z(i)-\zeta(i)],
$$
and with $\mathbf{1}$ a vector in $\mathbb{R}^n$ of ones. 

For problem \eqref{smdmodel21}-\eqref{smdmodel22}, SMD and ISMD are run taking distance generating function
$\omega(x)=\frac{1}{2}\|x\|_2^2$ (in this case, SMD is just the Robust Stochastic Approximation).
For this choice of $\omega$, if $x_{+}=\mbox{Prox}_{x}(\zeta)$  we have 
$$
x_+ = \left\{ 
\begin{array}{ll}
x-\zeta & \mbox{if }\|x-\zeta-x_0\|_2 \leq 1,\\
x_0 + \frac{x-\zeta-x_0}{\|x-\zeta-x_0\|_2} & \mbox{otherwise.}
\end{array}
\right.
$$

In SMD and ISMD, the interior point solver of the Mosek Optimization Toolbox \cite{mosek}
is used at each iteration to solve the quadratic second stage problem (given first stage decision $x_1^t$ and realization $\xi_2^t$ of $\xi_2$ at iteration $t$)
and constant steps are used: if there are $N$ iterations, the step $\gamma_t$ for iteration $t$ is $\gamma_t=\frac{1}{\sqrt{N}}$.
For ISMD, we limit the number of iterations of Mosek solver used to solve subproblems.\footnote{According to current Mosek documentation, it is not possible to use absolute
errors. Therefore, early termination of the solver can either be obtained limiting the number
of iterations or defining relative errors.} More precisely, we consider four strategies for the limitation of these numbers of iterations 
given in Table \ref{tablenumberiter0} in the Appendix, which
define four variants of ISMD denoted by ISMD 1, ISMD 2, ISMD 3, and ISMD 4.
The variants that most limit the number of iterations are ISMD 1 and ISMD 2.
All methods were implemented in Matlab and run on an Intel Core i7, 1.8GHz, processor
with 12,0 Go of RAM.

To check the implementations and compare the accuracy and CPU time of all methods,
we first consider problems \eqref{smdmodel11}-\eqref{smdmodel12}
and \eqref{smdmodel21}-\eqref{smdmodel22} with $n=5,10$, and a large sample 
of size $N=20\;000$ of $\xi_2$.\footnote{The deterministic equivalents of these
instances are already large size quadratic programs. For instance, for 
$n=10$, the deterministic equivalent of Problem \eqref{smdmodel21}-\eqref{smdmodel22} 
is a quadratically constrained quadratic program with 200 010
variables and 20 0001 quadratic constraints.} In these experiments, the L-shaped method terminates
when the relative error is at most 5\%. The CPU time needed to solve these instances 
with the L-shaped method, SAA, and SMD are given in Table \ref{cpufirstsmallinstances}. For these instances, we also report in Table \ref{optvaluefirstinstances}
the approximate optimal values given by all methods knowing that for the L-shaped
method we report the value of the last upper bound computed.
For SMD, the approximate optimal value after $N$ iterations is given by $\hat f_N$.
On the four experiments, all methods give very close approximations of the optimal
value, which is a good indication that the methods were well implemented.
SMD is by far the quickest and SAA by far the slowest. 
For the instance of Problem \eqref{smdmodel11}-\eqref{smdmodel12} with $n=10$,
we report in the left plot of Figure \ref{figSMD1} the evolution of the approximate optimal value
along the iterations of SMD.\footnote{Naturally, after 
running $t-1$ of the $N-1$ total iterations, the approximate optimal value computed by SMD is
$\displaystyle \frac{1}{\sum_{\tau=1}^t \gamma_\tau(N)} \sum_{\tau=1}^t  \gamma_\tau( N) \Big(  f_1(x_1^{N,\tau}) + f_2( x_2^{N,\tau }, x_1^{N,\tau}, \xi_2^{N,\tau}) \Big)$
obtained on the basis of sample $\xi_2^{N,1},\ldots,\xi_2^{N,t}$ of $\xi_2$.}
We also report on the right plot of this figure the evolution of the upper and lower
bounds computed along the iterations of the L-shaped method for  
the instance of Problem \eqref{smdmodel11}-\eqref{smdmodel12} with $n=10$.
For problem \eqref{smdmodel21}-\eqref{smdmodel22}, the evolution of 
the approximate optimal value
along the iterations of SMD is represented in Figure \ref{figSMD2}.
Observe that with SMD the approximate optimal value is not the value of the objective function
at a feasible point and therefore some of these approximations can be below the optimal value 
of the problem.

\begin{table}
\begin{tabular}{|c|c|c|c|c|c|}
\hline 
$n$ & $N$ & Problem & {\tt{L-shaped}} &   {\tt{SAA}} & {\tt{SMD}}\\
\hline 
5 & 20 000 &  \eqref{smdmodel11} & 57.3 & 3 698.7 & 18.5 \\
\hline
5 & 20 000 &  \eqref{smdmodel21} & 53.1 & 3 943.8 & 22.7 \\
\hline
10 & 20 000 &  \eqref{smdmodel11} & 278.1 &$3.32\small{\times}10^5$ & 28.2 \\
\hline
10 & 20 000 &  \eqref{smdmodel21} & 70.5 & 4 126.5  & 33.4 \\
\hline
\end{tabular}
\caption{CPU time in seconds required to solve instances of problems 
\eqref{smdmodel11}-\eqref{smdmodel12} and \eqref{smdmodel21}-\eqref{smdmodel22} (for 
$n=5,10$ and $N=20\;000$) obtained with
the L-shaped method, SAA, and SMD.}\label{cpufirstsmallinstances}
\end{table}

\begin{table}
\begin{tabular}{|c|c|c|c|c|c|}
\hline 
$n$ & $N$ & Problem & {\tt{L-shaped}} &   {\tt{SAA}} & {\tt{SMD}}\\
\hline 
5 & 20 000 &  \eqref{smdmodel11} & 210.9 & 210.7 & 210.6 \\
\hline
5 & 20 000 &  \eqref{smdmodel21} & 1.122$\small{\times}10^6$ & $1.121\small{\times}10^6$ & $1.120\small{\times}10^6$ \\
\hline
10 & 20 000 &  \eqref{smdmodel11} & 78.8 & 78.9 &78.6 \\ 
\hline
10 & 20 000 &  \eqref{smdmodel21} & $3.020\small{\times}10^6$ & $3.016\small{\times}10^6$ & $3.015\small{\times}10^6$ \\ 
\hline
\end{tabular}
\caption{Approximate optimal value of instances of 
problems \eqref{smdmodel11}-\eqref{smdmodel12} and 
\eqref{smdmodel21}-\eqref{smdmodel22}
(for $n=5, 10$ and $N=20\;000$)
obtained with the L-shaped method, SAA, and SMD.}\label{optvaluefirstinstances}
\end{table}

\begin{figure}[!htbp]
\centering
\begin{tabular}{cc}
\includegraphics[scale=0.65]{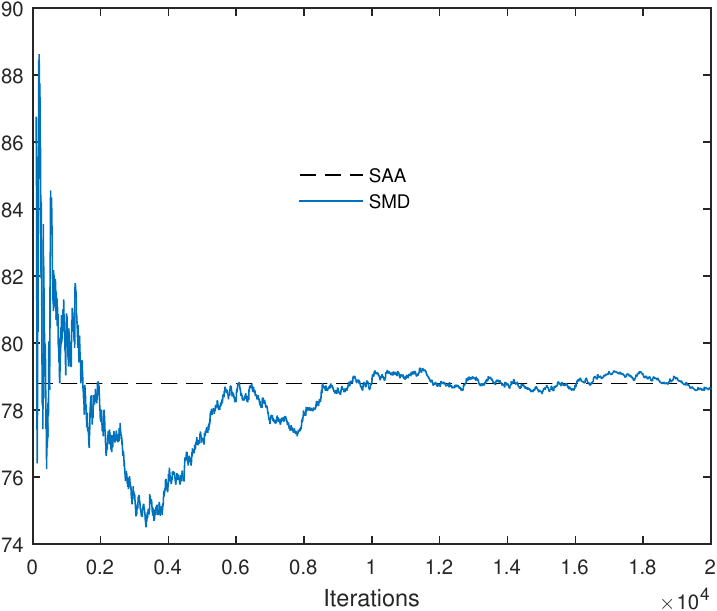}
&
\includegraphics[scale=0.65]{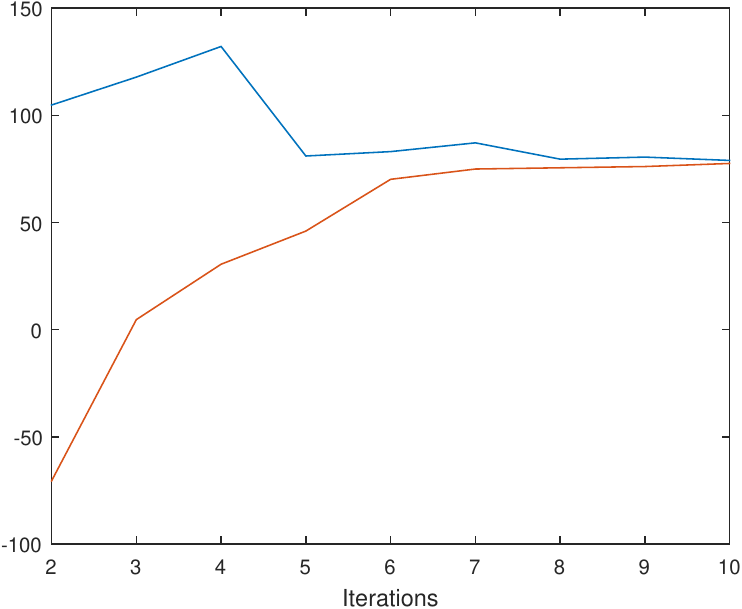}
\end{tabular}
\caption{Left plot: optimal value of our instance of  Problem \eqref{smdmodel11}-\eqref{smdmodel12} with 
$n=10$ estimated using SAA as well as evolution of the 
approximate optimal value computed along the iterations of SMD.
Right plot: for the same instance, evolution of the lower and upper bounds computed along the iterations
of the L-shaped method.}\label{figSMD1}
\end{figure}

\begin{figure}[!htbp]
\centering
\begin{tabular}{cc}
\includegraphics[scale=0.65]{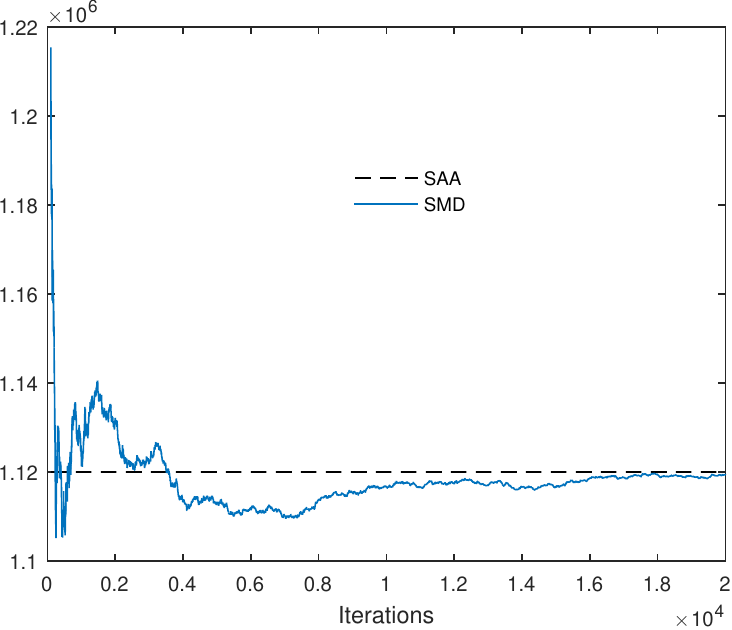}
&
\includegraphics[scale=0.65]{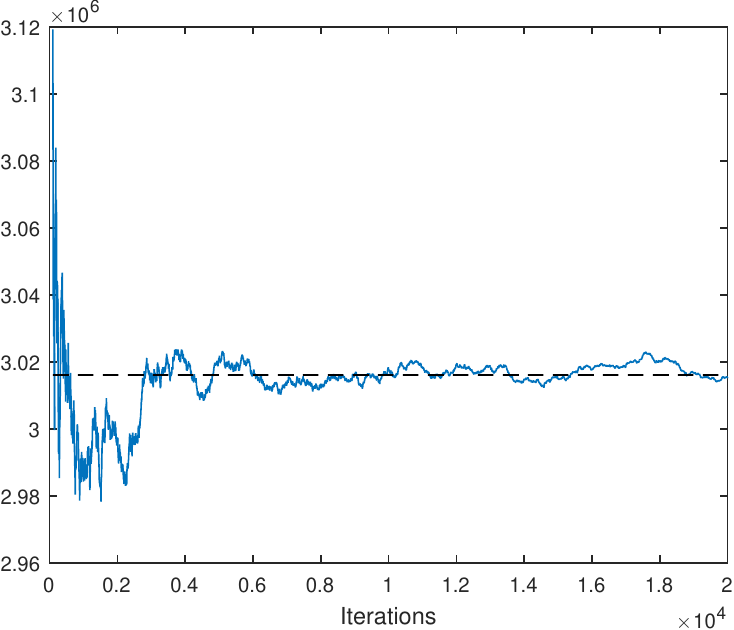}
\end{tabular}
\caption{Left plot: optimal value of our instance of  Problem \eqref{smdmodel21}-\eqref{smdmodel22} with 
$n=5$ estimated using SAA as well as evolution of the 
approximate optimal value computed along the iterations of SMD.
Right plot: same outputs for Problem \eqref{smdmodel21}-\eqref{smdmodel22} and 
$n=10$.}\label{figSMD2}
\end{figure}

We now consider larger instances taking $n=200$, $400,$ and $600$.
For these simulations we do not use SAA and L-shaped method anymore 
which were not 
as efficient as SMD on previous simulations and require prohibitive computational time
for $n=200, 400, 600$, and we compare the performance of SMD and
the four variants ISMD 1, ISMD 2, ISMD 3, and ISMD 4 of ISMD defined above. 

For $n=200$ and $n=400$, we run all methods 10 times taking samples of $\xi_2$
of size $N=2000$ for $n=200$, of size $N=1000$ for Problem \eqref{smdmodel11}-\eqref{smdmodel12}
and $n=400$, and  of size $N=500$ for  Problem \eqref{smdmodel21}-\eqref{smdmodel22}
and $n=400$. For $n=600$, it takes much more time to load and solve subproblems and we
only run SMD and ISMD once taking a sample  of size $N=500$ for Problem \eqref{smdmodel11}-\eqref{smdmodel12}
and of size $N=300$ for Problem \eqref{smdmodel21}-\eqref{smdmodel22}.\footnote{Due to the increase in computational time when $N$ increases, we do not take the largest sample size $N=2000$
for all instances. However, for all instances and values of $N$ chosen, we observe a stabilization of the approximate
optimal value before stopping the algorithm, which indicates a good solution has been found at termination.}

In Figure \ref{figSMD3}, we report for our instances of Problem 
\eqref{smdmodel11}-\eqref{smdmodel12} the mean (computed over the 10 
runs of the methods for $n=200, 400$)  approximate optimal values 
along the iterations of SMD and our variants 
of ISMD.\footnote{When SMD (and similarly for ISMD) is run on samples of $\xi_2$ of size $N$, we have seen
how to compute at iteration $t-1$ an estimation 
$\displaystyle \frac{1}{\sum_{\tau=1}^t \gamma_\tau(N)} \sum_{\tau=1}^t  \gamma_\tau( N) \Big(  f_1(x_1^{N,\tau}) + f_2( x_2^{N,\tau }, x_1^{N,\tau}, \xi_2^{N,\tau}) \Big)$
of the optimal value on the basis of sample $\xi_2^{N,1},\ldots,\xi_2^{N,t}$  of $\xi_2$.
The mean approximate optimal value after $t-1$ iterations is obtained running SMD on 10 independent
samples of $\xi_2$ of size $N$ and computing the mean of these values on these samples.} We also report on this figure the empirical distribution (over the 
10 runs of the methods for $n=200, 400$) of the total time required to
solve the problem instances with SMD and our variants of ISMD.

As expected, ISMD 1 and ISMD 2 complete the $N$ iterations quicker (since they run Mosek 
for less iterations) but start with worse approximations of the optimal values.
ISMD 3 and ISMD 4 also complete the $N$ iterations quicker than SMD but 
provide approximations of the optimal values very close to SMD along the iterations
of the method and in particular at termination, see also 
Table \ref{tabletimesmdbis} which gives the mean approximate optimal value at
the last iteration $N$ for all methods.
We should also note that most of the computational time for these methods
is spent in loading the data for Mosek solver through a series of loops
and this step requires the same computational time for all methods.
Therefore, the difference in computational time only comes from the 
time spent by Mosek to solve subproblems. With a C++ or Fortran implementation,
this time would remain similar but the loops for loading the data would be much
quicker and the total solution time would decrease by a much more important factor.
However, even with our Matlab implementation, the total time decreases significantly.

\begin{figure}[!htbp]
\centering
\begin{tabular}{cc}
\includegraphics[scale=0.58]{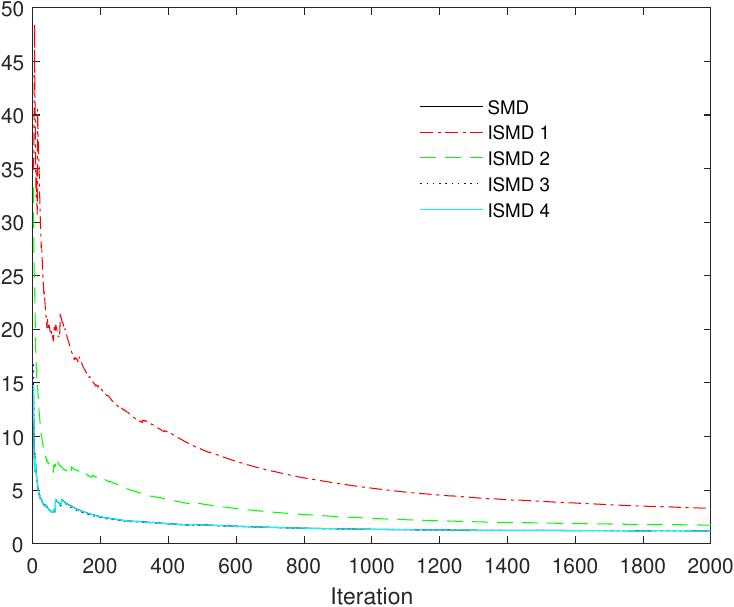}
&
\includegraphics[scale=0.58]{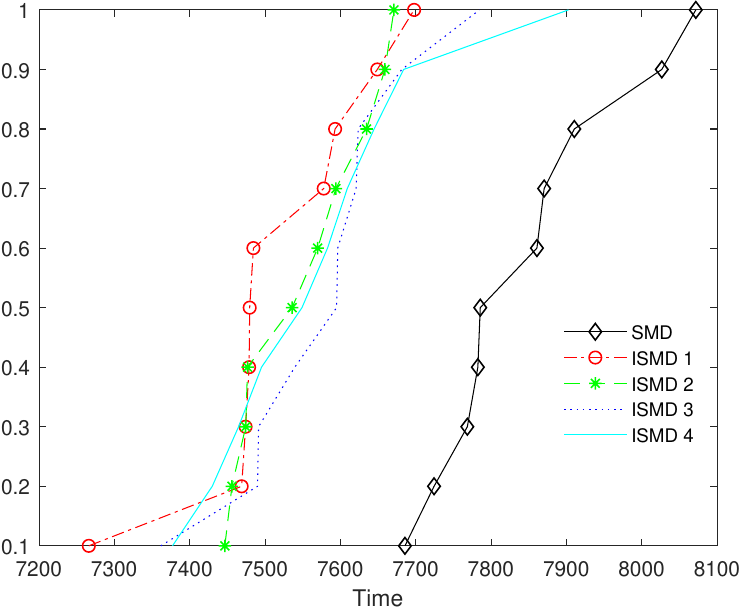}
\\
{$n=200$} &   {$n=200$}\\
\includegraphics[scale=0.58]{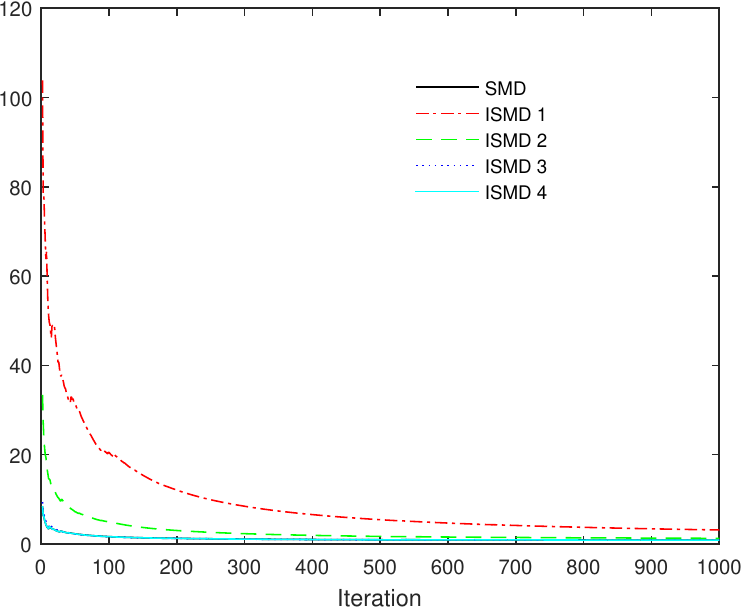}
&
\includegraphics[scale=0.58]{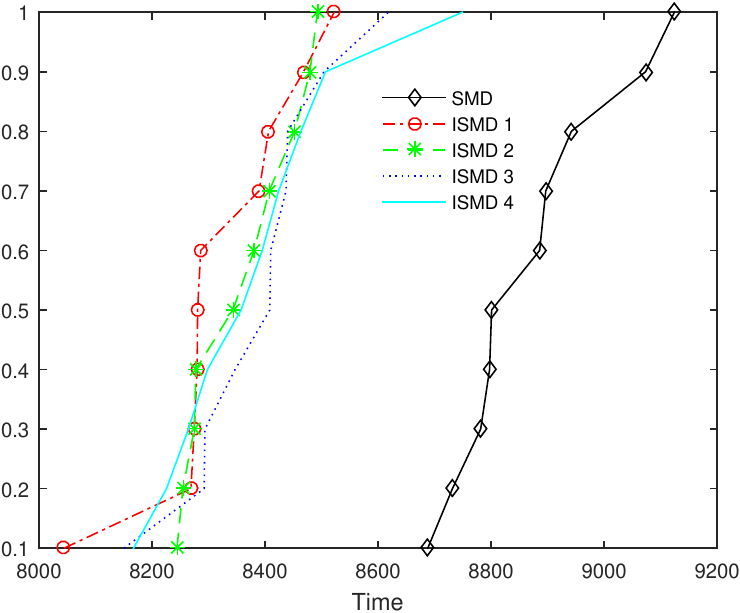}
\\
{$n=400$} &   {$n=400$}\\
\end{tabular}
\begin{tabular}{c}
\includegraphics[scale=0.58]{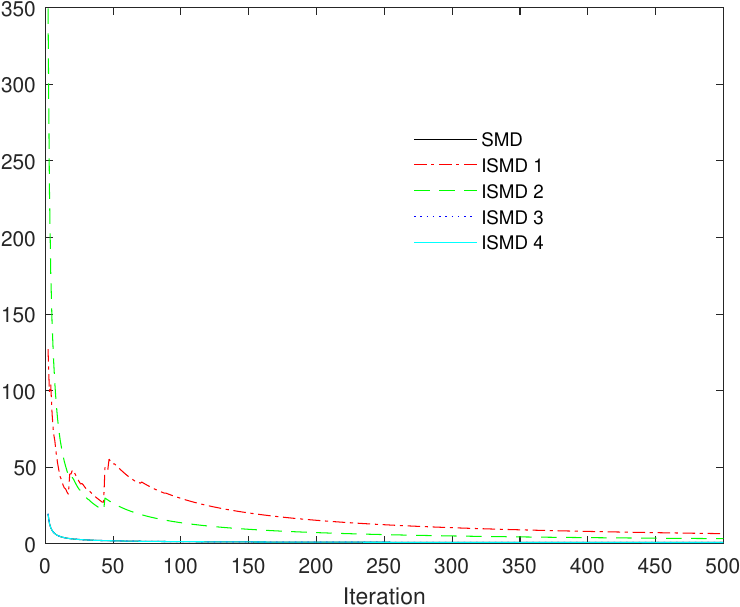}\\
$n=600$
\end{tabular}
\caption{Top left plot: approximate optimal values of our instance of 
Problem \eqref{smdmodel11} with  $n=200$ along the iterations of SMD and our variants
of ISMD. Top right plot: empirical distribution of the solution time in seconds
on this instance and for these methods. Middle plots: same as the top plot replacing 
$n=200$ by $n=400$. Bottom plot: same as the top left plot replacing $n=200$
by $n=600$.}\label{figSMD3}
\end{figure}

For our instances of Problem \eqref{smdmodel12}-\eqref{smdmodel22},
we report in Figure \ref{figSMD4} the mean (over the 10 runs for 
$n=200$ and $n=400$) approximate optimal values computed along the
iterations of SMD and our variants of ISMD. 
For the instances $n=200$ and $n=400$, we also report in Figure \ref{figSMD5}
the empirical distribution of the total solution time and of the
time required for Mosek to solve subproblems for SMD and all variants of
ISMD. The remarks made for Problem \eqref{smdmodel11} still apply for these
simulations performed on Problem \eqref{smdmodel21}.
We also refer to Table \ref{tabletimesmdbis} which provides
the mean approximate optimal value at
the last iteration $N$ for all methods. As for Problem \eqref{smdmodel11},
ISMD 3 and ISMD 4 provide after our $N-1$ iterations a good approximation of the
optimal value, very close to the approximation obtained with SMD but require
less computational time.

\begin{figure}[!htbp]
\centering
\begin{tabular}{cc}
\includegraphics[scale=0.65]{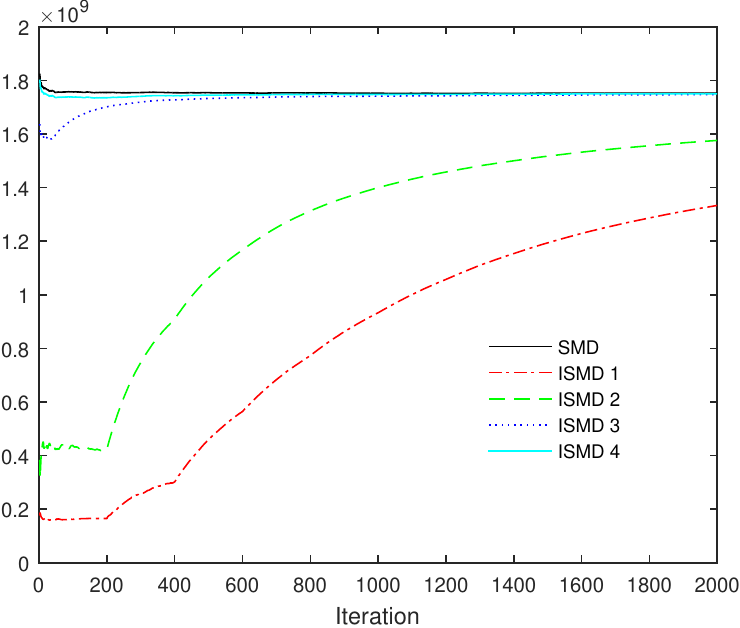}
&
\includegraphics[scale=0.65]{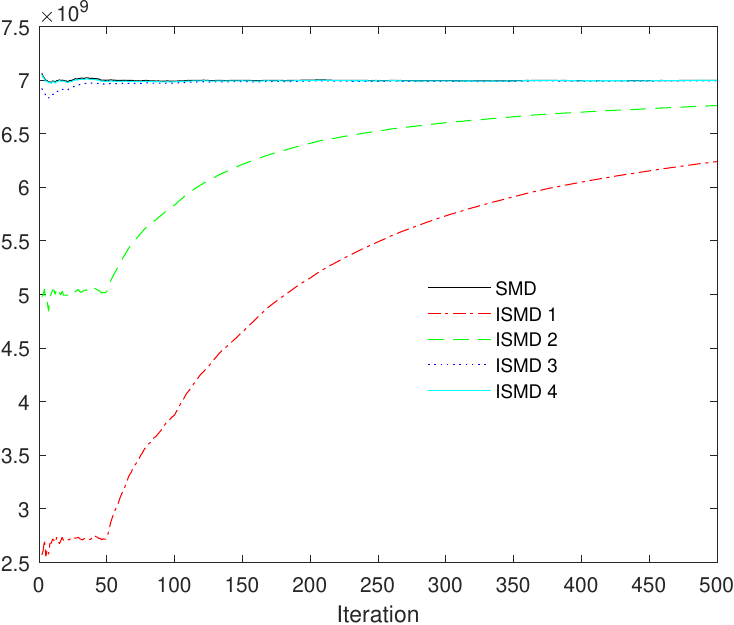}
\\
{$n=200$} &   {$n=400$}\
\end{tabular}
\begin{tabular}{c}
\includegraphics[scale=0.65]{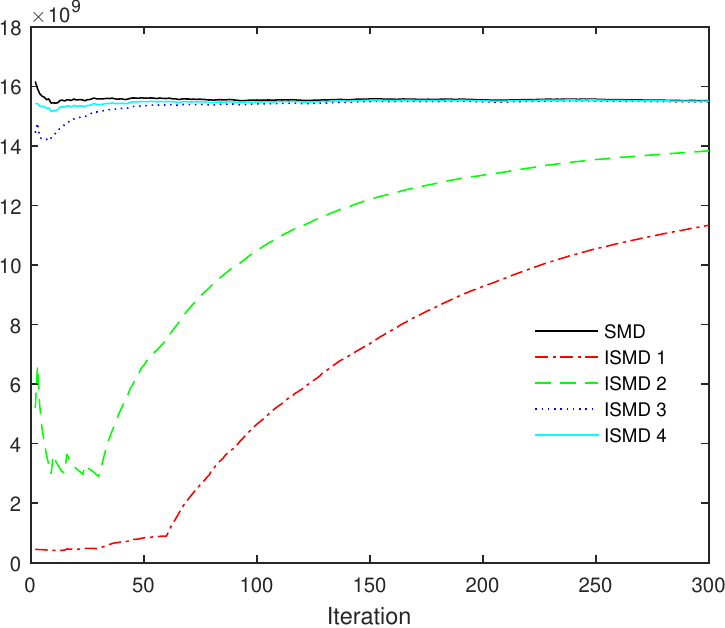}\\
{$n=600$}
\end{tabular}
\caption{Top left plot: approximate optimal values of our instance of 
Problem \eqref{smdmodel21} with  $n=200$ along the iterations of SMD and our variants
of ISMD. Top right and bottom plots provide the same graphs for respectively 
$n=400$ and $n=600$.}\label{figSMD4}
\end{figure}

\begin{figure}[!htbp]
\centering
\begin{tabular}{cc}
\includegraphics[scale=0.65]{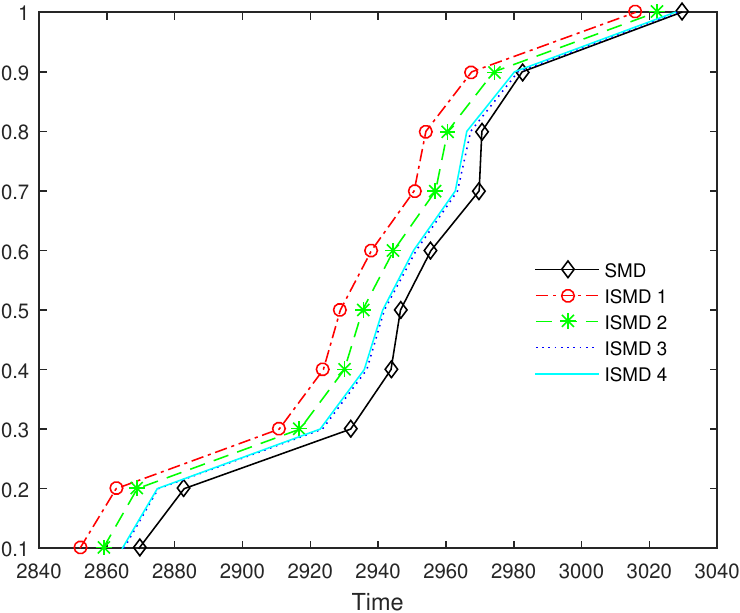}
&
\includegraphics[scale=0.65]{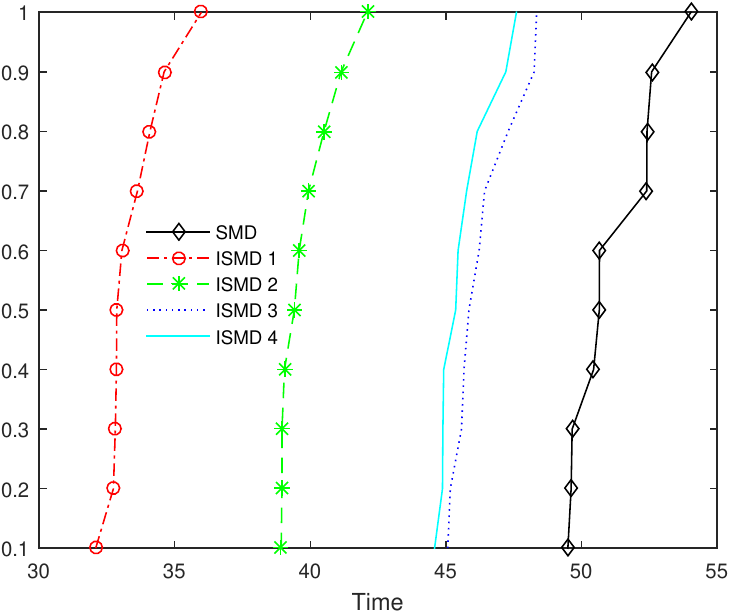}
\\
{$n=200$} &   {$n=200$}\\
\end{tabular}
\begin{tabular}{cc}
\includegraphics[scale=0.65]{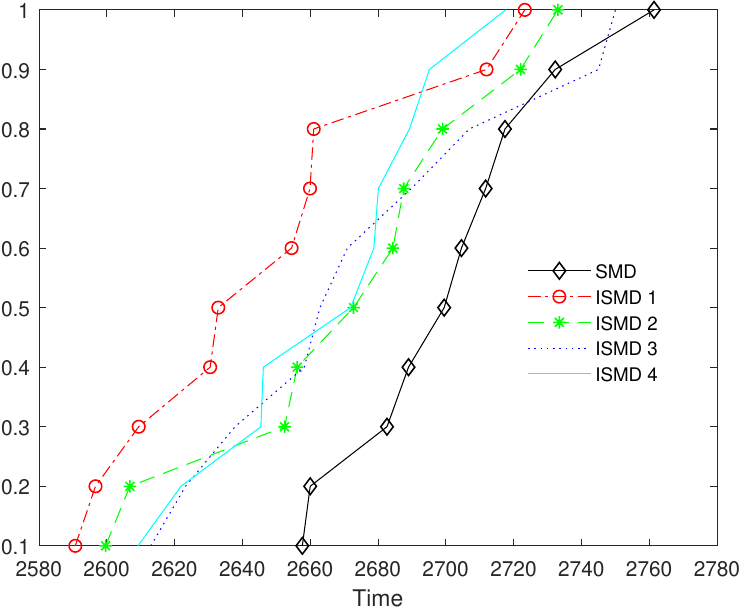}
&
\includegraphics[scale=0.65]{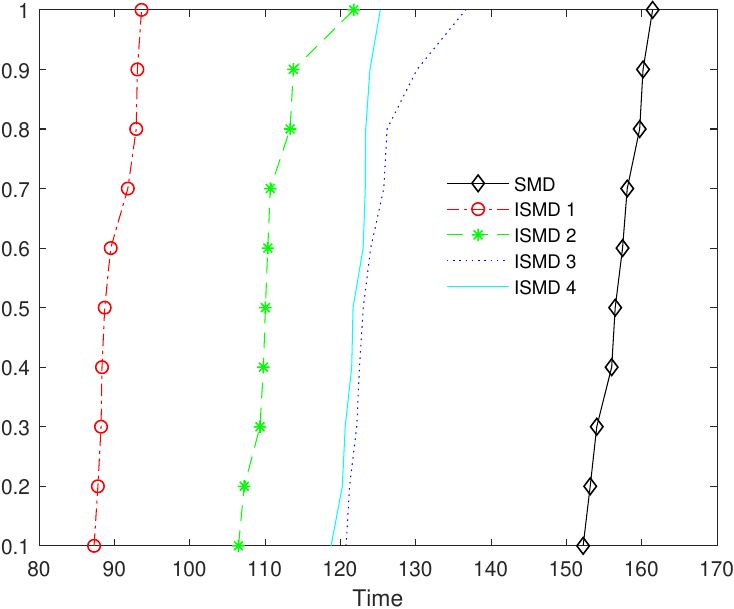}
\\
{$n=400$} &   {$n=400$}\\
\end{tabular}
\caption{
Top left: empirical distribution of the total solution time to solve Problem 
\eqref{smdmodel21} for SMD and and our four variants of ISMD for the instance with $n=200$. 
Right plot: empirical distribution of the time required for Mosek to solve 
all subproblems for that instance. Bottom plots: same outputs for $n=400$.}\label{figSMD5}
\end{figure}

\begin{table}
\centering
\begin{tabular}{|c|c|c|c|c|c|}
\hline
Instance  & SMD  &  ISMD 1 &  ISMD 2 &  ISMD 3 & ISMD 4    \\
\hline
$n=200$, Problem \eqref{smdmodel11}  &    1.2  & 3.2 & 1.7  &  1.2  &  1.2 \\
\hline
$n=400$, Problem \eqref{smdmodel11}  &    0.86  & 3.14  & 1.27  &  0.86  &  0.86 \\
\hline
$n=600$, Problem \eqref{smdmodel11}  &    0.81  & 6.59  & 3,33  &  0.81  &  0.81 \\
\hline
$n=200$, Problem \eqref{smdmodel21} & $1.7523\small{\times}10^9$    & 
$1.3335\small{\times}10^9$ & $1.5762\small{\times}10^9$ & $1.7472\small{\times}10^9$   & $1.7508\small{\times}10^9$  \\
\hline
$n=400$, Problem \eqref{smdmodel21} &   6.9978$\small{\times}10^9$  
& 6.2402$\small{\times}10^9$ &   6.7624$\small{\times}10^9$  & 
6.9943$\small{\times}10^9$ &   6.9972$\small{\times}10^9$  \\
\hline
$n=600$, Problem \eqref{smdmodel21} &    1.5524$\small{\times}10^{10}$  &1.1339$\small{\times}10^{10}$   & 
1.3838$\small{\times}10^{10}$& 1.5481$\small{\times}10^{10}$   &1.5512$\small{\times}10^{10}$   \\
\hline
\end{tabular}
\caption{Approximate optimal values of instances of Problems \eqref{smdmodel11} and \eqref{smdmodel21} 
estimated with SMD, ISMD 1, ISMD 2, ISMD 3, and ISMD 4.}
\label{tabletimesmdbis}
\end{table}

\section{Conclusion}

We introduced an inexact variant of SMD called ISMD to solve (general) nonlinear two-stage stochastic programs.
We have shown on two examples of two-stage nonlinear problems that ISMD can allow us to obtain  quicker than
SMD a good solution and a good approximation of the optimal value.

The method and convergence analysis was based on two studies of convex analysis: 
\begin{itemize}
\item[(a)] the computation of inexact cuts
for value functions of a large class of convex optimization problems having nonlinear objective and constraints
which couple the argument of the value function and the decision variable;
\item[(b)] the study of the strong concavity of the dual function of an optimization problem (used to derive one of our formulas for inexact cuts). 
\end{itemize}

It is worth mentioning that the formulas we derived for inexact cuts could also be used to propose inexact level methods \cite{lemnemnester}
to solve nonlinear two-stage stochastic programs \eqref{defpb4}-\eqref{pbsecondstage4}, when primal and dual second stage problems are solved
approximately (inexactly). 

It would also be interesting to test ISMD and the aforementioned inexact level methods on several relevant instances of nonlinear
two-stage stochastic programs. 



\section*{Acknowledgments} The author's research was 
partially supported by an FGV grant, CNPq grant 311289/2016-9,
and FAPERJ grant E-26/201.599/2014. The author would like to thank
Alberto Seeger for helpful discussions.

\addcontentsline{toc}{section}{References}
\bibliographystyle{plain}
\bibliography{Inexact_SMD_2_Stage}

\begin{thebibliography}{10}

\bibitem{mosek}
E.~D. Andersen and K.D. Andersen.
\newblock {\em The MOSEK optimization toolbox for MATLAB manual. Version 7.0},
  2013.
\newblock \url{https://www.mosek.com/}.

\bibitem{birge-louv-book}
J.~Birge and F.~Louveaux.
\newblock {\em {Introduction to Stochastic Programming}}.
\newblock Springer-Verlag, New York, 1997.

\bibitem{danglynn}
G.B. Dantzig and P.W. Glynn.
\newblock Parallel processors for planning under uncertainty.
\newblock {\em Annals of Operations Research}, 22:1--21, 1990.

\bibitem{guiguessiopt2016}
V.~Guigues.
\newblock Convergence analysis of sampling-based decomposition methods for
  risk-averse multistage stochastic convex programs.
\newblock {\em SIAM Journal on Optimization}, 26:2468--2494, 2016.

\bibitem{guiguesmstep17}
V.~Guigues.
\newblock Multistep stochastic mirror descent for risk-averse convex stochastic
  programs based on extended polyhedral risk measures.
\newblock {\em Mathematical Programming}, 163:169--212, 2017.

\bibitem{guigues2016isddp}
V.~Guigues.
\newblock {Inexact cuts in Stochastic Dual Dynamic Programming}.
\newblock {\em Siam Journal on Optimization}, 30:407--438, 2020.
\newblock \url{https://arxiv.org/abs/1809.02007}.

\bibitem{hhlem}
J-B Hiriart-Urruty and C.~Lemar\'echal.
\newblock {\em {Convex Analysis and Minimization Algorithms I}}.
\newblock Springer-Verlag, 1996.

\bibitem{infanger}
G.~Infanger.
\newblock {Monte Carlo (Importance) Sampling within a Benders Decomposition
  Algorithm for Stochastic Linear Programs}.
\newblock {\em Annals of Operations Research}, 39:69--95, 1992.

\bibitem{nestioud2010}
A.~Juditsky and Y.~Nesterov.
\newblock Primal-dual subgradient methods for minimizing uniformly convex
  functions.
\newblock {\em Available on arXiv at http://arxiv.org/abs/1401.1792}, 2010.

\bibitem{nemlansh09}
G.~Lan, A.~Nemirovski, and A.~Shapiro.
\newblock Validation analysis of mirror descent stochastic approximation
  method.
\newblock {\em Math. Program.}, 134:425--458, 2012.

\bibitem{lan2017}
G.~Lan and Z.~Zhou.
\newblock Dynamic stochastic approximation for multi-stage stochastic
  optimization.
\newblock {\em arXiv}, 2017.

\bibitem{lemnemnester}
C.~Lemar\'echal, A.~Nemirovski, and Y.~Nesterov.
\newblock New variants of bundle methods.
\newblock {\em Mathematical Programming}, 69:111--148, 1995.

\bibitem{nemjudlannem09}
A.~Nemirovski, A.~Juditsky, G.~Lan, and A.~Shapiro.
\newblock Robust stochastic approximation approach to stochastic programming.
\newblock {\em SIAM J. Optim.}, 19:1574--1609, 2009.

\bibitem{pereira}
M.V.F. Pereira and L.M.V.G Pinto.
\newblock Multi-stage stochastic optimization applied to energy planning.
\newblock {\em Math. Program.}, 52:359--375, 1991.

\bibitem{polyakjud92}
B.T. Polyak and A.~Juditsky.
\newblock Acceleration of stochastic approximation by averaging.
\newblock {\em SIAM J. Contr. and Optim.}, 30:838--855, 1992.

\bibitem{Rockafellar}
R.~T. Rockafellar and R.~J-B Wets.
\newblock {\em {Variational Analysis}}.
\newblock Grundlehren der Mathematischen Wissenschaften, Springer-Verlag, 1997.

\bibitem{ruz}
A.~Ruszczy\'nski.
\newblock A multicut regularized decomposition method for minimizing a sum of
  polyhedral functions.
\newblock {\em Mathematical Programming}, 35:309--333, 1986.

\bibitem{shadenrbook}
A.~Shapiro, D.~Dentcheva, and A.~Ruszczy\'nski.
\newblock {\em {Lectures on Stochastic Programming: Modeling and Theory}}.
\newblock SIAM, Philadelphia, 2009.

\bibitem{YuNeely2015}
H.~Yu and J.~Neely.
\newblock {On the Convergence Time of the Drift-Plus-Penalty Algorithm for
  Strongly Convex Programs}.
\newblock {\em Available at \url{http://arxiv.org/abs/1503.06235}}, 2015.

\end{thebibliography}

\section*{Appendix}
\if{
\begin{cor}\label{corinexactcuts} Consider the value functions $\mathcal{Q}: X \rightarrow \mathbb{R}$ where $\mathcal{Q}(x)$
is given by the optimal value of the following optimization problems:
\begin{equation}
\begin{array}{lll}
(a)\left\{ 
\begin{array}{l}
\min_{y } f(y,x)\\
Ay + B x=b,\\h(y)+k(x)\leq 0,\\
y \in Y,
\end{array}
\right.
&
(b)\left\{ 
\begin{array}{l}
\min_{y } f_0(y)+f_1(x)\\
Ay + B x=b,\\g(y,x)\leq 0,\\
y \in Y,
\end{array}
\right.
&
(c)\left\{ 
\begin{array}{l}
\min_{y } f_0(y)+f_1(x)\\
Ay + B x=b,\\h(y)+k(x)\leq 0,\\
y \in Y,
\end{array}
\right.\\
(d)\left\{ 
\begin{array}{l}
\min_{y }  f(y,x)\\
g(y,x)\leq 0,\\
y \in Y,
\end{array}
\right.
&
(e)\left\{ 
\begin{array}{l}
\min_{y } f(y,x)\\
h(y)+k(x)\leq 0,\\
y \in Y,
\end{array}
\right.
&
(f)\left\{ 
\begin{array}{l}
\min_{y } f_0(y)+f_1(x)\\
g(y,x)\leq 0,\\
y \in Y,
\end{array}
\right.\\
(g)\left\{ 
\begin{array}{l}
\min_{y } f_0(y)+f_1(x)\\
h(y)+k(x)\leq 0,\\
y \in Y,
\end{array}
\right.
& 
(h)\left\{ 
\begin{array}{l}
\min_{y } f(y,x)\\
A y +B x = b,\\
y \in Y,
\end{array}
\right.
& 
(i)\left\{ 
\begin{array}{l}
\min_{y } f_0(y)+f_1(x)\\
A y +B x = b,\\
y \in Y.
\end{array}
\right.
\end{array}
\end{equation}
For problems (b),(c),(f),(g), (i) above define $f(y,x)=f_0(y)+f_1(x)$ and for problems (a), (c), (e), (g) 
define $g(y,x)=h(y)+k(x)$.
With this notation, assume that (H1), (H2), (H3), (H4), (H5), (H6), and (H7) hold for these problems.
If $g$ is defined, let $L_{x}(y, \lambda, \mu)=f(y,x) + \lambda^T (Bx+Ay-b) + \mu^T g(y,x)$ be the Lagrangian 
and define 
$$
U = \max_{i=1,\ldots,p} \|\nabla_{x} g_{i}(\hat y(  \varepsilon  ), \bar x)\| \mbox{ and }\mathcal{U}_{\bar x}= 
\frac{f(y_{\bar x} , \bar x  )  - \mathcal{L}_{\bar x}  }{\min (-g_{i}(y_{\bar x} ,  \bar x), i=1,\ldots,p )}
$$
where  $\mathcal{L}_{\bar x}$ is any lower bound on $\mathcal{Q}( \bar x) $.
If $g$ is not defined, define $L_{x}(y, \lambda)=f(y,x) + \lambda^T (Bx+Ay-b)$.

Let $\bar x \in X$,
let $\hat y$ be an $\epsilon$-optimal feasible primal solution for problem \eqref{vfunctionq}
written for $x= \bar x$ 
and let $( {\hat \lambda}, \hat \mu )$ be an $\epsilon$-optimal feasible solution of the
corresponding dual problem, i.e., of problem
\eqref{dualpb} written for $x=\bar x$.

Then 
$
\mathcal{C}(x)=f(\hat y, \bar x)-\eta(\varepsilon,\bar x)+ \langle s(\bar x) , x - \bar x \rangle
$
is an inexact cut for $\mathcal{Q}$ at $\bar x$ where  the formulas
for $\eta(\varepsilon,\bar x)$ and $s(\bar x)$ in each of cases (a)-(i) above are the following:
\begin{equation}
\begin{array}{l}
(a)\left\{
\begin{array}{ll}
\eta(\varepsilon, \bar x)=\varepsilon + \Big( M_1(  \bar x ) \frac{1}{\sqrt{\alpha(\bar x)}}  +   \sqrt{2} \max(   \|B^T\|, \sqrt{p} U)  \frac{1}{\sqrt{\alpha_D( \bar x ) }} \Big) \emph{Diam}(X) \sqrt{2 \varepsilon},\\
s(\bar x)=\nabla_x f(\hat y, \bar x ) + B^T \hat \lambda + \sum_{i=1}^p \hat \mu_i \nabla_x k_i( \bar x ),
\end{array}
\right.\\
(b)\left\{
\begin{array}{ll}
\eta(\varepsilon, \bar x)=\varepsilon + \Big( M_2( \bar x) \mathcal{U}_{\bar x} \frac{1}{\sqrt{\alpha(\bar x)}} + \sqrt{2} \max(   \|B^T\|, \sqrt{p} U )  \frac{1}{\sqrt{\alpha_D( \bar x ) }} \Big) \emph{Diam}(X) \sqrt{2 \varepsilon},\\
s(\bar x)=\nabla_x f_1(\bar x ) + B^T \hat \lambda + \sum_{i=1}^p \hat \mu_i \nabla_x g_i(\hat y,  \bar x ),
\end{array}
\right.\\
(c)\left\{
\begin{array}{ll}
\eta(\varepsilon, \bar x)=\varepsilon + 2 \max(   \|B^T\|, \sqrt{p} U)\emph{Diam}(X) \sqrt{\frac{\varepsilon}{ \alpha_D( \bar x )}},\\
s(\bar x)=\nabla_x f_1(\bar x ) + B^T \hat \lambda + \sum_{i=1}^p \hat \mu_i \nabla_x k_i( \bar x ),
\end{array}
\right.\\
(d)\left\{
\begin{array}{ll}
\eta(\varepsilon, \bar x)=\varepsilon + \Big( ( M_1(  \bar x )  + M_2( \bar x) \mathcal{U}_{\bar x} )  \frac{1}{\sqrt{\alpha(\bar x)}}  +  U  \sqrt{\frac{p}{\alpha_D( \bar x )}} \Big) \emph{Diam}(X) \sqrt{2 \varepsilon},\\
s(\bar x)=\nabla_x f(\hat y, \bar x ) + \sum_{i=1}^p \hat \mu_i \nabla_x g_i(\hat y ,  \bar x ),
\end{array}
\right.\\
(e)\left\{
\begin{array}{ll}
\eta(\varepsilon, \bar x)=\varepsilon + \Big( \frac{M_1(  \bar x )}{\sqrt{\alpha(\bar x)}} + \sqrt{  \frac{p}{ \alpha_D(\bar x) }} U  \Big) \emph{Diam}(X) \sqrt{2 \varepsilon},\\
s(\bar x)=\nabla_x f(\hat y, \bar x ) +  \sum_{i=1}^p \hat \mu_i \nabla_x k_i( \bar x ),
\end{array}
\right.\\
(f)\left\{
\begin{array}{ll}
\eta(\varepsilon, \bar x)=\varepsilon + \Big( \frac{ M_2( \bar x)}{\sqrt{\alpha(\bar x)}} \mathcal{U}_{\bar x} +  U \sqrt{   \frac{p}{ \alpha_D(\bar x) }} \Big) \emph{Diam}(X) \sqrt{2 \varepsilon},\\
s(\bar x)=\nabla_x f_1(\bar x )  + \sum_{i=1}^p \hat \mu_i \nabla_x g_i(\hat y ,  \bar x ),
\end{array}
\right.\\
(g)\left\{
\begin{array}{ll}
\eta(\varepsilon, \bar x)=\varepsilon +  \emph{Diam}(X) \sqrt{\frac{2 \varepsilon p}{\alpha_D( \bar x )}}U,\\
s(\bar x)=\nabla_x f_1(\bar x ) + \sum_{i=1}^p \hat \mu_i \nabla_x k_i( \bar x ),
\end{array}
\right.\\
(h)\left\{
\begin{array}{ll}
\eta(\varepsilon, \bar x)=\varepsilon + \Big( \frac{M_1(  \bar x )}{\sqrt{\alpha(\bar x)}} +  \frac{\|B^T \|}{ \sqrt{\alpha_D(\bar x) }}  \Big) \emph{Diam}(X) \sqrt{2 \varepsilon},\\
s(\bar x)=\nabla_x f(\hat y, \bar x ) + B^T \hat \lambda,
\end{array}
\right.\\
(i)\left\{
\begin{array}{ll}
\eta(\varepsilon, \bar x)=\varepsilon +  \|B^T \| \sqrt{ \frac{2 \varepsilon}{\alpha_D( \bar x )}}   \emph{Diam}(X),\\
s(\bar x)=\nabla_x f_1(\bar x ) + B^T \hat \lambda.
\end{array}
\right.
\end{array}
\end{equation}
\end{cor}
\begin{proof} It suffices to follow the proof of Proposition \ref{defcutctk}, specialized to cases (a)-(i). For instance, let us check the formulas in case (g).
For (g), 
$
s(\bar x)=\nabla_x L_{\bar x}(\hat y,\hat \mu)=\nabla_x f_1(\bar x ) + \sum_{i=1}^p \hat \mu_i \nabla_x k_i( \bar x )
$
and
\begin{equation}\label{caseccutinexact}
\begin{array}{lll}
\|\nabla_x L_{\bar x}(\hat y,  \hat \mu) - \nabla_x L_{\bar x}(\bar y, \bar \mu)  \| 
& = &  \| \sum_{i=1}^p (\hat \mu_i - \bar \mu_i) \nabla_x k_i( \bar x )  \| \leq U \| \hat \mu - \bar \mu   \|_1 \\
& \leq & U \sqrt{p}\|\hat \mu - \bar \mu\| \leq U \sqrt{p} \sqrt{\frac{2 \varepsilon}{\alpha_D( \bar x )}}.
\end{array} 
\end{equation}
It then suffices to combine \eqref{qtcut} and \eqref{caseccutinexact}.\hfill
\end{proof}
}\fi

\begin{figure}[H]
\centering
\begin{tabular}{cc}
\includegraphics[scale=0.4]{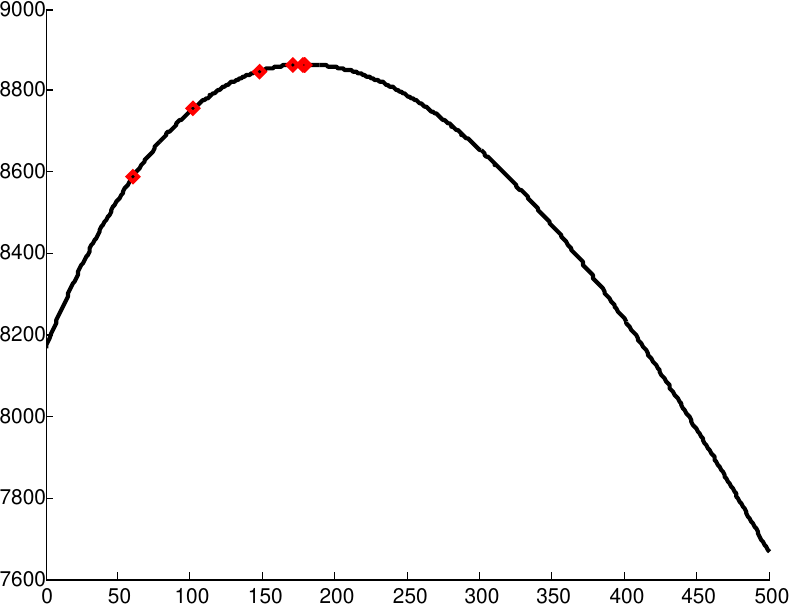}
&
\includegraphics[scale=0.4]{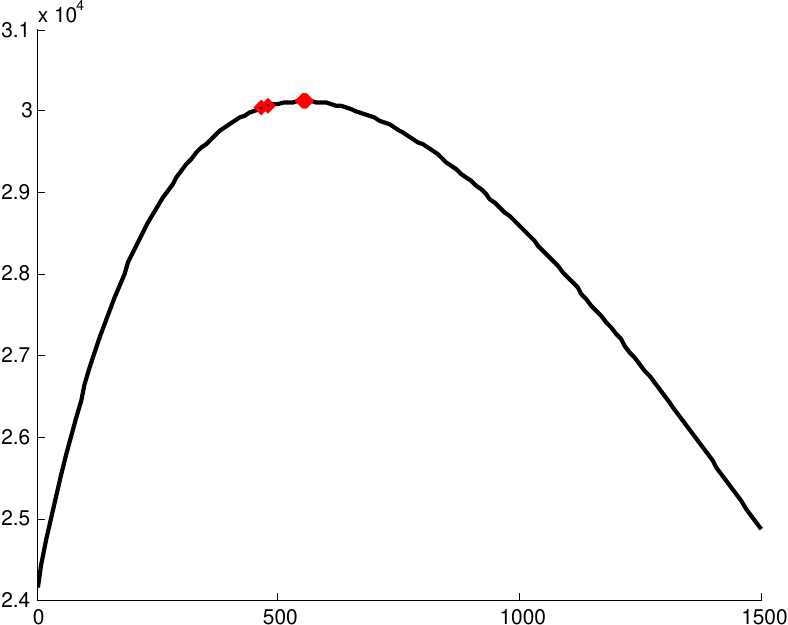}
\\
{$\lambda=1$, $n=1$, $\alpha(\bar x)=574$} &   {$\lambda=100$, $n=1$, $\alpha(\bar x)=637$}\\
\includegraphics[scale=0.4]{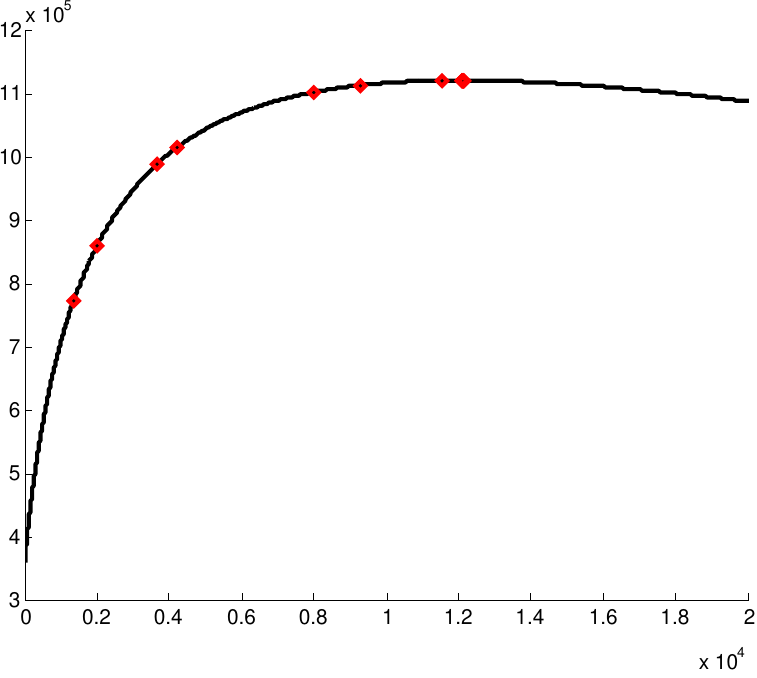}
&
\includegraphics[scale=0.4]{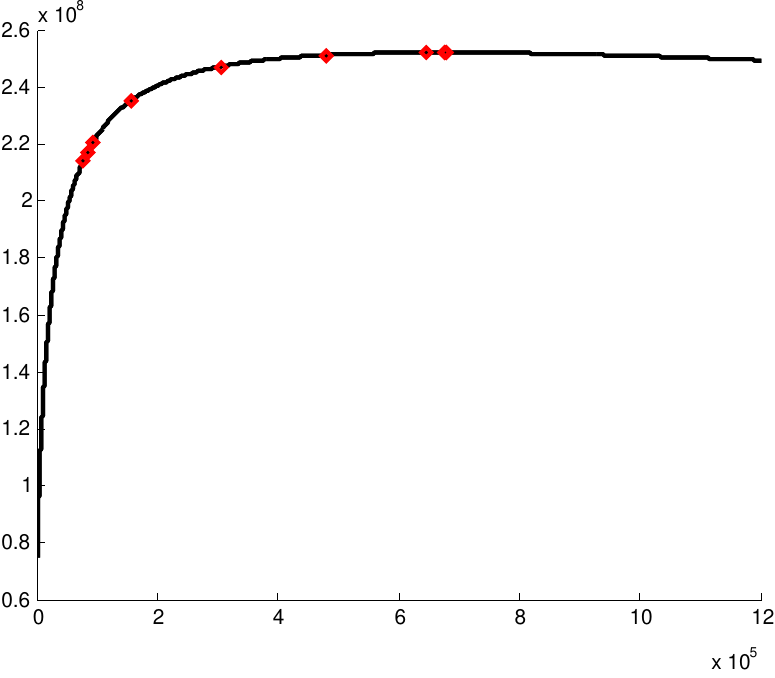}
\\
{$\lambda=1$, $n=10$, $\alpha(\bar x)=610$} &   {$\lambda=1$, $n=100$, $\alpha(\bar x)=2\,303$}\\
\includegraphics[scale=0.4]{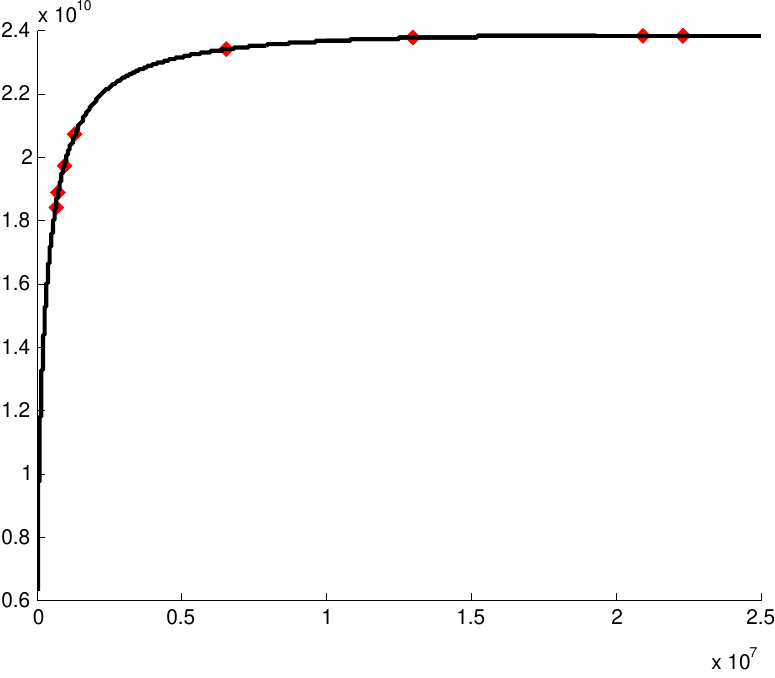}
&
\includegraphics[scale=0.4]{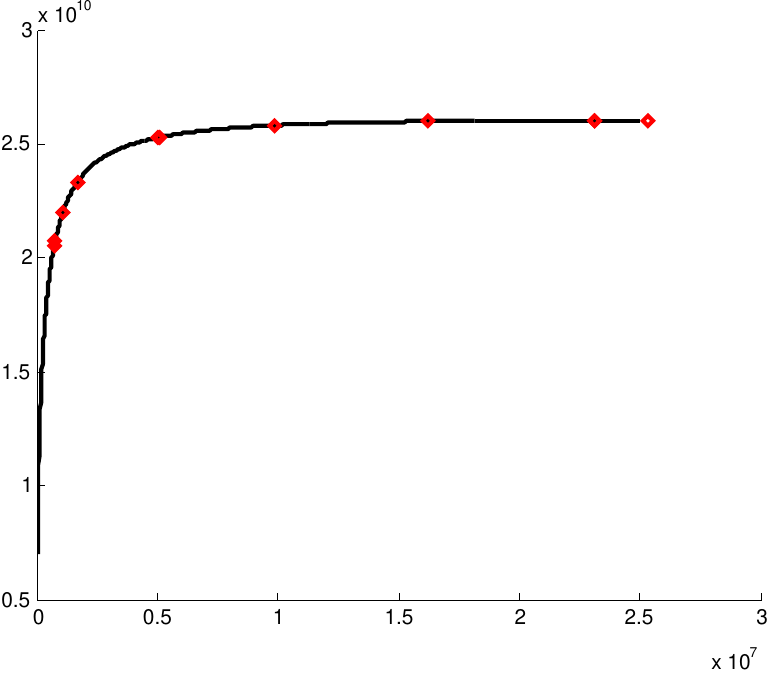}
\\
{$\lambda=1$, $n=1000$, $\alpha(\bar x)=23\,149$} &   {$\lambda=100$, $n=1000$, $\alpha(\bar x)=23\,988$}\\
\includegraphics[scale=0.4]{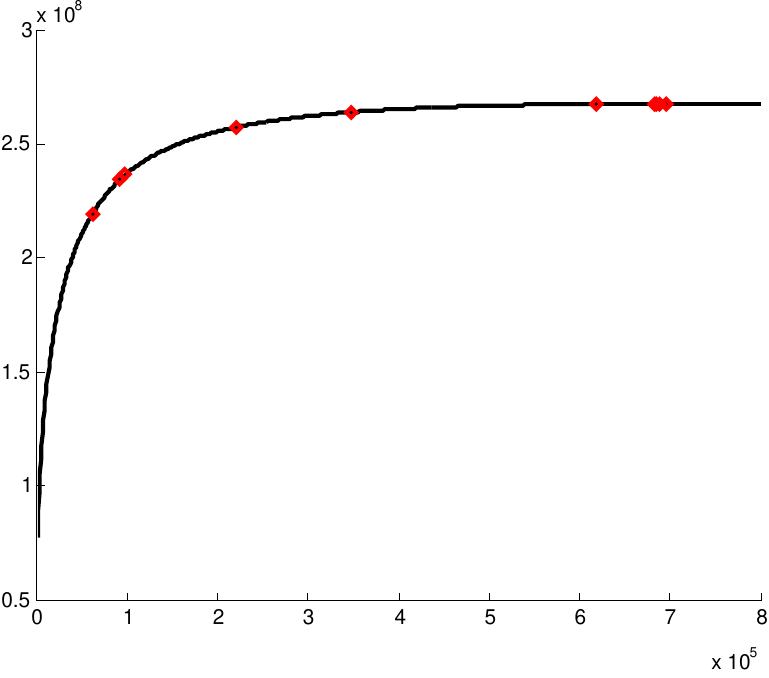}
&
\includegraphics[scale=0.4]{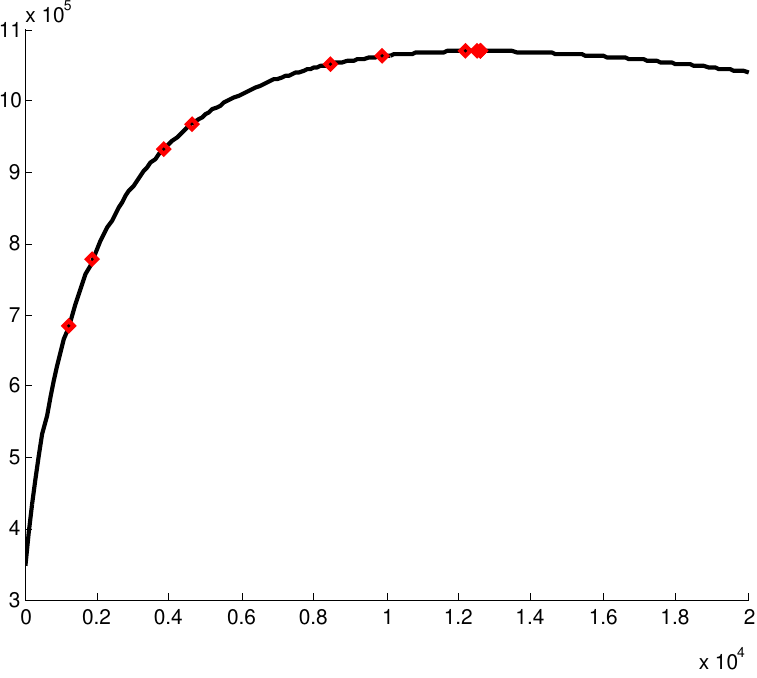}
\\
{$\lambda=100$, $n=100$, $\alpha(\bar x)=2\,839$} &   {$\lambda=100$, $n=10$, $\alpha(\bar x)=526$}
\end{tabular}
\caption{ \label{fig:f_1} Dual function $\theta_{\bar x}$ of problem \eqref{defqxsimple2} for some  $\bar x$ randomly drawn in ball $\{x \in \mathbb{R}^n: \|x-x_0 \|_2 \leq 1\}$, $S=AA^T + \lambda I_{2 n}$ for some random 
matrix $A$ with random entries in $[-20,20]$, and several values of the pair $(n,\lambda)$. The dual iterates are represented by red diamonds.}
\end{figure}

\begin{figure}[H]
\centering
\begin{tabular}{cc}
\includegraphics[scale=0.42]{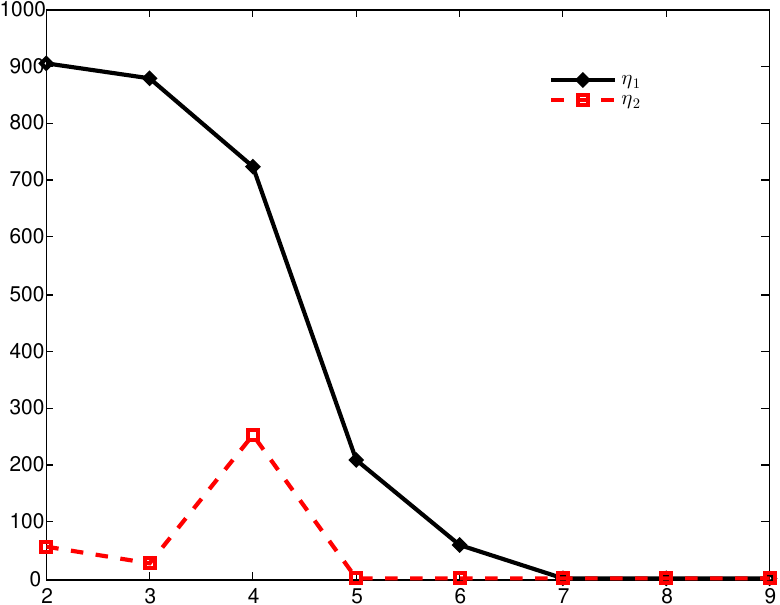}
&
\includegraphics[scale=0.42]{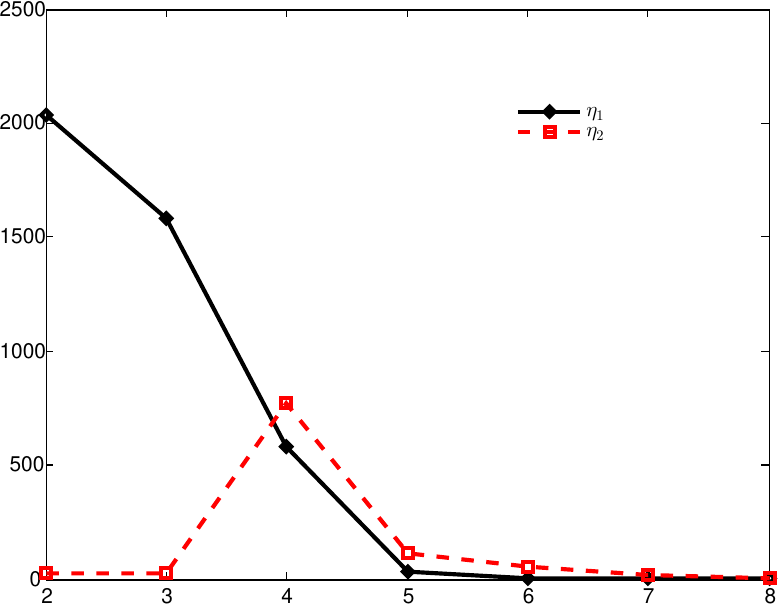}
\\
{$\lambda=1$, $n=1$, $\alpha(\bar x)=574$} &   {$\lambda=100$, $n=1$, $\alpha(\bar x)=637$}\\
\includegraphics[scale=0.42]{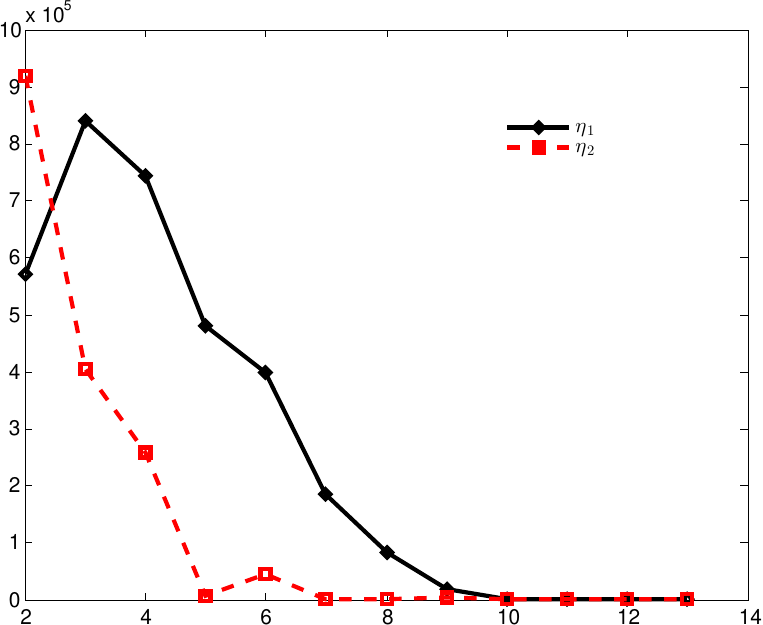}
&
\includegraphics[scale=0.42]{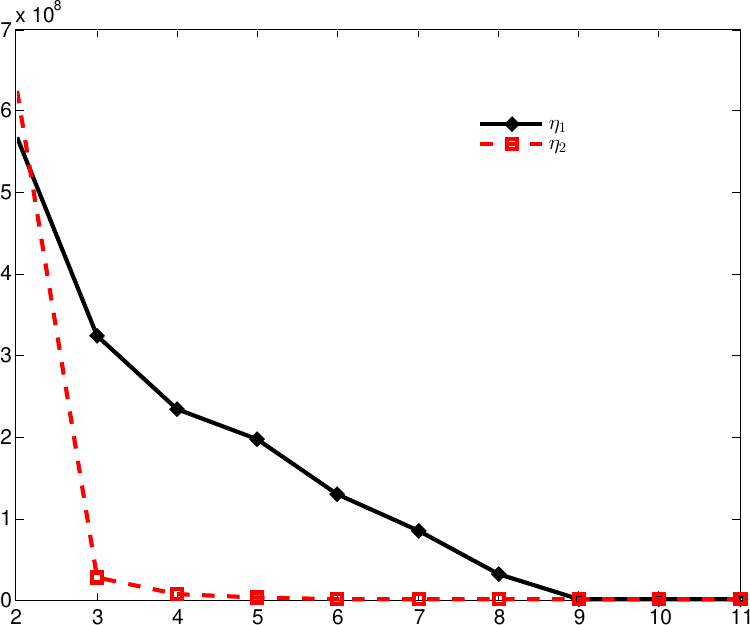}
\\
{$\lambda=1$, $n=10$, $\alpha(\bar x)=610$} &   {$\lambda=1$, $n=100$, $\alpha(\bar x)=2\,303$}\\
\includegraphics[scale=0.42]{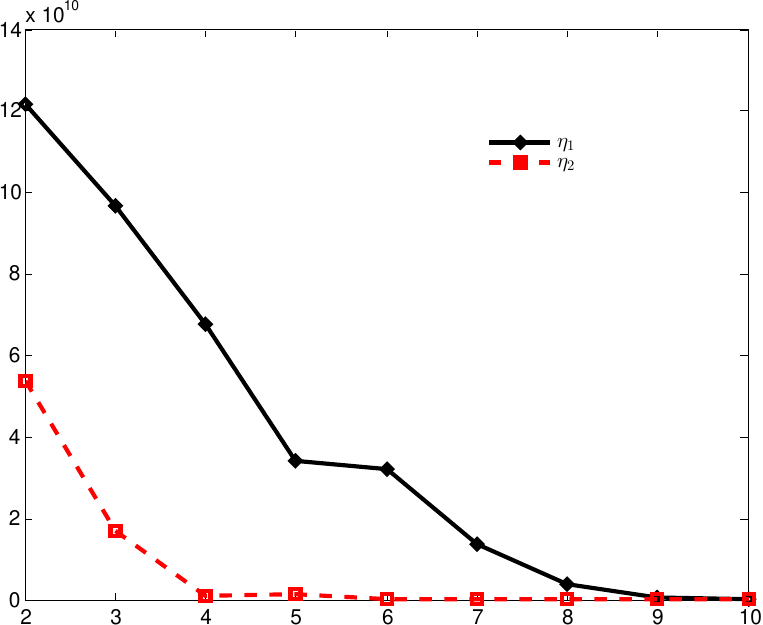}
&
\includegraphics[scale=0.42]{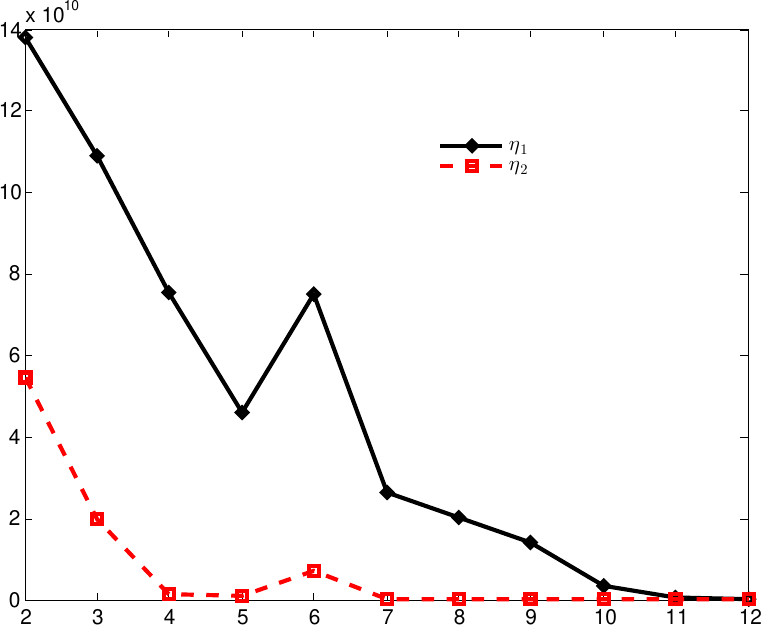}
\\
{$\lambda=1$, $n=1000$, $\alpha(\bar x)=23\,149$} &   {$\lambda=100$, $n=1000$, $\alpha(\bar x)=23\,988$}\\
\includegraphics[scale=0.42]{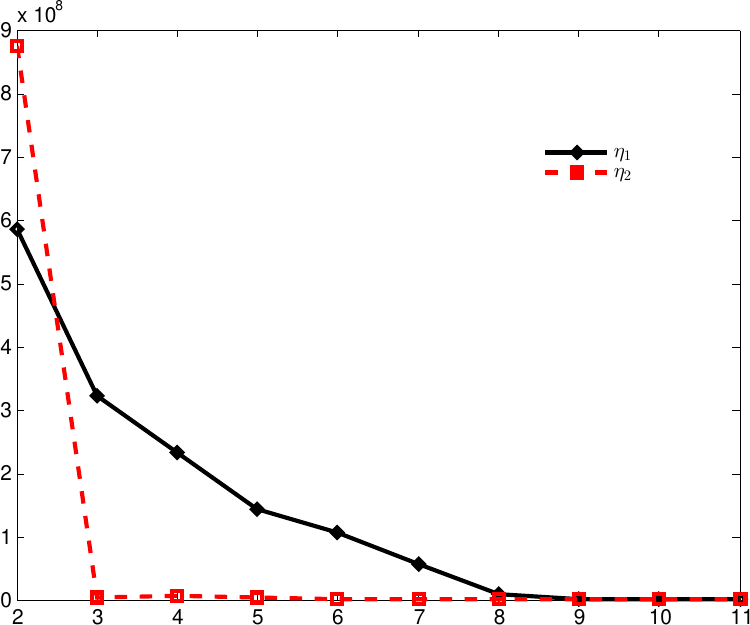}
&
\includegraphics[scale=0.42]{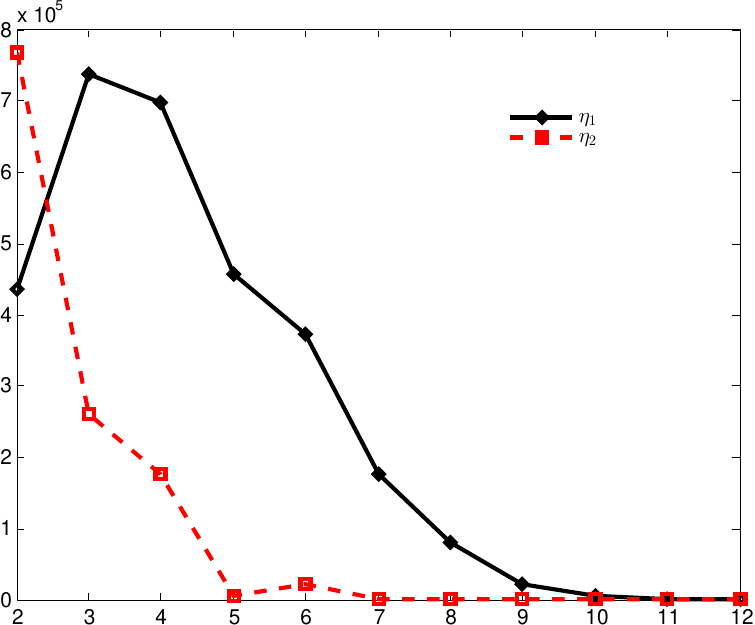}
\\
{$\lambda=100$, $n=100$, $\alpha(\bar x)=2\,839$} &   {$\lambda=100$, $n=10$, $\alpha(\bar x)=526$}
\end{tabular}
\caption{ \label{fig:f_2} Plots of $\eta_{1}(\varepsilon_k, \bar x)$ and $\eta_2(\varepsilon_k, \bar x)$ as a function of iteration $k$ where $\varepsilon_k$
is the duality gap at iteration $k$ for
problem \eqref{defqxsimple2} for some  $\bar x$ randomly drawn in ball $\{x \in \mathbb{R}^n: \|x-x_0 \|_2 \leq 1\}$, $S=AA^T + \lambda I_{2 n}$ for some random 
matrix $A$ with random entries in $[-20,20]$, and several values of the pair $(n,\lambda)$.}
\end{figure}

\begin{table}[H]
\centering
\begin{tabular}{c}
{\tt{ISMD 1}}\\
{\small{
\begin{tabular}{|c|c|c|c|c|}
\hline
Iteration number &  $[1,\left \lceil 0.1N \right \rceil ]$ & $[\left \lceil 0.1N \right \rceil +1,\left \lceil 0.2N \right \rceil]$ &  $[\left \lceil 0.2N \right \rceil +1,\left \lceil 0.3N \right \rceil]$    & $[\left \lceil 0.3N \right \rceil+1,\left \lceil 0.4N \right \rceil]$\\
\hline
\begin{tabular}{l}
IP solver maximal\\ 
number of iterations
\end{tabular}
& $\left \lceil 0.1 I_{\max} \right \rceil$ & $\left \lceil 0.2 I_{\max} \right \rceil$ &
$\left \lceil 0.3 I_{\max} \right \rceil$ & $\left \lceil 0.4 I_{\max} \right \rceil$ \\
\hline
\end{tabular}
}}\\
{\small{
\begin{tabular}{|c|c|c|c|}
\hline
Iteration number &  $[\left \lceil 0.4N \right \rceil +1,\left \lceil 0.5N \right \rceil]$ & $[\left \lceil 0.5N \right \rceil +1,\left \lceil 0.6N \right \rceil]$ &  $[\left \lceil 0.6N \right \rceil +1,\left \lceil 0.7N \right \rceil]$   \\
\hline
\begin{tabular}{l}
IP solver maximal\\ 
number of iterations
\end{tabular}
& $\left \lceil 0.5 I_{\max} \right \rceil$ & $\left \lceil 0.6 I_{\max} \right \rceil$ &
$\left \lceil 0.7 I_{\max} \right \rceil$  \\
\hline
\end{tabular}
}}\\
{\small{
\begin{tabular}{|c|c|c|c|}
\hline
Iteration number &  $[\left \lceil 0.7N \right \rceil +1,\left \lceil 0.8N \right \rceil]$ & $[\left \lceil 0.8N \right \rceil +1,\left \lceil 0.9N \right \rceil]$ &  $[\left \lceil 0.9N \right \rceil +1,N]$   \\
\hline
\begin{tabular}{l}
IP solver maximal\\ 
number of iterations
\end{tabular}
& $\left \lceil 0.8 I_{\max} \right \rceil$ & $\left \lceil 0.9 I_{\max} \right \rceil$ &
$I_{\max} $  \\
\hline
\end{tabular}
}}\\
{\tt{ISMD 2}}\\
{\small{
\begin{tabular}{|c|c|c|c|}
\hline
Iteration number &  $[1,\left \lceil 0.1N \right \rceil ]$ & $[\left \lceil 0.1N \right \rceil  +1,
\left \lceil 0.2N \right \rceil ]$ &  $[\left \lceil 0.2N \right \rceil +1,\left \lceil 0.3N \right \rceil]$    \\
\hline
\begin{tabular}{l}
IP solver maximal\\ 
number of iterations
\end{tabular}
& $\left \lceil 0.2I_{\max} \right \rceil$ & $\left \lceil 0.4 I_{\max} \right \rceil$ & $ \left \lceil 0.6 I_{\max} \right \rceil$   \\
\hline
\end{tabular}
}}\\
{\small{
\begin{tabular}{|c|c|c|c|}
\hline
Iteration number &  $[\left \lceil 0.3N \right \rceil +1,\left \lceil 0.4N \right \rceil ]$ & $[\left \lceil 0.4N 
\right \rceil +1,\left \lceil 0.5N \right \rceil ]$ &  $[\left \lceil 0.5N \right \rceil +1,N]$  \\
\hline
\begin{tabular}{l}
IP solver maximal\\ 
number of iterations
\end{tabular}
& $\left \lceil 0.8I_{\max} \right \rceil$ & $\left \lceil 0.9I_{\max} \right \rceil$ & $I_{\max}$ \\
\hline
\end{tabular}
}}\\
{\tt{ISMD 3}}\\
{\small{
\begin{tabular}{|c|c|c|c|}
\hline
Iteration number &  $[1,\left \lceil 0.02N \right \rceil ]$ & $[\left \lceil 0.02N \right \rceil +1,\left \lceil 0.04N\right \rceil]$ &  $[\left \lceil 0.04N \right \rceil +1,\left \lceil 0.06N \right \rceil]$    \\
\hline
\begin{tabular}{l}
IP solver maximal\\ 
number of iterations
\end{tabular}
& $\left \lceil 0.5I_{\max} \right \rceil$ & $\left \lceil 0.6I_{\max} \right \rceil$ & $\left \lceil0.7I_{\max} \right \rceil$   \\
\hline
\end{tabular}
}}\\
{\small{
\begin{tabular}{|c|c|c|c|}
\hline
Iteration number &  $[\left \lceil 0.06N \right \rceil +1,\left \lceil 0.08N \right \rceil ]$ & $[\left \lceil 0.08N\right \rceil+1,\left \lceil 0.1N \right \rceil]$ &  $[\left \lceil 0.1N \right \rceil+1,N]$  \\
\hline
\begin{tabular}{l}
IP solver maximal\\ 
number of iterations
\end{tabular}
& $\left \lceil 0.8I_{\max} \right \rceil$ & $\left \lceil 0.9I_{\max} \right \rceil$ & $I_{\max}$ \\
\hline
\end{tabular}
}}\\
{\tt{ISMD 4}}\\
{\small{
\begin{tabular}{|c|c|c|c|c|}
\hline
Iteration number &  $[1,\left \lceil 0.1N \right \rceil ]$ & $[\left \lceil 0.1N \right \rceil +1,\left \lceil 0.2N \right \rceil ]$ &  $[\left \lceil 0.2N \right \rceil +1,\left \lceil 0.3N \right \rceil ]$ &  $[\left \lceil 0.3N \right \rceil +1,N]$   \\
\hline
\begin{tabular}{l}
IP solver maximal\\ 
number of iterations
\end{tabular}
& $\left \lceil 0.7I_{\max} \right \rceil$ & $\left \lceil0.8I_{\max}\right \rceil$ & $\left \lceil 0.9I_{\max} \right \rceil$ & $I_{\max}$  \\
\hline
\end{tabular}
}}
\end{tabular}
\caption{Maximal number of iterations for Mosek interior point solver used to solve second stage problems as a function of the iteration number 
$i=1,\ldots,N$, of ISMD and the maximal number of iterations $I_{\max}$ allowed for Mosek 
solver to solve subproblems with SMD. In this table, $\left \lceil x \right \rceil$ is the smallest integer 
larger than or equal to $x$.
For problem \eqref{smdmodel11}-\eqref{smdmodel12} and $n=200, 400, 600$ 
and problem \eqref{smdmodel21}-\eqref{smdmodel22} and $n=200$,
we take $I_{\max}=15$,
for problem \eqref{smdmodel21}-\eqref{smdmodel22} and $n=400$ we take $I_{\max}=25$,
and for problem \eqref{smdmodel21}-\eqref{smdmodel22} and $n=600$ we take $I_{\max}=28$.
For instance for ISMD 1, $N=2000$, and problem \eqref{smdmodel11}-\eqref{smdmodel12}, for iterations 
$[\left \lceil 0.4N \right \rceil +1, \left \lceil 0.5N \right \rceil ]$, i.e., for iterations 
$0.4\times 2000 +1,\ldots,0.5\times 2000=801,\ldots,1000$, Mosek interior point solver is run to solve second stage problems
limiting the maximal number of iterations to $\left \lceil 0.5I_{\max} \right \rceil =8$.}\label{tablenumberiter0}
\end{table}

\end{document}